\documentclass[a4paper, 11pt,reqno]{amsart}
\usepackage{amsfonts, amsthm, amssymb, amsmath, stackengine, scalerel}
\usepackage{mathrsfs,array,tikz-cd}
\usepackage{hyperref}
\usepackage{eucal,fullpage,times,color,enumerate,accents, comment}
\usepackage[all]{xy}
\usepackage{xr}
\usepackage{url}
\usepackage{enumitem}
\usepackage[new]{old-arrows}
\usepackage{extpfeil}
\usepackage{turnstile} 
\usepackage{verbatim}  
\usepackage{graphics,float} 
\usetikzlibrary{graphs,decorations.pathmorphing,decorations.markings}
\input xy
\xyoption{all}
\usepackage{adjustbox}

\usepackage{algorithm,caption}  
\DeclareCaptionFont{lightblue}{\color{teal}}
\newcommand{\MING}[1]{{\color{blue} {\tiny \bf (M:)} {\bf #1}}}
\newcommand{\EXP}[1]{{\color{teal} {\tiny \bf (Exposition:)} {\bf #1}}}
\newcommand{\KIH}[1]{{\color{red} {\tiny \bf (K:)} {\bf #1}}}
\newcommand{\Ktxt}[1]{{\color{magenta} {#1}}}

\newextarrow{\xbigtoto}{{20}{20}{20}{20}}
{\bigRelbar\bigRelbar{\bigtwoarrowsleft\rightarrow\rightarrow}}    

\newcommand{\calX}{{\mathcal{X}}}

\newcommand{\calL}{{\mathcal{L}}}
\newcommand{\calU}{{\mathcal{U}}}
\newcommand{\calV}{{\mathcal{V}}}
\newcommand{\calA}{{\mathcal{A}}}
\newcommand{\calE}{\mathcal{E}}
\newcommand{\calB}{{\mathcal{B}}}

\newcommand{\calD}{{\mathcal{D}}}

\newcommand{\calR}{{\mathcal{R}}}
\newcommand{\calF}{\mathcal{F}}
\newcommand{\calG}{\mathcal{G}}
\newcommand{\calI}{\mathcal{I}}

\newcommand{\calH}{\mathcal{H}}
\newcommand{\calC}{\mathcal{C}}
\newcommand{\calS}{\mathcal{S}}
\newcommand{\calP}{\mathcal{P}}

\newcommand{\W}{\mathrm{W}}

\newcommand{\id}{\mathrm{id}}

\newcommand{\Tmod}{\mathbb{T}\text{-}\mathrm{mod}}

\newcommand{\frap}{\mathfrak{p}}

\newcommand{\m}{\mathfrak{m}}
\newcommand{\baseS}{\mathcal{S}}

\newcommand{\thT}{\mathbb{T}}

\newcommand{\N}{\mathbb{N}}

\newcommand{\E}{\mathcal{E}}

\newcommand{\Q}{\mathbb{Q}}

\DeclareFontFamily{U}{wncy}{}
\DeclareFontShape{U}{wncy}{m}{n}{<->wncyr10}{}
\DeclareSymbolFont{mcy}{U}{wncy}{m}{n}
\DeclareMathSymbol{\Sh}{\mathord}{mcy}{"58}

\newcommand{\U}{\mathcal{U}}
\newcommand{\V}{\mathcal{V}}
\DeclareSymbolFont{matha}{OML}{txmi}{m}{it}
\DeclareMathSymbol{\varv}{\mathord}{matha}{118}
\newcommand{\om}{\omega}

\newcommand{\dom}{\mathrm{dom}}
\newcommand{\Eff}{\mathrm{Eff}}

\newcommand{\rstr}{\!\!\upharpoonright\!\!_}

\newcommand{\lranglet}[2]{\langle#1,#2\rangle}
\newcommand{\lrangles}[1]{\langle#1\rangle}

\def\forkindep{\mathrel{\raise0.2ex\hbox{\ooalign{\hidewidth$\vert$\hidewidth\cr\raise-0.9ex\hbox{$\smile$}}}}}

\newcommand{\frat}{\mathfrak{t}}
\newcommand{\Fin}{\mathrm{Fin}}

\newcommand{\code}[1]{\lceil{#1}\rceil}

\newcommand{\Pw}{\calP(\w)}

\newcommand{\dn}{{\lnot\lnot}}

\newcommand{\calW}{\mathcal{W}}

\newcommand{\ww}{{^{\omega}{\omega}} }
\newcommand{\lww}{{^{<\omega}{\omega}} }

\newcommand{\lrk}{\leq_{\mathrm{RK}}}
\newcommand{\erk}{\equiv_{\mathrm{RK}}}

\newcommand{\lk}{\leq_{\mathrm{K}}}

\newcommand{\glt}{\leq_{\mathrm{LT}}^{\mathsf{o}}}
\newcommand{\elt}{\equiv_{\mathrm{LT}}^{\mathsf{o}}}
\newcommand{\lT}{\leq_{\mathrm{Tuk}}}
\newcommand{\eT}{\equiv_{\mathrm{Tuk}}}

\newcommand{\w}{\omega}
\newcommand{\restr}[1]{\!\!\upharpoonright_{#1}}
\newcommand{\Poly}{\mathcal{P}\mathrm{oly}}

\newcommand{\clt}{\leq_{\mathrm{LT}}}
\newcommand{\eclt}{\equiv_{\mathrm{LT}}}
\newcommand{\sclt}{<_{\mathrm{LT}}}
\newcommand{\sglt}{<^{\circ}_{\mathrm{LT}}}
\newcommand{\Sumf}{\mathrm{Sum}_f}
\newcommand{\Sumn}{{\mathrm{Sum}_{1/n}}}
\newcommand{\EDfin}{{\mathrm{ED}_{\mathrm{fin}}}}

\newcommand{\pcolon}{\colon\!\!\subseteq}

\newcommand{\tto}{\rightrightarrows}

\newcommand{\fr}{\mbox{}^\smallfrown}

\newcommand{\dar}{\!\!\downarrow}
\newcommand{\uar}{\!\!\uparrow}
\newcommand{\GUV}{\mathfrak{G}(\calU,\calV)}

\newcommand{\up}{\uparrow}
\newcommand{\Vod}{\calV^{\otimes[\delta]}}
\newcommand{\Uod}{\calU^{\otimes[\delta]}}

\newcommand{\wT}{\widetilde{T}}

\newcommand{\Bmid}{\,\mathrel{\Big|}\,} 
\newcommand{\bmid}{\mathrel{\big|}} 

\newcommand{\mer}{\rm Merlin}
\newcommand{\nim}{{\rm Nimue}}
\newcommand{\art}{{\rm Arthur}}

\newcommand{\bil}[2]{(#1 \; | \;#2)}

\newcommand{\phin}{\Phi^{[\lranglet{m}{n}]}}
\newcommand{\wphin}{\widetilde{\Phi}^{[n]}}

\newcommand{\Dt}{\mathcal{D}_{\rm T}}

\makeatletter
\newcommand*{\relrelbarsep}{.386ex}
\newcommand*{\relrelbar}{%
	\mathrel{%
		\mathpalette\@relrelbar\relrelbarsep
	}%
}
\newcommand*{\@relrelbar}[2]{%
	\raise#2\hbox to 0pt{$\m@th#1\relbar$\hss}%
	\lower#2\hbox{$\m@th#1\relbar$}%
}
\providecommand*{\rightrightarrowsfill@}{%
	\arrowfill@\relrelbar\relrelbar\rightrightarrows
}
\providecommand*{\leftleftarrowsfill@}{%
	\arrowfill@\leftleftarrows\relrelbar\relrelbar
}
\providecommand*{\xrightrightarrows}[2][]{%
	\ext@arrow 0359\rightrightarrowsfill@{#1}{#2}%
}
\providecommand*{\xleftleftarrows}[2][]{%
	\ext@arrow 3095\leftleftarrowsfill@{#1}{#2}%
}
\makeatother

\makeatletter
\newcommand*\Lslash
{%
	\text{\rlap{\kern.03em\@xxxii}}L%
}
\makeatother

\newcommand{\calUS}{\overline{\calU}}
\newcommand{\calVS}{\overline{\calV}}
\newcommand{\calWS}{\overline{\calW}}

 \newcommand{\wgamma}{\widetilde{\gamma}^{[n]}}

\newcommand{\cosimp}[3]{\xymatrix@1{#1 \ar@<.4ex>[r] \ar@<-.4ex>[r] & {\ }#2 \ar@<0.8ex>[r] \ar[r] \ar@<-.8ex>[r] & {\ } #3 \ar@<1.2ex>[r] \ar@<.4ex>[r] \ar@<-.4ex>[r] \ar@<-1.2ex>[r] & \cdots }}
\newsavebox{\pullback}
\sbox\pullback{%
	\begin{tikzpicture}%
	\draw (0,0) -- (1ex,0ex);%
	\draw (1ex,0ex) -- (1ex,1ex);%
	\end{tikzpicture}}

\makeatletter
\newcommand*\dotp{\mathpalette\dotp@{.5}}
\newcommand*\dotp@[2]{\mathbin{\vcenter{\hbox{\scalebox{#2}{$\m@th#1\bullet$}}}}}
\makeatother

\newcommand{\equalizer}[2]{\xymatrix@1{#1 \ar@<.4ex>[r] \ar@<-0.4ex>[r] & {\ } #2}}

\newcommand{\adjunction}[4]{\xymatrix@1{#1{\ } \ar@<-0.3ex>[r]_{ {\scriptstyle #2}} & {\ } #3 \ar@<-0.3ex>[l]_{ {\scriptstyle #4}}}}

\usepackage{epigraph}

\usepackage{xcolor}
\usepackage{csquotes}
\usepackage{lipsum}

\definecolor{quotemark}{gray}{0.7}
\makeatletter
\newlength\origparskip

\newcommand{\fquote}{%
	\@ifnextchar[{\fquote@i}{\fquote@i[]}
}

\def\fquote@i[#1]{%
	\@ifnextchar[{\fquote@ii{#1}}{\fquote@ii{#1}[]}
}%

\def\fquote@ii#1[#2]{%
	\def\pqm@tempa{#1}%
	\def\pqm@tempb{#2}%
	\noindent
	\list
	{}
	{\setlength{\leftmargin}{0.3\textwidth}%
		\setlength{\rightmargin}{0.1\textwidth}%
		\setlength{\origparskip}{\parskip}}%
	\item[]%
	\begin{picture}(0,0)%
	\put(-15,-8){\makebox(0,0){\scalebox{4}{%
				\textcolor{quotemark}{\textquotedblright}}}}%
	\end{picture}%
	\begingroup
	\itshape
	\ignorespaces}%

\def\endfquote{%
	\endgroup
	\par
	\raggedleft
	\ifx\pqm@tempa\empty
	\else
	{\bfseries --- \pqm@tempa\par}%
	\setlength{\parskip}{\origparskip}%
	\ifx\pqm@tempb\empty
	\else
	(\pqm@tempb)%
	\fi
	\fi
	\par
	\endlist}
\makeatother

\begin{document}
	\bibliographystyle{alpha}
	\newtheorem{theorem}{Theorem}[section]
	\newtheorem*{theorem*}{Theorem}
	\newtheorem*{condition*}{Condition}
	\newtheorem*{definition*}{Definition}
	\newtheorem*{corollary*}{Corollary}
	\newtheorem{proposition}[theorem]{Proposition}
	\newtheorem{lemma}[theorem]{Lemma}
	\newtheorem{corollary}[theorem]{Corollary}
	\newtheorem{claim}[theorem]{Claim}
	\newtheorem{conclusion}[theorem]{Conclusion}
	\newtheorem{hypothesis}[theorem]{Hypothesis}
	\newtheorem{conjecture}[theorem]{Conjecture}
	\newtheorem{setup}[theorem]{Setup}
	\newtheorem{sumthm}[theorem]{Summary Theorem}

	\newtheorem{maintheorem}{Theorem}
	\renewcommand*{\themaintheorem}{\Alph{maintheorem}}
	\newtheorem{mainprop}[maintheorem]{Proposition}
	
	\theoremstyle{definition}
	\newtheorem{definition}[theorem]{Definition}
	\newtheorem{question}[theorem]{Question}
	\newtheorem{action}[theorem]{Action Item}
	\newtheorem{answer}[theorem]{Answer}
	\newtheorem{goal}[theorem]{Goal}
	\newtheorem{exercise}[theorem]{Exercise}
	\newtheorem{remark}[theorem]{Remark}
	\newtheorem{observation}[theorem]{Observation}
	\newtheorem{discussion}[theorem]{Discussion}
	\newtheorem{guess}[theorem]{Guess}
	\newtheorem{example}[theorem]{Example}
	\newtheorem{condition}[theorem]{Condition}
	\newtheorem{warning}[theorem]{Warning}
	\newtheorem{notation}[theorem]{Notation}
	\newtheorem{construction}[theorem]{Construction}
	
	\newtheorem{problem}{Problem}
	\newtheorem{fact}[theorem]{Fact}
	\newtheorem{thesis}[theorem]{Thesis}
	\newtheorem{convention}[theorem]{Convention}
	\newtheorem{summary}[theorem]{Summary}

	\title{The Game-Theoretic Kat\v{e}tov Order and Idealised Effective Subtoposes}

	\author{Takayuki Kihara and Ming Ng}
	\thanks{TK was partially supported by JSPS KAKENHI Grant Numbers 22K03401 and 23K28036. MN was partially supported by a JSPS PostDoc Fellowship (Short-Term). } 
	\begin{abstract} This paper addresses the longstanding problem of determining the structure of the $\leq_{\mathrm{LT}}$-order in the Effective Topos, known to effectively embed the Turing degrees. In a surprising discovery, we show that the $\leq_{\mathrm{LT}}$-order is in fact tightly controlled by the combinatorics of filters on $\omega$, raising deep questions about how combinatorial and computable complexity interact, both within this order and beyond it.
    
    To make the connection precise, we introduce a game-theoretic (``gamified'') variant of the Kat\v{e}tov order on filters over $\omega$, which turns out to exhibit a striking mix of coarseness and subtlety. For one, it is strictly coarser than the classical Rudin-Keisler order and, when viewed dually on ideals, collapses all MAD families to a single equivalence class. On the other hand, the order also supports a rich internal structure, including an infinite strictly ascending chain of ideal classes, which we identify by way of a new separation technique. 
    
 From the computability-theoretic perspective, we show that a computable (and extended) variant of the gamified Kat\v{e}tov order is isomorphic to the original $\leq_{\mathrm{LT}}$-order. Moreover, our work brings into focus a new degree-spectrum invariant for filters $\mathcal{F}$, $$\mathcal{D}_{\rm T}(\mathcal{F}):=\{\,[f\colon\omega\to\omega] \mid f\leq_{\mathrm{LT}} \mathcal{F} \},$$ 
   which is shown to always determine a proper initial segment of the Turing degrees. Extending this, given any non-principal $\Delta^1_1$ filter $\calF$, 
   we show that $\mathcal{D}_{\rm T}(\calF)$ is precisely the class of hyperarithmetic degrees. This significantly generalises previous results obtained by van Oosten \cite{vO14} and Kihara \cite{Kih23}.

The proofs draw on ideas from general topology, descriptive set theory, and computability theory.
	\end{abstract}

	\maketitle
	
	The guiding perspective of this paper is that different notions of complexity, arising in areas as far apart as computability theory and set theory, can be seen to be controlled by the same underlying mechanism — once placed in the right topological framework. Logicians are familiar with the key notion of a filter, either as a generalised point or as an abstract notion of largeness (e.g. in ultraproducts). Topos theorists push this idea much further by asking: what is a {\em generalised space}?
	
	We will focus on one such answer, namely {\em Lawvere-Tierney topologies} [hereafter: LT topologies]. Informally, an LT topology gives an abstract notion of localness. A space $X$ might be said to satisfy some property $P$ locally if $X$ can be covered by a family of opens  $\{U_i\}_{i\in I}$, each satisfying the property $P$. Abstracting this, one might require:
	\begin{enumerate}
		\item If ``$P$ implies $Q$'', then ``locally-$P$ implies locally-$Q$''
		\item $P$ implies locally-$P$
		\item Locally-locally-$P$ is equivalent to locally-$P$.
	\end{enumerate}
	Such a notion is well-defined in any (elementary) topos, and there are several equivalent formulations; compare, for instance, \cite[\S V.1]{MM} with \cite[\S 16]{HylandEffective}. Importantly, given any topos, the collection of all its LT topologies carries a natural partial order $\clt$, under which it forms a complete lattice. We call this the {\em Lawvere--Tierney order}, but will usually write ``$\clt$-order'' for short.

    The present paper focuses on the longstanding problem of understanding this $\clt$-order in the case of the Effective Topos $\Eff$, which is known to effectively embed the Turing degrees. In this setting, the LT topologies admit the following explicit description:
	
	\begin{definition}\label{def:LT-topology} 
    An {\em LT topology} in $\Eff$ is an endomorphism 
		$$j\colon \Pw\to \Pw$$  
		subject to the following conditions:
		\begin{enumerate}
			\item $\forall p,q\in\Pw. \left( p\to q\right )\to \left(j(p)\to j(q)\right)$
			\item  $\forall p\in\Pw. \left( p\to j(p)\right)$
			\item  $\forall p\in\Pw.\,\, j\circ j(p)=j(p)$
		\end{enumerate}	
		Here, $\Pw$ denotes the powerset of the natural numbers, and ``$\to$'' denotes the standard realisability interpretation of implication. That is, fixing an enumeration of partial computable functions $\{\varphi_{e}\}_{e\in\w}$,
		$$p\to q:=\left\{e\mid  \forall a\in p\,,\, \varphi_{e}(a) \,\text{is defined and}\, \varphi_e(a)\in q\right\}.$$
		This means, for instance, $j$ validates Condition (2) just in case there exists $e\in \w$ such that $e\in (p\to j(p))$ for all $p\in\Pw$.
	\end{definition} 

	Once we fix a notion of localness, one can ask precise questions about when locally consistent data assemble into a coherent global structure, as well as the barriers to such assembly. 
Viewing the LT topologies in this light points to a kind of hidden geometry underlying the Turing degrees, yet our understanding of this geometry — and of how different notions of localness interact — remains limited.
	
	
This paper significantly advances this understanding by uncovering a surprising connection: the $\clt$-order on LT topologies in $\Eff$ is tightly connected to the combinatorics of ultrafilters on $\w$. To make this precise, we introduce a game-theoretic variant of the Kat\v{e}tov order on the upper sets over $\w$. An organising theme of this paper is exploring how this order admits meaningful connections to well-known orders in set theory (e.g. Rudin-Keisler, Tukey) whilst remaining finely tuned to shifts in complexity within the LT topologies. For further background and discussion of the main results, see the introduction below.

	\subsection*{Acknowledgments} It is a pleasure to acknowledge S. Fuchino, D. Madore, M. Maietti, M. Malliaris and S. Nakata for thoughtful remarks that improved the exposition of the paper. Special thanks are due to F. Parente for bringing to our attention the negative solution of Isbell's Problem regarding the Tukey Order. 
	
	
	\section{Introduction}\label{sec:Intro}
	
	The Effective Topos $\Eff$ was introduced by Hyland \cite{HylandEffective} as a category-theoretic framework for computable mathematics. There are various levels on which to appreciate both the power and depth of this perspective (see, for instance, \cite{vOBook}). Abstractly, each LT topology $j$ determines a subtopos $\Eff_j\subseteq \Eff$, and these subtoposes are partially ordered by inclusion. In this paper, we shall focus on the corresponding partial order on the LT topologies, given explicitly by\footnote{For the category theorist: here we view the LT topologies externally, but see the discussion after \cite[Proposition 16.1]{HylandEffective} for the internal perspective. Also, when we say ``these subtoposes are partially ordered by inclusion'', the term ``inclusion'' should be taken to mean {\em geometric inclusion} in the sense of geometric morphisms.}
	$$j\clt k : \iff \forall p.j(p)\to k(p) \, \text{is valid}\,.$$
   A key structural fact is that the Turing Degrees embed effectively into this $\clt$-order \cite[\S 17]{HylandEffective}. Concretely, each decision problem $A \subseteq \w$ determines an LT topology $j_A$, whose associated subtopos may be understood as ``the universe of $A$-computable mathematics'', and this assignment preserves order: $j_A \clt j_B$ if and only if $A \leq_T B$. This perspective substantiates the opening remarks of this paper: topology, properly understood, can be used to calibrate computable complexity in a meaningful way.

	Nonetheless, in the 40 plus years since Hyland's seminal paper, progress on understanding the $\clt$-order in $\Eff$ has been slow. General topos theory establishes the existence of a maximum and minimum class on non-trivial LT topologies \cite[Ch. 5]{PittsPhD}:\footnote{This paper focuses on the non-trivial LT topologies, but for completeness: the  {\em trivial topology} in $\Eff$ is the endomorphism $\top\colon \Pw\to\Pw$ sending all subsets to $\w$. By definition of $\clt$, it is obvious $\top$ defines the $\eclt$-maximum class.
	}
	\begin{itemize}
		\item[] \textbf{Minimum.} The {\em identity topology} $\id$, which sends $p\mapsto p$. 
		\item[] \textbf{Maximum.} The {\em double-negation topology} $\dn$, which sends $\emptyset\mapsto\emptyset$, and everything else to $\w$.
	\end{itemize}
	Phoa \cite{Pho89} further showed that any non-trivial topology $j$ that $\clt$-dominates all the Turing degree topologies must be the double negation topology. A couple other non-trivial LT topologies were also identified: Pitts~\cite{PittsPhD} constructed a topology in which the representable functions are exactly the hyperarithmetic ones\footnote{This is shown in \cite{vO14}. Note that this does not imply correspondence with hyperarithmetic realisability. In fact, its multi-valued strength possesses a strength exceeding hyperarithmetic; see \cite{Kih23}.}, while van Oosten introduced a topology corresponding to Lifschitz realisability ~\cite{vO91,vO96}. Aside from these isolated results, not much else was known for a long while.
	
	The main difficulty has to do with the level of abstraction involved. Most toposes occuring ``in nature'' (say, in algebraic geometry) are {\em Grothendieck toposes}, which are by definition categories of sheaves on some Grothendieck site. This site presentation allows for concrete descriptions of their LT topologies -- e.g. as Grothendieck topologies \cite[Theorem V.4.1]{MM}, as quotients of a first-order geometric theory \cite[\S 3.2]{Caramello} etc. By contrast, the Effective Topos is the standard example of a topos which is {\em not} a Grothendieck topos. In practice, this means having fewer tools at our disposal, and so the $\clt$-order on LT topologies in $\Eff$ remains comparatively mysterious.

	A key turning point emerged in Lee-van Oosten's striking discovery that all LT topologies in $\Eff$ are constructed from certain basic building blocks \cite{Lee,LvO13}. The following summary theorem makes this statement precise, and includes a couple other highlights. 
	\begin{sumthm}\label{sumthm-1} Any family of subsets $\calA\subseteq \Pw$ defines an LT topology, which we denote as $$j_\calA\colon \Pw\to \Pw.$$
		An LT topology of this form is called a {\em basic topology}. In particular:
		\begin{enumerate}[label=(\roman*)]
			\item Every LT topology in $\Eff$ is a recursive join of basic topologies.
			\item Denote $\bigcap \calA$ to be the intersection of all $A\in\calA$. \underline{Then}: 
			$$j_\calA=\id \iff \bigcap\calA\neq \emptyset\quad .$$
			\item Say $\calA$ has the {\em $n$-intersection property} if for any $n$ elements $A_1,\dots,A_n\in\calA$, their intersection is non-empty: $$\bigcap^{n}_{i=1} A_i\neq \emptyset.$$ 
			Suppose $\calA\subseteq\Pw$ has the $n$-intersection property whereas $\calB\subseteq\Pw$ {\em fails} this property. 
			\underline{Then}:
			$$j_\calB \not\clt j_{\calA}.$$
		\end{enumerate}
	\end{sumthm}
	\begin{proof} For details on Lee-van Oosten's construction\footnote{In Lee-van Oosten's notation, a basic topology is represented as $L(G_\calA)$. See also Section~\ref{sec:poly-functors}.} of $j_\calA$, see \cite[\S 2]{LvO13}. Item (i) follows from \cite[Theorem 2.3]{LvO13}. Item (ii) is stated without proof in \cite[Proposition 3.1]{LvO13}, and attributed to Hyland-Pitts; the details are given in \cite[Calculation 3.3.1]{Lee}. Item (iii) is \cite[Proposition 4.17]{LvO13}.    
	\end{proof}

	These results bring into focus a central question: what kind of complexity is $\clt$ actually measuring?
	\begin{itemize}
		\item \textbf{Computable Complexity.} Hyland's effective embedding of the Turing degrees \cite{HylandEffective} already shows that $\clt$ detects the computable complexity of decision problems (single-valued functions). Subsequent work by the first author \cite{Kih22,Kih23} extends this picture: the full $\clt$-order on LT topologies reflects the computable complexity of a broader class of problems, namely the gamified Weihrauch degrees of search problems (partial multi-valued functions).

		\item \textbf{Combinatorial Complexity.} By contrast, Summary Theorem~\ref{sumthm-1} ties the complexity of LT topologies to the structure of subset families in $\w$. Items (ii) and (iii) are particularly suggestive; they show how $\clt$ systematically sounds out the depth of intersection within a family $\calA\subseteq\Pw$. It is therefore interesting to ask: what other intersection patterns might $\clt$ detect?
	\end{itemize}

	To proceed, let us translate item (ii) of the Summary Theorem as follows: $\clt$ induces a preorder on subset families $\calA\subseteq \Pw$, and $\calA$ belongs to the $\eclt$-minimum class iff $\bigcap\calA\neq\emptyset$. This restatement will be suggestive to the set theorist because of the following well-known fact:
	
	\begin{fact}[see e.g. {{\cite[Exercise 1.6.3]{Gold22}}}] The {\em Rudin-Keisler order} ($\lrk$) defines a preorder on ultrafilters over $\w$. In particular, there exists a minimum class, which is characterised by the principal ultrafilters. More explicitly: an ultrafilter $\calF\subseteq\Pw $ belongs to the $\erk$-minimum class iff $\bigcap\calF\neq\emptyset$.
	\end{fact}

	As it turns out, the Rudin-Keisler order is well-defined not just on ultrafilters, but also on upper sets over $\w$ (i.e. subset families closed upwards under $\subseteq$). Let us therefore distill our previous observations into the following test problem:
	
	\begin{problem}\label{prob:test} How do the partial orders $\clt$ and $\lrk$ compare on the upper sets over $\w$? In addition, how might their structural connections, if any, reflect back onto computability theory? 
	\end{problem}

	\subsection*{Discussion of Main Results} Our first main result is the discovery of a new partial order on upper sets over $\w$, which we call the  {\bf Game-Theoretic Kat\v{e}tov order} (or {\bf Gamified Kat\v{e}tov order}, for short). Sections~\ref{sec:conv} and \ref{sec:brave-new-order} lay the foundations, and establish the surprising connection: 

\begin{maintheorem}[Theorem~\ref{thm:GTK}]\label{mthm:main-thm}  The Gamified Kat\v{e}tov order, written suggestively as
	$$\calU\glt\calV,\qquad\qquad\qquad\calU,\calV\subseteq \; \text{upper sets}\,,$$
	 is a preorder on upper sets over $\w$. In particular:
	\begin{enumerate}[label=(\roman*)]
		\item The \emph{Gamified Kat\v{e}tov order} is equivalent to the {\em Kat\v{e}tov order} closed under {\em well-founded iterations of Fubini powers}. In particular, it is strictly coarser than {\em Rudin-Keisler} and {\em Kat\v{e}tov}.
		\item The {\em Gamified Kat\v{e}tov order} admits an explicit game-theoretic description, justifying its name.
		\item The {\em computable Gamified Kat\v{e}tov order} is equivalent to the {\em $\clt$-order} on upper sets over $\w$.
	\end{enumerate}
\end{maintheorem}

Theorem~\ref{mthm:main-thm} highlights several design features which distinguish the Gamified Kat\v{e}tov order from other preorders in the set theory literature. First, it is formulated in terms of the Kat\v{e}tov order on upper sets, as opposed to e.g. the classical Rudin-Keisler order on ultrafilters. Second, it relies on ``well-founded iterations of Fubini powers'',  which we formalise via our notion of a $\delta$-Fubini power (Definition~\ref{def:delta-Fubini}). This notion is of independent interest, and generalises the Ramsey-theoretic construction of generating an ultrafilter by trees on a front, as seen in \cite{Dob20}. Third, rather than working with the Kat\v{e}tov order {\em as is}, we close it under all such well-founded Fubini iterations. Related closure principles appear implicitly in Blass' work on the Kleene degrees of ultrafilters \cite{Bl85}, but were not isolated there for explicit study.

Taken by themselves, these choices may seem puzzling to the reader. In fact, they reflect a single underlying structural phenomenon, which becomes visible only after a shift in perspective. The key takeaway from Theorem~\ref{mthm:main-thm} is that our Gamified Kat\v{e}tov order shows up naturally in three different guises, each of interest in its own right:
	\begin{itemize} 
		\item[$\diamond$] as a natural variant of the Kat\v{e}tov order (set theory);
		\item[$\diamond$] as a concrete representation of abstract topologies, namely LT topologies (category theory);

        \item[$\diamond$] as a new hierarchy of computability notions, namely ``computability by majority''\footnote{This involves running multiple computations in parallel, and adopting the result decreed by the ``majority'' of these computations; for a discussion on what majority means and further details, see Section \ref{sec:comp-majority}.
        } (computability theory).

	\end{itemize}
	Seen in this light, the topos theory brings to the surface a series of deep structural connections that we would otherwise be hard-pressed to identify. The remainder of this paper explores some consequences of this unexpected convergence.

\medskip
    
	Section~\ref{sec:Tukey} considers another well-known partial order called the {\em Tukey order} and proves the following:
	
	\begin{maintheorem}[Theorem~\ref{thm:Tukey}]\label{mthm:Tukey} The Gamified Kat\v{e}tov and Tukey order are incomparable on filters over $\w$ (in ZFC).
	\end{maintheorem}
	
	Theorem~\ref{mthm:Tukey} speaks to two different readers. For the \textbf{set theorist}, it is worth recalling that Tukey is also coarser than Rudin-Keisler. Theorem~\ref{mthm:Tukey} thus indicates that the Gamified Kat\v{e}tov order detects something related to, yet fundamentally distinct from, the cofinality types of filters. For the \textbf{constructive mathematician}, we remark that incomparability cannot be proved in ZFC if one restricts to just ultrafilters;  
	in order to obtain decisive evidence of incomparability, one needs to consider a wider class of objects. One can therefore read Section~\ref{sec:Tukey} as a case study on how relying on the Axiom of Choice sometimes obscures the basic picture as opposed to clarifying it. 

\medskip
    
	Section~\ref{sec:separation} develops the combinatorial perspective on the Gamified Kat\v{e}tov order, and highlights how its structure posseses a kind of productive coarseness. On the one hand, in stark contrast to the Rudin-Keisler order, we show that all MAD (Maximally Almost Disjoint) families in $\w$ are $\elt$-equivalent. On the other hand, the $\glt$ order is still sensitive to other shifts in combinatorial complexity:
	
	\begin{maintheorem}[Theorem~\ref{thm:strict-ascend}]\label{mthm:separation} Consider the $\glt$-order induced on the ideals of $\w$.\footnote{Recall: any ideal $\calI$ on $\w$ defines $\calI^* :=\{ \w\setminus A \mid A\in\calI \}$, i.e. the dual filter of $\calI$.} \underline{Then}, there exists an infinite strictly ascending chain 
	$$\Fin \sglt \EDfin[1]\sglt \dots \sglt \EDfin[m]\sglt \EDfin[m+1]\sglt \dots \sglt \EDfin\sglt \mathrm{Sum}_{1/n}\,$$
	on ideals over $\w$.
	\end{maintheorem}
	\noindent All relevant definitions will be given in  Section~\ref{sec:separation}. For now, let us say a few words on why we view this result to be an important highlight of this paper. 
    
    \begin{itemize}
        \item 
    For the \textbf{set theorist}, Theorem~\ref{mthm:separation} essentially says: despite the apparent coarseness of the $\glt$-order (e.g. the collapse of MAD families), there still exists a high degree of separation within the order, pointing to a rich internal structure. The proof is also interesting: it relies on a newly-developed separation technique (Section~\ref{sec:separation-technique}) which we expect to find useful applications elsewhere.
    

\item For the \textbf{computability theorist}, Theorem~\ref{mthm:main-thm} combines with Theorem~\ref{mthm:separation} to yield the same separation within the $\clt$-order; see Corollary~\ref{cor:strict-clt}. Concretely, the chain above realises infinitely many distinct strengths of ``computability by majority'' notions (where the notion of ``majority'' is determined by the ideal), and thus provides an answer to \cite[Question~1]{Kih23}. Combined with results already obtained in \cite{Kih23}, we obtain Figure~\ref{fig:game-Katetov}.
    \end{itemize}

\begin{figure}[h!]
\includegraphics[width=12cm]{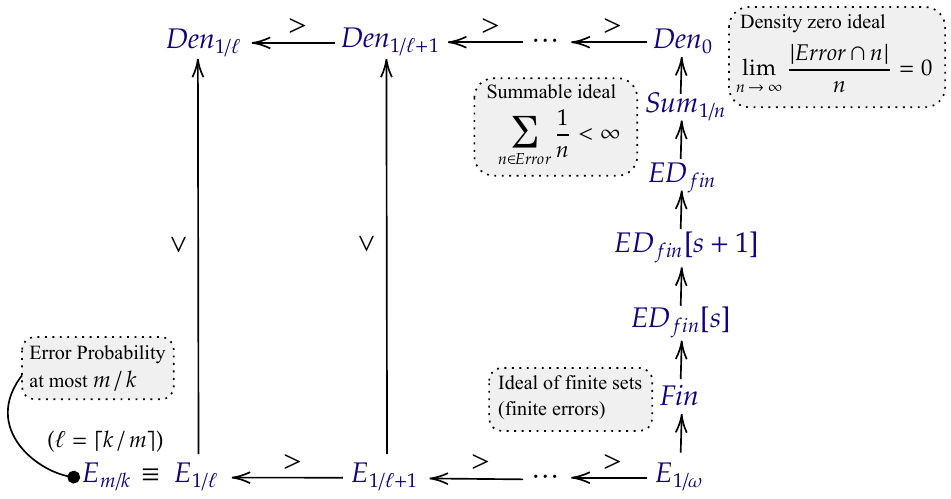}
\caption{The structure of the Gamified Kat\v{e}tov order on lower sets}\label{fig:game-Katetov}
\end{figure}

In Section~\ref{sec:Turing}, we turn to the interaction between combinatorial and computable complexity within the $\clt$-order. Serious connections between set theory and computability theory are, of course, not new, and there are many previous results in this direction. What is distinctive in the present setting, however, is that the $\clt$-order provides a uniform framework in which both forms of complexity can be compared directly. 

\medskip

Our first main result in this vein shows that the Gamified Kat\v{e}tov order can be understood as the $\clt$-order relativised to an arbitrary Turing oracle. More precisely:

\begin{maintheorem}[Theorem~\ref{thm:preorder}]\label{mthm:preorder} Let $\calU,\calV$ be upper sets on $\w$.
\underline{Then}, $$\calU\glt\calV \iff \calU\clt f+ \calV\;,\quad\text{for some }\; f\colon\w\to\w\;.$$
\end{maintheorem}

Theorem~\ref{mthm:preorder} highlights the implicit role that Turing degrees play in the overall structure of the $\glt$-order. However, this interaction between the computable and combinatorial also emerges at the level of individual filters. By Summary Theorem~\ref{sumthm-1}, every filter $\calF\subseteq\Pw$ determines an initial segment of the Turing degrees:
$$\calD_{\rm T}(\calF):=\big \{ \,[f\colon\w\to\w] \bmid f\clt \calF \big\}\,.$$
We call $\Dt(\calF)$ the {\em Turing degree profile} of $\calF$. 
One may wonder if filters and Turing degrees are largely orthogonal within $\clt$, with relatively few direct points of contact. 
Encouragingly, the following Cofinality Theorem shows this is very much not the case.

\begin{maintheorem}[Cofinality, Theorem~\ref{thm:cofinal}]\label{mthm:cofinality}  The following is true:
		\begin{enumerate}[label=(\roman*)]
			\item For any $f\colon\omega\to\omega$, there exists a summable ideal $\calI$ such that
			$$f \clt \calI^*,$$
			where $\calI^*$ is the dual filter.
			\item There does not exist a filter $\calF\subseteq\Pw$ such that 
			$$f\clt\calF,\qquad\text{for all Turing oracles}\; f\;.$$
		\end{enumerate}
	\end{maintheorem}
\noindent In other words, every Turing degree $f$ belongs to the profile $\calD_{\rm T}(\calF)$ of some filter $\calF$, yet no single filter captures all the Turing degrees. Theorem~\ref{mthm:cofinality} therefore identifies $\calD_{\rm T}$ as a meaningful new invariant by way of an unexpected connection: one can systematically probe the combinatorial structure of filters on $\w$ by examining which Turing degrees they bound -- and which they necessarily omit. 

\medskip

Building on this perspective, we revisit and extend some results from the topos theory literature. In our language, van Oosten \cite{vO14} and Kihara \cite{Kih23} previously showed that $\calD_{\rm T}({{\rm Fin}^\ast})$ and $\calD_{\rm T}({{\rm Den}_0^\ast})$ are exactly the hyperarithmetic degrees.\footnote{Here, $\Fin$ denotes the ideal of finite sets,
${\rm Den}_0$ the ideal of asymptotic density zero sets, and we take their dual filters.} It turns out this extends to all non-principal $\Delta^1_1$-filters:

\begin{maintheorem}[Theorem~\ref{thm:filter-hyperarithmetic}] Let $\calF\subseteq \Pw$ be a non-principal $\Delta^1_1$-filter. Then, for any $f\colon\w\to\w$, 
$$f\clt \calF \iff f\;\;\text{is hyperarithmetic}\;.$$
\end{maintheorem}

Finally, a key technical ingredient in the proof of Cofinality Theorem~\ref{mthm:cofinality} is Phoa's theorem \cite{Pho89}, which relates the double-negation topology to the Turing degrees. Our observation is that this is not an isolated coincidence. In fact, it reflects a more general structural phenomenon within the $\clt$-order:


	\begin{maintheorem}[Generalised Phoa, Theorem~\ref{thm:gen-phoa}]\label{mthm:gen-phoa} Let $\calU\subseteq \Pw$ be an upper set, and let $\calVS$ be a partial $\omega$-sequence of upper sets on $\w$. \underline{Then},
		$$\calU\clt \calVS \iff F\clt \calVS\;\;\text{for every}\;\calU\text{-function}\; F\quad.$$
	\end{maintheorem}
\noindent Informally, any upper set $\calU$ may be regarded either as a single combinatorial notion,
 or as the generator of a family of derived computability-theoretic notions (a ``$\calU$-function'').
Theorem~\ref{mthm:gen-phoa} states that domination of $\calU$ within $\clt$ is equivalent to the uniform domination of all derived oracles associated to $\calU$. Phoa's original theorem is the case where $\calU$ corresponds to the double-negation topology $\dn$, and the associated $\calU$-functions recover the classical Turing degrees. 

What makes this perspective interesting is that many familiar notions in \textbf{computability theory} and \textbf{descriptive set theory} exhibit this fundamental relationship. To illustrate, fix an $\omega$-parametrised complexity class $\Gamma$ (e.g.~$\Sigma^i_n$ and $\Pi^i_n$), and restrict attention to the family of all $\Gamma$-sets from $\calU$. This presents a natural choice problem: given a $\Gamma$-index $e$ of a $\calU$-large set $A_e$, output an element of $A_e$. Theorem~\ref{mthm:gen-phoa} then states that an oracle notion is more complex than $\calU$ if and only if it is more complex than the $\Gamma$-restriction of $\calU$ for any complexity class $\Gamma$. Further discussion is deferred to Section~\ref{sec:generalised-Phoa}. 

	\medskip
	
	We close the introduction by returning to the broader perspective. Problem~\ref{prob:test} asked about the relationship between the Rudin-Keisler order and the $\clt$-order, and the implications for computability theory. Together, Theorems~\ref{mthm:main-thm} - ~\ref{mthm:gen-phoa} provide a fairly complete answer, but they do so by revealing deep structural links that are only beginning to come into focus. In particular, these connections give new energy to exploring how important dividing lines from one area — e.g. cardinal invariants of the continuum in set theory, separation problems in constructive reverse mathematics, combinatorial notions of independence in model theory, etc. — may find meaningful translations across fields. As such, far from closing off a line of investigation, our results provide the vocabulary for asking new questions, laying the groundwork for a wide-ranging research programme. Section~\ref{sec:test-prob} develops this view, and records some test problems for future investigation.

	\tableofcontents

	\section{Notation and Conventions}\label{sec:conv}

	\begin{convention} 	$\w$ denotes the set of natural numbers, and $\calP(I)$ the power set of a set $I$. Typically, $I\subseteq\w$.
	\end{convention}
	
	
	\begin{convention}[Strings]\label{con:string} $\ww$ denotes the set of infinite strings of natural numbers; $\lww$ denotes the set of all finite strings. For any string $x$ (finite or infinite), denote the initial segment of length $n$ as $x\rstr n$; such an initial segment is called a {\em prefix} of $x$. Further, given finite strings $\sigma,\tau\in\lww$ :
		\begin{itemize}
			\item 	Write $\sigma\preceq\tau$ if $\tau$ extends $\sigma$. 
			Call a pair $(\sigma,\tau)$ {\em incomparable} just in case $\sigma\not\preceq \tau$ and $\tau\not\preceq \sigma$. 
			
				\item The empty string is denoted $\o$; it is clear $\o \preceq \sigma$ for all $\sigma\in\lww$.
			\item  The concatenation of $\sigma$ and $\tau$ is denoted by $\sigma\fr\tau$; for a singleton $\lrangles{m}$, abbreviate $\sigma\fr\lrangles{m}$ as $\sigma\fr m$.
			\item The length of $\sigma$ is denoted $|\sigma|$. In this language, the Cartesian power $\w^n$ is the set of all (incomparable) strings $\sigma\in\lww$ such that $|\sigma|=n$. Notice this definition excludes the empty string $\o$.
		\end{itemize}
	\end{convention}
	
	\begin{convention}[Tree] A {\em tree} is a subset $T\subseteq \lww$ whereby:
	\begin{enumerate}[label=(\alph*)]
		\item $T$ contains the empty string $\o$, the {\em root} of $T$. 
		\item $T$ is downward-closed under initial segments: if $\tau \preceq \sigma \in T$, then $\tau \in T$.
	\end{enumerate}	
We often refer to the elements of $T$ as its {\em nodes}. The {\em height} of a tree $T$ is the maximum length of a string in $T$ -- e.g. the empty tree consisting only of the root $\o$ has height 0. The {$k$-truncation} of $T$, denoted $T\rstr k$ is the set of all nodes $\sigma\in T$ such that $|\sigma|\leq k$. Finally, an {\em infinite path through $T$} is a string $p \in \ww$ such that $p \rstr n \in T$ for all $n$. We write $[T]$ for the set of all infinite paths through $T$. 
	\end{convention}
	
	\begin{convention}[Partial Functions]
The notation $f\pcolon X\to Y$ denotes $f$ is a partial function from $X$ to $Y$, and the notation $f\pcolon X\tto Y$ denotes $f$ is a partial multi-valued function from $X$ to $Y$, i.e., $f\pcolon X\to\mathcal{P}(Y)$.
We often write $f(x)\dar$ if $x\in \dom(f)$ and $f(x)\uar$ otherwise;
	\end{convention}

	\begin{convention}[Partial Continuous Functions]\label{con:pcf} 	 
	A {\em partial continuous function} $$\Phi\pcolon \ww \to \w$$
		is determined by a partial function $\varphi\pcolon \lww \to \w$ satisfying:
		\begin{enumerate}[label=(\alph*)]
			\item {\em Extension.}
				$$p \in \dom(\Phi) \,\, \text{and} \,\, \Phi(p) = c \iff \exists n\in\w \,\,\text{s.t.}\, p \rstr n \in \dom(\varphi) \,\,\text{and}\,\, \varphi(p \rstr n) = c.$$
			\item {\em Coherence.} If $\sigma,\tau\in\dom(\varphi)$ and $\sigma\preceq \tau$, then $\varphi(\sigma)=\varphi(\tau).$ 
		\end{enumerate}
		
		Condition (a) says $\Phi$ extends $\varphi$ from finite to infinite strings in the obvious way; Condition (b) ensures the extension is well-defined.
Alternatively, Condition (b) may be replaced with 
	\begin{enumerate}
	\item[(b$^\prime$)] {\em Antichain.} ${\rm dom}(\varphi)$ is an {\em antichain} (i.e., a set of pairwise incomparable strings).
	\end{enumerate}
	For (b)$\Rightarrow$(b$^\prime$), restrict the domain of $\varphi$ to $\preceq$-minimal elements. For
(b$^\prime$)$\Rightarrow$(b), notice we may extend $\dom(\varphi)$ to its $\preceq$-upward closure while preserving coherence.
In such a case, we often identify $\Phi$ with $\varphi$, writing $\Phi(\sigma)$, which is just $\varphi(\sigma)$.
	\end{convention}

\begin{convention}[Upper Sets and Filters]\label{con:filters} Any collection of subsets $\calA\subseteq \Pw$ is called a {\em subset family}, and will be denoted in caligraphic font. In particular:
\begin{itemize}
	\item An {\em upper set} $\calU \subseteq \Pw$ is a subset family upward closed under $\subseteq$, i.e. if $A \in \calU$ and $A \subseteq B$, then $B \in \calU$. To avoid trivialities, we assume $\calU$ is non-empty and  $\emptyset\notin\calU$. A lower set is defined dually.
	\item  For an upper set $\calU$, its {\em dual lower set} is $\calU^*:= \{A \subseteq \w : \w \setminus A \in \calU\}$; the dual upper set is defined analogously. 
	\item A {\em filter} is an upper set closed under finite intersections; dually, an {\em ideal} is a lower set closed under finite unions. 
\end{itemize}	
\end{convention}

\begin{convention}[Base-Change]\label{con:base-change}
We often consider subset families $\calA \subseteq \calP(\w^n)$ defined on finite powers of $\w$.
To regard them as subset families in $\Pw$, as usual, fix a computable bijection
$$\lranglet{-}{-}\colon \w\times\w\longrightarrow \w$$
and iterate this construction: $\langle x_1,x_2,\dots,x_n\rangle:=\langle x_1,\langle x_2,\dots,x_n\rangle\rangle$.
In general, as in standard literature, a finite string $\sigma\in\lww$ is also implicitly identified with a natural number via a computable bijection 
\[\code{-}\colon\lww\simeq\omega.\]
To emphasise this coding for a given $A\subseteq\lww$, we use the notation $\code{A}=\{\code{\sigma}\in\w\mid\sigma\in A\}$ and for $\calA\subseteq\mathcal{P}(\lww)$, use $\code{\calA}=\{\code{A}:A\in\calA\}$.
%
%
\end{convention}

\begin{convention}[Upward Closure] For any subset family $\calA \subseteq \Pw$, its {\em upward closure} is defined by
	\[
	\calA^{\up} := \{X \subseteq \w \bmid \exists A \in \calA \text{ with } A \subseteq X\}.
	\]
 $\calA^\up$ is an upper set by construction. If $\calA \subseteq \mathcal{P}(\w^n) $, we write $\calA^{\up}$ to mean the upward closure of $\code{\calA}$ in $\Pw$.
\end{convention}

\begin{convention}[Projections]\label{con:coordinate} Every number $x\in\w$ is identified with a pair $\lranglet{a}{b}$ via our fixed computable bijection $\lranglet{-}{-}$. 
Accordingly, we let
\[
\pi_0,\pi_1\colon\omega\to\omega,\qquad 
\pi_i(\langle a,b\rangle)=
\begin{cases}
a,& i=0,\\
b,& i=1,
\end{cases}
\]
and extend both projections coordinate-wise to $\lww$ and $\w$; in the case of $\lww$, set $\pi_0(\o)=\o=\pi_1(\o)$. Finally, since $\lranglet{-}{-}$ is computable by assumption, so are all the projection maps. 
\end{convention}

\begin{convention}[Realisability]\label{con:realis} Fix once and and for all an enumeration of partial computable functions $\{\varphi_{e}\pcolon\w\to\w\}_{e\in\w}$. Here $\varphi_e$ denotes the {\em $e$th partial computable function on $\w$}, and $\varphi^\alpha_e$ denotes the {\em $e$th partial computable function relative to oracle $\alpha$}. Following Convention~\ref{con:pcf}, for any partial function $\varphi$, write $\varphi(n)\dar$ if $\varphi(n)$ is $n\in\dom (\varphi)$ (i.e. ``$\varphi(n)$ is defined''), and write $\varphi(n)\uar$ if otherwise. For any pair of subsets $A,B\subseteq \w$, the {\em realisability implication} is defined as
	$$ A\to B :=\{ e \mid \forall a\in A\,, \; \varphi_e(a)\dar\text{and}\; \varphi(a)\in B\,\}.$$
\end{convention}

\begin{convention}[$\clt$ vs. $\glt$]\label{con:orders} There are two important orders in this paper. {\em A priori}, they are unrelated but of course the point of the paper is to make the connection. 
	\begin{enumerate}[label=(\roman*)]
		\item \textbf{Lawvere-Tierney order.} Unless stated otherwise, an LT topology $j$ will always mean an endomorphism $$j\colon \Pw\to\Pw$$ validating all three conditions in Definition~\ref{def:LT-topology}.
	The $\clt$-order on LT topologies is defined as
			$$j\clt k : \iff \forall p.j(p)\to k(p) \, \text{is valid}.$$	
Explicitly, this means there exists some partial computable function $\varphi_{e}$ witnessing {\em all} realisability implications listed in Definition~\ref{def:LT-topology}, and another witnessing $j\clt k$ ; for details, see \cite[\S 2]{LvO13}. Equivalence classes in $\clt$ are referred to as {\em LT degrees}. Finally, given a pair of subset families $\calA,\calB\subseteq \Pw$, we write
	$$\calA\clt \calB$$
	to mean $j_\calA\clt j_\calB$ in the sense of Summary Theorem~\ref{sumthm-1}.
		\item \textbf{Gamified Kat\v{e}tov order.} This order is defined on upper sets over $\w$, and is denoted $$\calU\glt\calV$$ 
		for upper sets $\calU,\calV\subseteq\Pw$. For the precise definition of $\glt$, see Definition~\ref{def:GTK}. The suggestive notation anticipates the connection with the original $\clt$-order. In particular, the $\mathsf{o}$ superscript stands for ``oracle'', for reasons that will be explained in due course (Conclusion~\ref{con:preorder}).
	\end{enumerate}
\end{convention}	

Finally, a remark on the exposition. To balance the demands of precision versus accessibility within the paper, the technical statements of our results are typically presented as \textbf{Theorems}, while \textbf{Conclusions} provide an informal, conceptual account of their significance.

	\section{A Brave New Order}\label{sec:brave-new-order} This section develops the set theory underlying the Gamified Kat\v{e}tov order, which will lead to connections with the Effective Topos $\Eff$. We start by laying the groundwork: Section~\ref{sec:ultf} reviews some basic preliminaries on ultrafilters, which motivates Section~\ref{sec:GTK} on the Gamified Kat\v{e}tov order. Section~\ref{sec:fubini} makes precise the slogan: ``The Gamified Kat\v{e}tov order is equivalent to the Kat\v{e}tov order closed under well-founded iterations of Fubini powers.'' Section~\ref{sec:game} explains the game-theoretic perspective. Finally, Section~\ref{sec:clt} makes the important connection: the $\clt$ order on upper sets is equivalent to the computable analogue of the Gamified Kat\v{e}tov order.
	
	\subsection{Preliminaries on Ultrafilters}\label{sec:ultf} These preliminaries are written with the non-set-theorist in mind. The expert reader may read them as a statement of intent, or skip to Section~\ref{sec:GTK} and refer back as needed.

	\begin{definition}[Ultrafilters]\label{def:ultrafilter} Let $\calU\subseteq \Pw$ be a filter.
		\begin{enumerate}[label=(\roman*)]
			\item Call $\calU$ an {\em ultrafilter} if it has an opinion  opinion on all subsets of $\w$: given any $A\subseteq\w$, either $A\in\calU$ or $\w\setminus A\in\calU$ (but not both). Notice $\calU$ is an ultrafilter iff it is maximal, i.e. not strictly contained by another filter.  
			\item Call $\calU$  {\em principal} just in case there exists some $S\subseteq\w$ such that $\calU=\{A\subseteq \w \mid S\subseteq A\}$; otherwise, call $\calU$ {\em non-principal}. Notice $\calU$ is a principal ultrafilter iff the set $S=\{\ast\}$ is a singleton.
		\end{enumerate} 
	\end{definition}
	
	In fact, Definition~\ref{def:ultrafilter} can be extended to define ultrafilters on an arbitrary set $I$, not necessarily countable. Ultrafilters have a long and varied history, see e.g.~\cite{Gold22}. Part of their subtlety has to do with their interaction with the Axiom of Choice. If we eliminate choice, then it is consistent with ZF that all ultrafilters are principal \cite{BlassUltrafilters}, and thus trivial from a combinatorial standpoint. By contrast, while invoking choice guarantees the existence of non-principal ultrafilters, their non-constructive nature makes them difficult to see directly, and thus understand. 
	
	There are various approaches to this problem. One strategy is to work indirectly: study ultrafilters via their effects on other objects, and draw our conclusions based on the interaction. 
	 Another approach is to develop a framework for comparing the complexity of ultrafilters directly, as below:
	
	\begin{definition}\label{def:rk-kat-order} Let $\calU,\calV\subseteq\Pw$ be upper sets.
		\begin{enumerate}[label=(\roman*)]
			\item \textbf{Rudin-Keisler order.} We define the relation
			$$\calU\lrk \calV$$
			just in case there is a function $h\colon\w\to\w$ such that 
			$$A\in\calU \quad \iff \quad h^{-1}[A]\in\calV \quad \text{for any}\, A\subseteq\w.$$
			\item \textbf{Kat\v{e}tov order.} We define the relation 
			$$\calU\lk\calV$$
			just in case there is a function $h\colon\w\to\w$ such that 
			$$A\in\calU \quad \implies \quad h^{-1}[A]\in\calV \quad \text{for any}\, A\subseteq\w.$$
		\end{enumerate}
	\end{definition}	
	
Informally, the upper sets $\calU,\calV$ each defines an abstract notion of largeness. The Kat\v{e}tov order says: $\calU$ is simpler than $\calV$ just in case there is a witness $h$ such that every $\calU$-large set is also seen as $\calV$-large. Rudin-Keisler strengthens this condition: a set is $\calU$-large {\em iff} it is witnessed to be $\calV$-large.
	
	\begin{remark} Notice that Definition~\ref{def:rk-kat-order} is stated for upper sets, not just ultrafilters. Some follow-ups:
		\begin{enumerate}[label=(\roman*)]
			\item The broader definition should reassure the constructive mathematician: our results will apply not only to ultrafilters (which depend on choice) but also to upper sets (which do not).\footnote{In fact, the Kat\v{e}tov order was originally introduced as an order on non-principal filters (not necessarily maximal) in order to study convergence in topological spaces \cite{Kat68}.} 
			\item As stated, the Kat\v{e}tov order is {\em a priori} coarser than Rudin-Keisler since it only requires forward implication. Nonetheless, the two orders coincide when restricted to the ultrafilters: $\calU\lrk \calV \iff \calU\lk \calV$ whenever $\calU,\calV$ are ultrafilters, see e.g. \cite[Lemma 2.1.7]{FlasPhD}.
			\item The same definition defines $\lrk$ and $\lk$ for lower sets.
		\end{enumerate}	
	\end{remark}

	With these definitions in place, questions about ultrafilters now translate into structural questions about the partial order. How many equivalence classes are there? Is it linear? Well-ordered? etc. One might hope the $\lrk$-order offers a systematic framework for separating major differences between ultrafilters from the minor ones. Unfortunately, as Blass writes, this is too optimistic: 
	\medskip
	
	\begin{fquote}[A. Blass \cite{Bl85}]
		Although much is known about the Rudin-Keisler ordering, [...] most of this information supports the view that [it is] \textbf{extremely complex} and it is not reasonable to try to classify ultrafilters up to isomorphism [i.e. up to $\erk$-equivalence]. \textbf{It is natural, in this situation, to consider coarser equivalence relations than isomorphism.}\footnote{The emphasis in the quoted passage is ours.}
	\end{fquote}
	
One major source of this complexity comes from Fubini products. To see why, we recall the standard definition for both upper and lower sets.
	
	\begin{definition}[Fubini Products]\label{def:fubini-prod} Given any $A\subseteq \w^2$, define its {\em $m$-section} $A_{(m)}:=\{l\in \w\mid  (m,l) \in A\}$. 
		\begin{enumerate}[label=(\roman*)]
			\item \textbf{Upper Sets.} For upper sets $\calU,\calV\subseteq \Pw$, their {\em Fubini product} is defined
			$$A\in \calU\otimes\calV\iff  \{m\in \w\mid A_{(m)}\in\mathcal{V}\}\in\mathcal{U}.$$
			Iterating this construction yields $$\calU^{\otimes n}:=\underbrace{\calU\otimes \dots \otimes \calU}_{n\;\text{times}}\,,$$ the {\em $n$th Fubini power} of $\calU$. By convention, we set $\calU^{\otimes 1}:=\calU$.
 \item \textbf{Lower Sets.} For lower sets $\calH,\calI \subseteq \Pw$, their {\em Fubini product} is defined
 \[A\in\calH\otimes\calI \iff \{m\in\omega\mid A_{(m)}\not\in\calI \}\in\calH.\]
		\end{enumerate}
Notice that, e.g. $\calU\otimes \calV$ can always be regarded as an upper set on $\w$ via a fixed bijection $\w^2\simeq \w$ (Convention~\ref{con:base-change}).        
	\end{definition}

\begin{remark}\label{rem:fubini-consistent} The Fubini product is typically defined for filters/ideals, so it is worth verifying that the same definitions also extend consistently to upper/lower sets. Given lower sets $\calH,\calI\subseteq\Pw$, 
	$$A\in\calH\otimes\calI \Leftrightarrow \{ n\mid A_{(n)} \notin \calI \}\in\calH \Leftrightarrow \{ n\mid A_{(n)} \in \calI \}\in\calH^\ast \Leftrightarrow \{ n\mid (\omega\setminus A)_{(n)} \in \calI^\ast\}\in\calH^\ast.$$
Hence, $A\in (\calI\otimes\calI)^\ast \Leftrightarrow \{n \mid A_{(n)}\in\calI^\ast\}\in\calI^\ast$, and so $(\calI^*\otimes\calI^*)=(\calI\otimes\calI)^\ast$. 

\end{remark}

The following fact delivers the punchline:
	
	\begin{fact}[{{\cite[Prop. 1.6.10]{Gold22}}}]\label{fact:rk-max} Given any pair of non-principal ultrafilters $\calU,\calV$, their Fubini product $\calU\otimes\calV$ defines an ultrafilter such that 
		$$\calU<_{\mathrm{RK}} \calU\otimes \calV \qquad \text{and} \qquad \calV<_{\mathrm{RK}} \calU\otimes \calV.$$
	\end{fact}
	
Two implications stand out. One, there cannot exist an $\lrk$-maximal non-principal ultrafilter $\calU$, since we always have $\calU <_{\mathrm{RK}} \calU \otimes \calU$. More fundamentally, Fubini products generate a great deal of combinatorial noise within the Rudin-Keisler Order. In view of Blass' remark, one may wonder if there exists a framework where this noise is muted whilst still preserving meaningful distinctions. This sets the stage for our Gamified Kat\v{e}tov order, which we turn to next.


\subsection{What is the Gamified Kat\v{e}tov order?}\label{sec:GTK} We begin with a simple reformulation of the Kat\v{e}tov order; this will point to the natural direction for our generalisation. First, a basic observation:

\begin{observation}\label{obs:katetov-upset-B} Let $\calU,\calV\subseteq\Pw$ be upper sets. 
	\begin{enumerate}[label=(\roman*)]
		\item $\calU\lrk \calV $ just in case there exists a function $h\colon\w\to\w$ such that
			$$A\in\calU \quad \iff \quad h[B]\subseteq A \,\, \text{for some}\, B\in\calV.$$	\item $\calU\lk \calV $ just in case there exists a function $h\colon\w\to\w$ such that
			$$A\in\calU \quad \implies \quad h[B]\subseteq A \,\, \text{for some}\, B\in\calV.$$
	\end{enumerate}
\end{observation}
\begin{proof} Obvious: since $\calV$ is an upper set, $h^{-1}[A]\in\calV$ iff $h[B]\subseteq A$ for some $B\in\calV$. 
\end{proof}


The key idea is to use Observation \ref{obs:katetov-upset-B} to interpret Kat\v{e}tov reducibility as a game. Concretely, fix upper sets $\calU,\calV$ and consider the following imperfect information two-player game:
\[
\begin{array}{rccccc}
{\rm I}\colon	& A\in\mathcal{U}	&		& x\in B	&		\\
{\rm II}\colon	&		& B\in\mathcal{V}	&		& y\in \w	
\end{array}
\]
\textbf{Rules of the game.} Call any function $h\colon\w\to\w$ a {\em strategy} of Player II. At the start of the game, Player II fixes some strategy $h$, which is then challenged by Player I. 
\begin{itemize}
	\item \textbf{Round 1a.} Player I selects some $A\in\calU$.
	\item \textbf{Round 1b.} Player II queries the oracle $\calV$, which returns $B\in\calV$.
	\item \textbf{Round 2a.} In response, Player I picks some $x\in B$. 
	\item \textbf{Round 2b.} Player II outputs $y$, whereby $h(x)=y$, and the game terminates. 
\end{itemize}
{\em \underline{What do we mean by ``queries the oracle''}?} Notice Observation~\ref{obs:katetov-upset-B}(ii) only requires that
\[
h[B]\subseteq A \quad \text{for some } B\in\calV,
\]
rather than insisting that $B:=h^{-1}[A]$ as per the usual formulation of $\calU\lk\calV$. That is, for each $A\in\calU$, we only need the existence of some witnessing set $B\in\calV$, with no prescribed rule for how it is obtained. Regarding $\calV$ as an ``oracle'' gives a convenient way to package this: given any problem $A\in\calU$, the oracle $\calV$ responds with $B\in\calV$, which Player II then accepts as a black box in the remainder of the game. 
\smallskip 

\noindent {\bf  Winning Condition.} Player II's output is called  {\em valid} if $y\in A$. In particular, Player II has a \textbf{winning strategy} if there exists an $h$ that gives a valid answer to all challenges by Player I. Stated more plainly: given any $A\in\calU$, there exists a $B\in\calV$ such that $h[B]\subseteq A$. By Observation~\ref{obs:katetov-upset-B}, this is exactly the direct image characterisation of
$$\calU\lk\calV.$$ 


\medskip 
This raises a natural question: if $\lk$ corresponds to a one-query game, what partial order emerges from allowing Player II to make \emph{finitely many} queries instead? That is, imagine Players I and II generate a finite sequence
\[
{\small
\begin{array}{rccccccccc}
{\rm I}\colon	& A\in\mathcal{U}	&		& x_0\in B_0	&		 & x_1\in B_1 & \dots &                               & x_k\in B_k & \\
{\rm II}\colon	&		& B_0\in\mathcal{V}	&		& B_1\in\mathcal{V}	 & &  \dots & B_k\in\mathcal{V} &                   & y\in \w
\end{array}
}\,,\]
where Player II outputs a response $y\in\w$ based solely on $(x_0,\dots, x_k)$. What would it mean for II to have a winning strategy in this multi-round setting? To capture this formally, we pass from single elements to finite sequences, and from functions $f\colon\w\to\w$ to partial maps on $\ww$. We lay the foundations by first introducing the notion of $\calV$-branching.

\begin{definition}\label{def:game-tree} Let $\calV\subseteq\Pw$ be an upper set, and $T\subseteq\lww$ a tree. A node $\sigma\in  T$ is called {\em $\calV$-branching} if the set of immediate successors $\{n\in\w\mid \sigma\fr n\in T\}$ is in $\calV$. A tree $T$ is {\em $\calV$-branching} just in case: 
	\begin{enumerate}[label=(\alph*)]
		\item $T$ is non-empty; and
		\item For any $\sigma\in T$, $\sigma$ is $\calV$-branching.
	\end{enumerate}
We emphasise Condition (b) includes the case when $\sigma=\o$, the root of $T$.
\end{definition}

\begin{definition}[Gamified Kat\v{e}tov order]\label{def:GTK} Let $\calU,\calV\subseteq\Pw$ be upper sets. We define the relation
	$$\calU\glt\calV$$
	just in case there exists a partial continuous function (Convention~\ref{con:pcf}) $$\Phi\pcolon \ww\to\w$$
	such that, for any $A\in\calU$, there exists some $\calV$-branching tree $T$ such that $[T]\subseteq\dom(\Phi)$ and $\Phi([T])\subseteq A$, that is, $\Phi(p)\in A$ for any infinite path $p$ through $T$. We call $\glt$ the {\em Gamified Kat\v{e}tov order}.
\end{definition}

\begin{remark} One also obtains a Gamified Kat\v{e}tov order on lower sets by dualising. Namely, given lower sets $\calH,\calI \subseteq \Pw$, we write
	$$\calH\glt\calI$$  
just in case $\calH^*\glt \calI^*$ on the dual upper sets.
\end{remark}

At first glance, it may not be obvious that Definition~\ref{def:GTK} characterises the finite-query Kat\v{e}tov game just described. The connection will be made clear in Theorem~\ref{thm:game-GTK}, whose proof we defer to Section~\ref{sec:clt}. In the meantime, it may be helpful to keep the following informal picture in mind:

\begin{discussion}\label{dis:inf-GTK} In the classical Kat\v{e}tov game, Player II fixes a function $h\colon\w\to \w$, and uses it to respond to any challenge. In the gamified setting, Player II instead fixes a ``strategy'' $\varphi\pcolon\lww\to\omega$ that assigns an output to each possible sequence of Player I's moves, which determines a partial continuous function $\Phi\pcolon \ww\to\w$ (Convention \ref{con:pcf}). The $\calV$-branching tree $T$ encodes all admissible sequences of Player I's moves (relative to Player II's strategy). The condition $[T]\subseteq\dom(\Phi)$ ensures Player II can respond to every move sequence; the condition $\Phi(T)\subseteq A$ ensures that all such responses remain inside the chosen $\calU$-large set $A$.
\end{discussion}

\subsection{Interactions with Fubini Powers}\label{sec:fubini} To illuminate the core mechanism of the Gamified Kat\v{e}tov order, we examine its relationship with Fubini products. For this purpose, Kat\v{e}tov's original definition of the Fubini product is useful.

\smallskip

\begin{definition}[Kat\v{e}tov \cite{Kat68}]\label{def:Kat-sum}~
\begin{itemize}
\item For $A\subseteq I$ and a family $(B_a)_{a\in A}$ of subsets of $J$, their set-theoretic sum is defined as:
\[\sum_{a\in A}B_
a=\{(a,x)\in I\times J\mid a\in A\mbox{ and }x\in B_a\}.\]
\item For sets $\U\subseteq \mathcal{P}(I)$ and $\V\subseteq \mathcal{P}(J)$, their concatenation $\U\ast\V\subseteq \mathcal{P}(I\times J)$ is defined as: 
\[\U\ast \V=\left\{\sum_{a\in A}B_a\ \middle|\ A\in\U\mbox{ and }(\forall a\in A)\ B_a\in\V\right\}.\]

Iterating this construction yields
\[\calU^{\ast n}:=\underbrace{\calU\ast \dots \ast \calU}_{n\;\text{times}}.\]
\end{itemize}
\end{definition}

If $\U,\V\subseteq\mathcal{P}(\omega)$, then $\U\ast\V\subseteq\mathcal{P}(\omega^2)$.
Of course, via the canonical bijection $\langle-,-\rangle\colon\omega^2\simeq\omega$, this may also be identified with a subset of $\mathcal{P}(\omega)$. In particular, Kat\v{e}tov introduces the Fubini product of $\U$ and $\V$ as $(\U\ast\V)^\uparrow$, which we justify below:

\begin{lemma}\label{lem:Katetov-Fubini-product}
For upper sets $\U,\V\subseteq\mathcal{P}(\omega)$, $(\U\otimes\V)=(\U\ast\V)^\uparrow$.
\end{lemma}

\begin{proof}
For $\subseteq$:
If $B\in\U\otimes\V$, put $A=\{a\in\omega\mid B_{(a)}\in\V\}$.
Then, by the definition of the Fubini product, we have $A\in\U$, so $\sum_{a\in A}B_{(a)}\in\U\ast\V$.
We also have $\sum_{a\in A}B_{(a)}\subseteq B$; hence $B\in(\U\ast\V)^\uparrow$.

For $\supseteq$:
If $C\in(\U\ast\V)^\uparrow$ then $\sum_{a\in A}B_a\subseteq C$ for some $A\in\U$ and $B_a\in\V$.
Clearly, $B_a\subseteq C_{(a)}$ for any $a\in A$; hence $A\subseteq\{a\in\omega\mid C_{(a)}\in\mathcal{V}\}\in\mathcal{U}$, which means $C\in\U\otimes\V$.
\end{proof}

Read in the context of our paper, Lemma~\ref{lem:Katetov-Fubini-product} motivates closer examination of how concatenation and upward closure interact. Applying Observation \ref{obs:katetov-upset-B} (ii), we start by extending the definition of the Kat\v{e}tov order to subset families in general:
\begin{itemize}
\item For $\U,\V\subseteq\mathcal{P}(\om)$, we write $\U\preceq_{\rm K}\V$ if there exists $\varphi\colon\om\to\om$ such that for any $A\in\U$ there exists $B\in\V$ with $\varphi[B]\subseteq A$.
We write $\U\approx_{\rm K}\V$ if $\U\preceq_{\rm K}\V$ and $\V\preceq_{\rm K}\U$.
\end{itemize}
In particular, Observation \ref{obs:katetov-upset-B} (ii) states that $\U\leq_{\rm K}\V$ iff $\U\preceq_{\rm K}\V$ for any upper sets $\U,\V\subseteq\mathcal{P}(\omega)$.

\begin{lemma}\label{lem:upward-closure-basic-properties}
Let $\U,\V\subseteq\mathcal{P}(\om)$ be subset families, not necessarily upper sets.
\begin{enumerate}[label=(\roman*)]
\item $(\U\ast \V)^{\uparrow}=(\U^\uparrow\ast\V^\uparrow)^\uparrow$.
\item $\U^\uparrow\approx_{\rm K}\U$.
\end{enumerate}
\end{lemma}

\begin{proof}~
(i): For $\subseteq$: Clearly $\U\ast\V\subseteq\U^\uparrow\ast\V^\uparrow$, so $(\U\ast \V)^{\uparrow}\subseteq(\U^\uparrow\ast\V^\uparrow)^\uparrow$.

For $\supseteq$: If $A\in (\U^\uparrow\ast\V^\uparrow)^\uparrow$ then $A^-\subseteq A$ for some $A^-\in\U^\uparrow\ast\V^\uparrow$.
Then, $A^-$ is of the form $\sum_{i\in B}C_i$ for some $B\in\U^\uparrow$ and $C_i\in\V^\uparrow$ for $i\in B$.
Then, there are $B^-,C^-$ such that $B\supseteq B^-\in\U$ and $C\supseteq C^-\in\V$.
This implies $$ \underbrace{\sum_{i\in B^-}C_i^-}_{\in\,\U\ast\V}\subseteq\sum_{i\in B}C_i=A^-\subseteq A.$$
In other words, $A\in(\U\ast \V)^\uparrow$.

(ii): For $\succeq_{\sf K}$: Any $A\in\U$ is also in $\U^\uparrow$, and ${\rm id}[A]=A$.

For $\preceq_{\sf K}$: By definition, for any $A\in\U^\uparrow$ there exists $A^-\in\U$ with $A^-\subseteq A$.
\end{proof}



We now explain the connection to the Gamified Kat\v{e}tov order. As a warm-up, let $\calU,\calV,\calW\subseteq\Pw$ be upper sets, and consider the following {\bf 2-query Kat\v{e}tov game}:
\[
\begin{array}{rccccccc}
{\rm I}\colon	& A\in\mathcal{U}	&		& x_0\in B_0	&		& x_1\in B_1 & & \\
{\rm II}\colon	&		& B_0\in\mathcal{V}	&		& B_1\in\mathcal{W} & & y\in \w	
\end{array}
\]
In our language, Player II queries two oracles in sequence: $\calV$ first, then $\calW$. As before, Player I chooses a target $A\in\calU$, as well as elements $x_0\in B_0, x_1\in B_1.$ Player II must then produce a valid response $y\in A$. A strategy for Player II is determined by functions 
$$\eta_0(A)=B_0, \quad \eta_1(A,x_0)=B_1,\quad h(x_0,x_1)=y,$$
where Player II's final move depends only on $(x_0,x_1)$; the intermediate queries may depend on Player I's entire play history before. Player II has a winning strategy if $y\in A$ for every choice of $(A,x_0,x_1)$.

\begin{proposition}\label{prop:fubini-GTK} Let $\calU,\calV,\calW\subseteq\Pw$ be upper sets. Then $$\calU\lk\calV\otimes \calW$$
iff Player II has a winning strategy in the 2-query Kat\v{e}tov game. 
\end{proposition}
\begin{proof}
By Observation \ref{obs:katetov-upset-B} (ii) and Lemmas \ref{lem:Katetov-Fubini-product} and \ref{lem:upward-closure-basic-properties} (ii), $\U\lk\V\otimes\W$ iff $\calU\preceq_{\rm K}\calV\ast \calW$.
The latter is witnessed by some $h$ iff for any $A\in\U$ there exists $B\in\V\ast\calW$ such that $h[B]\subseteq A$.
Such a $B$ is of the form $\sum_{i\in I}B_i$ for some $I\in\V$ and $B_i\in\calW$ for $i\in I$; hence $h[B]\subseteq A$ means that for any $i\in I$ and $b\in B_i$ we have $h(i,b)\in A$.
This condition can be expressed as the following game:
\[
\begin{array}{rccccccc}
{\rm I}\colon	& A\in\mathcal{U}	&		& i\in I	&		& b\in B_i & & \\
{\rm II}\colon	&		& I\in\mathcal{V}	&		& B_i\in\mathcal{W} & & h(i,b)\in A	
\end{array}
\]

\begin{itemize}
\item If $h$ witnesses $\U\preceq_{\rm K}\V\ast\calW$ then put $\eta_0(A)=I$ and $\eta_1(A,i)=B_i$ for any $i\in I$.
Then $(\eta_0,\eta_1,h)$ yields a winning strategy for Player II as above.
\item If Player II has a winning strategy $(\eta_0,\eta_1,h)$ then, for any $A\in\U$, we have $B:=\sum_{i\in\eta_0(A)}\eta_1(A,i)\in\V\ast\calW$ and $h[B]\subseteq A$; hence $h$ witnesses $\U\preceq_{\rm K}\V\ast\calW$.
\end{itemize}
%
\end{proof}

Proposition~\ref{prop:fubini-GTK} justifies viewing the $n$th Fubini product as defining an \textbf{$n$-query Kat\v{e}tov game}. Explicitly, 
$$\calU\lk \calV_1\otimes \dots \otimes \calV_n$$ 
holds iff Player II has a winning strategy in the game 	
$$	\begin{array}{rccccccccc}
{\rm I}\colon	& A\in\mathcal{U}	&		& x_1\in B_1	&		 & x_2\in B_2 & \dots &                               & x_n\in B_n & \\
{\rm II}\colon	&		& B_1\in\mathcal{V}_1	&		& B_2\in\mathcal{V}_2	 & &  \dots & B_n\in\mathcal{V}_n &                   & y\in A.
\end{array}$$
To connect these games with our original \textbf{tree-based} Definition \ref{def:GTK} of the Gamified Kat\v{e}tov order, it is worth highlighting the underlying tree structure of Fubini powers. Say that a tree $T\subseteq\lww$ of height $n$ is {\em $\U$-branching} if any node $\sigma\in T$ of length less than $n$ is $\U$-branching.

\begin{proposition}\label{prop:n-concatenation-branching-tree}
For $\U\subseteq\mathcal{P}(\omega)$ and $n\in\omega$, the $n$-concatenation $\U^{\ast n}\subseteq\mathcal{P}(\omega^n)$ is characterised as:
\[\U^{\ast n}=\{T\cap\omega^n\mid\mbox{$T$ is a $\U$-branching tree of height $n$}\}.\]
\end{proposition}

\begin{proof}
Inductively assume that the claim has been verified for $n$.
By definition, $L\in\U^{\ast(n+1)}$ iff there are $A\in\U$ and $L_a\in\U^{\ast(n)}$ for $a\in A$ such that
\[
L=\sum_{a\in A}L_a=\{(a,\sigma)\in\om\times\om^n\mid a\in A\mbox{ and }\sigma\in L_a\}.
\]

By the induction hypothesis, each $L_a$ is the set of length $n$ nodes of a $\U$-branching tree of height $n$.
By identifying $(a,\sigma)$ and $a\fr\sigma$, the set $L\subseteq\om^{n+1}$ can be thought of as the set of length $(n+1)$ nodes of a $\U$-branching tree of height $n+1$.

Conversely, given a $\U$-branching tree $T$ of height $n+1$, put $A=\{a\in\omega\mid\langle a\rangle\in T\}$.
Then $T\cap\om^{n+1}=\sum_{a\in A}L_a$, where $L_a=\{\sigma\in\om^n\mid a\fr\sigma\in T\}$.
Since $T$ is $\U$-branching, we have $A\in\U$, and for each $a\in A$, $L_a$ is the set of length $n$ nodes of a $\U$-branching tree $T_a$. We know that $T_a\in\U^{\ast n}$ by the induction hypothesis. Hence, $T=\sum_{a\in A}T_a\in\U^{\ast(n+1)}$.
\end{proof}


\begin{proposition}\label{prop:n-Fubini-branching-tree}
Let $\U$ be an upper set.
The $n$th Fubini power $\U^{\otimes n}\subseteq\mathcal{P}(\omega^n)$ is characterised as:
\[\U^{\otimes n}=\{T\cap\omega^n\mid\mbox{$T$ is a $\U$-branching tree of height $n$}\}^\uparrow.\]
\end{proposition}

\begin{proof}
Inductively assume $\U^{\otimes n}=(\U^{\ast n})^\uparrow$.
Then, by the induction hypothesis, and Lemmas \ref{lem:Katetov-Fubini-product} and \ref{lem:upward-closure-basic-properties} (i), compute:
\[\U^{\otimes (n+1)}=(\U^{\ast n})^\uparrow\otimes\U=((\U^{\ast n})^\uparrow\ast\U)^\uparrow=(\U^{\ast n}\ast\U)^\uparrow=(\U^{\ast(n+1)})^\uparrow.\]
By Proposition \ref{prop:n-concatenation-branching-tree}, this means that $\U^{\otimes(n+1)}$ is the upward closure of the set of all $\U$-branching trees of height $n+1$.
\end{proof}

 If $n$-query games correspond to the $n$th Fubini power, then finite-query games ought to correspond to well-founded iterations of the Fubini product. Fubini iterates have been studied before, primarily in the context of ultrafilters and ideals \cite{Rato12,Dob20}. However, the same basic idea extends to subset families: 
\begin{enumerate}[label=(\roman*)]
\item {\em Limit construction.} Consider a subset family $\U\subseteq\mathcal{P}(I)$
and a sequence of subset families $\{\V_n\}_{n\in I}$
where $\V_n\subseteq \mathcal{P}(J_n)$. Then define the following limit notion:
\[
A\in\lim_{n\to\U}\V_n
\iff
\{m\in I\mid A_{(m)}\in\mathcal{V}_m\}\in\mathcal{U}\ ,
\]
for any $A\subseteq \sum_{n\in I} J_n=\{(n,j) \mid n\in I \, , \, j\in J_n \}$.\footnote{Notice: if $\calV_n=\calV$ is the constant sequence, then this recovers the concatenation $\calU\ast \calV$ from Definition~\ref{def:Kat-sum}.} 

\smallskip

\item {\em Iterated powers.} Let $\alpha=\sup_{n\in\om}\beta_n$ be a countable limit ordinal. Assume inductively that $\U^{\otimes \beta_n}$ have been defined. Then, set
\[\U^{\otimes \alpha}:=\lim_{n\to\U}\U^{\otimes\beta_n}.\]
\end{enumerate}

\begin{discussion}\label{dis:iterate-fubini} 
If $\U,\V_n\subseteq\mathcal{P}(\omega)$ then, via $\om^2\simeq\om$, we may think of $\lim_{n\to\U}\V_n$ as a subset of $\mathcal{P}(\omega)$ rather than $\mathcal{P}(\omega^2)$. Thus, for any countable limit ordinal $\alpha$, we may regard $ \U^{\otimes\alpha}\subseteq\mathcal{P}(\omega)$. However, this presentation introduces some structural ambiguity since it depends on a choice of sequence $\{\beta_n\}_{n\in\w}$ converging to $\alpha$.

We resolve this by leveraging the underlying tree structure of Fubini powers, already highlighted in Proposition~\ref{prop:n-Fubini-branching-tree}. In particular, we introduce a combinatorial device (what we call a ``fence'') that records the critical positions of the strings within the tree. Concretely:
\begin{itemize}
\item As a warm-up, notice the $n$th Fubini power $\calU^{\otimes n}$ admits a presentation as a subfamily of $\calP(\w^n)$; this is from Proposition \ref{prop:n-Fubini-branching-tree}. In which case, we let the fence be $\w^n$. Notice $\omega^n$ defines an antichain when viewed as $\omega^n\subseteq\lww$.
\item 
More generally, consider the limit construction above $\lim_{n\to \calU}\V_n$. Suppose each subset family $\V_n$ admits a presentation as a subfamily of $\mathcal{P}(\delta_n)$ for some antichain $\delta_n\subseteq\lww$. Then, for $\calU\subseteq\Pw$, a natural fence for
$\lim_{n\to\U}\V_n$ is
\[
\delta \;=\; \{\, n\fr\sigma \mid n\in\omega \text{ and } \sigma\in\delta_n \,\} \ .
\]
In particular, $\lim_{n\to\U}\V_n$ can be presented as a subfamily of
$\mathcal{P}(\delta)$. Moreover, $\delta$ is again an antichain whenever all
$\delta_n$ are antichains.
\end{itemize}
With this perspective in place, we now make our next move.
\end{discussion}

In the context of the Gamified Kat\v{e}tov order, we wish to specify the {\em admissible range of queries} available to our Player II. Towards this end, we introduce the combinatorial notion of a {$\delta$-Fubini power.}

	\begin{definition}[$\delta$-Fubini powers]\label{def:delta-Fubini}
	Fixing a (non-trivial) upper set $\calU\subseteq\Pw$:
		\begin{enumerate}[label=(\roman*)]
			\item  Given an antichain $\delta$, a tree $T\subseteq \lww$ is called {\em $\calU$-branching up to $\delta$} if:
			\begin{enumerate}
				\item $T$ is non-empty;
				\item For any infinite path $x\in[T]$, there exists $n\in\w$ such that $x\rstr n\in\delta$. 
				\item For any $\sigma\in T$, either $\sigma\in\delta$ or $\{n\in\w\mid \sigma\fr n\in T\}\in \calU$. 
				Informally: either $\sigma$ has permission from $\delta$ to terminate, or its set of immediate successors has to be $\calU$-large.\footnote{Here ``permission to terminate'' does not mean that such nodes always have no successors, only that they are exempt from the $\calU$-branching requirement.}
			\end{enumerate}
		\item A {\em $\calU$-fence} is a set $\delta\subseteq\lww$ such that 
		\begin{enumerate}[label=(\alph*)]
			\setcounter{enumii}{3} 
			\item $\delta$ is an antichain.
			\item There exists a $\calU$-branching tree up to $\delta$. 
		\end{enumerate}
			\item For any $\calU$-fence $\delta$, define the \emph{$\delta$-Fubini power} of $\calU$ as
			\[
			\calU^{\otimes [\delta]} \;:=\Big\{\,\code{T\cap\delta} \Bmid T \text{ is }\calU\text{-branching up to }\delta \,\Big\}^{\up},
			\]
			where $\code{-}$ is from Convention \ref{con:base-change} and $\{\text{---}\}^{\uparrow}$ denotes upward closure taken in $\Pw$.
		\end{enumerate}	
		
	\end{definition}

\begin{remark} Various technical conditions are placed within Definition~\ref{def:delta-Fubini} to prevent trivialities.
	\begin{itemize}
		\item \emph{Condition (b).}  This ensures every $T$ contributing to $\Uod$ intersects our chosen $\delta$,\footnote{Why? Suppose $T\cap\delta=\emptyset$ for some tree $\calU$-branching up to $\delta$. Then, every  $\sigma\in T$ must have a $\calU$-large set of successors, so $T$ cannot have bounded height. Therefore $T$ admits an infinite path $x\in\ww$, which in turn must have a prefix in $\delta$, contradiction.} and so
            \[
			\calU^{\otimes [\delta]} \;=\Big\{\,\code{T\cap\delta} \Bmid T \text{ is well-founded and }\calU\text{-branching up to }\delta \,\Big\}^{\up}.
			\]
        Importantly, this implies $\emptyset\notin\Uod$ for any non-trivial upper set $\calU$, and so our construction does not trivialise. Otherwise $\emptyset\in\Uod$ forces the identity $\Uod=\Pw$, thereby washing out, rather than illuminating, the combinatorial structure of $\calU$. 
		\item \emph{Condition (ii).} In principle, $\delta$-Fubini powers can be defined with respect to arbitrary antichains
		However, we prefer to exclude degenerate cases where, at least in the opinion of $\calU$, $\delta$ is too ``small'', resulting in $\Uod=\emptyset.$ 
	\end{itemize}

\end{remark}

\begin{remark}\label{rem:Fubini-generalise} Proposition~\ref{prop:n-Fubini-branching-tree} and Definition~\ref{def:delta-Fubini} present two slightly different perspectives on the usual $n$th Fubini power $\calU^{\otimes [n]}$. Both cases start with the same subfamily presentation
\[ \{T\cap\omega^n\mid\mbox{$T$ is a $\U$-branching tree of height $n$}\} \subseteq \calP(\w^n).\]
The difference lies in how this subfamily is transported to an upper set on $\w$.\footnote{ Recall: Proposition~\ref{prop:n-Fubini-branching-tree} first takes the upward closure in $\calP(\w^n)$, before identifying $\w^n$ with $\w$ via a fixed bijection $\w^n\cong \w$. By contrast, the $\delta$-Fubini framework views $\w^n\subseteq\lww$ as an antichain of finite strings, and embeds it into $\omega$ via the fixed bijection $\lww\cong\w$, before taking the upward closure in $\Pw$. } Nonetheless, the two approaches are Kat\v{e}tov equivalent (in fact, computably so, since the identifications involved are computable). In this sense, the $\delta$-Fubini power construction recovers the usual $n$th Fubini power.
\end{remark}

\begin{remark}\label{rem:dob} The set theorist may wish to compare Definition~\ref{def:delta-Fubini} with the Ramsey-theoretic construction of generating ultrafilters by trees on a front, as introduced by Dobrinen in \cite[\S 3]{Dob20}. The two constructions are closely related, except Dobrinen's definition imposes some conditions that are {\em a priori} stronger than ours.\footnote{{\em Details.} In Dobrinen's context, a {\em front} is a family $\calB$ of finite subsets of $\w$ satisfying the conditions: (1) $a\not\subseteq b$ for any $a\neq b$ in $\calB$; (2) the base $M:=\bigcup \calB$ is infinite; and (3) any infinite $X\subseteq M$ has an initial segment in $\calB$. This plays a role analogous to our notion of a $\calU$-fence $\delta$, except that Dobrinen's Condition~(3) looks to be stronger than what we require. In our case, rather than demanding that \emph{every} infinite subset of $M$ (or every infinite string using the underlying base) has an initial prefix in $\delta$, we only require that a fence $\delta$ intersects with all infinite paths within {\em some} $\calU$-branching tree $T$.} 
\end{remark}

\smallskip

Reviewing our context once more, Proposition~\ref{prop:fubini-GTK} re-interprets the relation
$$\calU\lk\calV^{\otimes n}$$
as an \textbf{$n$-query Kat\v{e}tov game}, in which Player~II wins by always making exactly $n$ queries to $\calV$. Importantly, this does \emph{not} yet capture the finite-query Kat\v{e}tov game itself. Requiring $\calU\lk\calV^{\otimes n}$ for
some fixed $n$ imposes a \emph{uniform} bound on the number of queries across
all $A\in\calU$, which is too restrictive.  In the actual game, the number of queries may vary with the $A$ selected by Player I — so long as a valid answer is found after finitely many queries, Player II has a winning strategy.

Our motivation behind introducing the \emph{$\delta$-Fubini power} construction was to capture this added flexibility. By Remark~\ref{rem:Fubini-generalise}, we know that the construction generalises the classical Fubini powers. We now show that this is in fact the \emph{correct} generalisation: $\delta$-Fubini powers characterise the Gamified Kat\v{e}tov order.

\begin{theorem}\label{thm:fubini-GTK} Let $\calU,\calV\subseteq \Pw$ be upper sets. \underline{Then}, 
	$$\calU\glt \calV \iff \text{there exists a $\calV$-fence}\,\delta \;\text{such that} \; \calU\lk \calV^{\otimes [\delta]}.$$
\end{theorem}
\begin{proof}
\hfill
\begin{itemize} 
	\item[$\implies$:]  By Definition~\ref{def:GTK}, $\calU\glt\calV$ is witnessed by a partial continuous function $\Phi\pcolon \ww\to\w$.
	\begin{itemize}
		\item[$\diamond$] {\it Antichain.}
		By Convention~\ref{con:pcf}, $\Phi$ is determined by a partial function $\varphi\pcolon\lww\to\w$ whose domain is an antichain $\delta:={\rm dom}(\varphi)$ in $\lww$.
		\item[$\diamond$] {\it Kat\v{e}tov reduction.}
		Define a function $h\colon\w\to\w$ as the total extension of $\varphi$ in the following sense:
		$$h(n):=\begin{cases} \varphi(\sigma) \;\;\quad\qquad \text{if}\; n=\code{\sigma}\;\text{for some}\; \sigma\in\delta\\
		0 \;\qquad\;\;\;\quad\quad  \text{otherwise}
		\end{cases}\quad.$$
		Notice $h$ is well-defined since $\code{-}$ is an injection.
	\end{itemize}

\noindent We now establish the implication. To see why $h$ witnesses $\calU\lk\Vod$, suppose $A\in\calU$. By definition, there exists a $\calV$-branching tree $T$ such that $[T]\subseteq \dom(\Phi) $ and $\Phi[T]\subseteq A$. Since $T$ is clearly $\calV$-branching up to $\delta$, this gives $\code{T\cap\delta}\in\Vod$.
This also shows that $\delta$ is a $\calV$-fence.
Further, given any $n\in \code{T\cap\delta}$,  we have $h(n)=\varphi(\sigma)$ for some prefix $\sigma\in\delta$, and so 
$$h(\code{T\cap\delta})=\Phi[T]\subseteq A\;.$$ 
Since $\Vod$ is upper set, conclude from Observation~\ref{obs:katetov-upset-B} that
$$ \calU\lk \Vod\;.$$ 
	\item[$\impliedby$:] Assume there exists a $\calV$-fence $\delta$ and a function $h\colon\w\to\w$ witnessing  $\calU\lk\calV^{\otimes[\delta]}$. By Observation~\ref{obs:katetov-upset-B}, this means: for any $A\in\calU$, there exists $B\in\Vod$ such that $h[B]\subseteq A$. We proceed by extracting the relevant combinatorial data, before establishing the witnessing condition.
	
\subsubsection*{Step 1: $B$ induces a $\calV$-branching tree.} Since $B\in\Vod$, there exists a tree $\wT_B$ that is $\calV$-branching up to $\delta$ with
$$\code{\wT_B\cap \delta}\subseteq B.$$
To upgrade $\wT_B$ into an honest $\calV$-branching tree (not just up to $\delta$), define
$$T_B:=\wT_B\cup\{\sigma\fr\tau\bmid \sigma\in \wT_B\cap\delta\,\text{and}\, \tau\in\lww\}.$$ 
In English: all nodes in $\wT_B$ are $\calV$-branching except for a distinguished subset selected by $\delta$; the tree $T_B$ resolves this by turning all such $\sigma\in \wT_B\cap\delta$ into dummy $\w$-branching nodes. 

\subsubsection*{Step 2: $h$ induces a partial continuous function.}	Build the finite representative $\varphi\pcolon \lww\to\w$ as follows:
\begin{itemize}
	\item[$\diamond$] For each $\sigma\in\delta$, set 
	$$\varphi(\sigma):=h(\code{\sigma}).$$
%
\end{itemize}
\smallskip
\noindent By Convention~\ref{con:pcf}, $\varphi$ extends to a partial continuous function $\Phi_h\pcolon \ww\to \w$, constant on each cone $\{\tau\in\ww \mid \tau\succeq\sigma\}$ for each $\sigma\in \delta$.

\subsubsection*{Step 3: Finish.} Let $A\in\calU$. By Steps 1 and 2, we obtain a $\calV$-branching tree $T_B$ and a partial continuous function $\Phi_{h}\pcolon \ww\to\w$. In addition, by construction, any infinite path $x\in [T_B]$ extends some prefix in $\sigma\in  \wT_B\cap\delta$ ; in which case, $x$ belongs to the domain of $\Phi_h$ and $\Phi_h(x)=h(\code{\sigma})\in A$.
In other words, $[T_B]\subseteq \dom(\Phi_h)$ and  $\Phi_h[T_B]\subseteq A,$ and thus $$\calU\glt\calV.$$	
\end{itemize}	
\end{proof}


Theorem~\ref{thm:fubini-GTK} yields several important structural corollaries.

\begin{corollary}\label{cor:GTK-filters} The $\delta$-Fubini power of a filter is still a filter. Hence, the Gamified Kat\v{e}tov order has a well-defined restriction to filters over $\w$. 
\end{corollary}
\begin{proof} Let $\calF$ be a filter, and let $\delta$ be an $\calF$-fence. By construction, $\calF^{\otimes[\delta]}$ is an upper set. It remains to check closure under pairwise intersection. Suppose $X,Y\in \calF^{\otimes[\delta]}$. By definition, there exists subsets $A\subseteq X$ and $B\subseteq Y$ so that  $A=\code{T_A\cap\delta}$ and $B=\code{T_B\cap\delta}$, where $T_A,T_B$ are $\calF$-branching trees up to $\delta$. 

First observation: if we can show $A\cap B \in\calF^{\otimes[\delta]}$, then this implies $X\cap Y\in\calF^{\otimes[\delta]}$ since $A\cap B\subseteq X\cap Y$ and $\calF^{\otimes[\delta]}$ is upward closed. Second observation: since $\calF$ is a filter, this implies $T_{A\cap B}:=T_A\cap T_B$ is also $\calF$-branching up to $\delta$.\footnote{\label{fn:intersect-branch-tree}Why? Suppose $\sigma\in T_{A\cap B}$. If $\sigma\in\delta$, then there is nothing to check; otherwise, if $\sigma\notin\delta$, then $\sigma$ has two sets of $\calF$-large successors (from both $T_A$ and $T_B$), the intersection of which is also $\calF$-large since $\calF$ is a filter. } Since $\code{-}$ is an injection, deduce that $A\cap B= \code{T_A\cap T_B\cap\delta} =\code{(T_{A\cap B})\cap\delta}\in\calF^{\otimes[\delta]}$, as desired. 	
\end{proof}

\begin{corollary}\label{cor:princ-ultra}  Any $\delta$-Fubini power of a principal ultrafilter is still a principal ultrafilter. Hence, for any principal ultrafilter $\calF$ and any non-principal ultrafilter $\calG$,
	$$\calF <^\circ_{\mathrm{LT}} \calG\;.$$
\end{corollary}
\begin{proof} The argument proceeds by unpacking Definition~\ref{def:delta-Fubini}. Let $\calF$ be a principal ultrafilter whereby $\bigcap\calF=\{n\}$, and let $\delta$ be an $\calF$-fence. By definition, there exists a tree $T$ that is $\calF$-branching up to $\delta$. The set of successors of any non-terminating node in $T$ is $\calF$-large, and thus contains $n$ by principality. Hence, the infinite path  $p:=(n,n,n,n,\dots)$ has a prefix $p\rstr k\in\delta$ for some $k$.\footnote{Why? If $T$ has infinite height, then $p\in[T]$ and Condition~(b) of Definition~\ref{def:delta-Fubini} ensures that $p$ meets $\delta$. If $T$ has finite height, then there exists some $k$ such that $p\rstr k\in T$ terminating in $T$, which implies $p\rstr k\in\delta$ by Condition~(c).} 
	
In fact, since $\{n\}$ is $\calF$-large by principality, $p$ itself defines an $\calF$-branching tree up to $\delta$. Since $\delta$ is an antichain, this implies $\calF^{\otimes[\delta]}$ contains a singleton, namely $\code{p\cap\delta}$. By Corollary~\ref{cor:GTK-filters}, deduce that $\calF^{\otimes[\delta]}$ is a principal ultrafilter. In particular, given any {\em non-principal} ultrafilter $\calG$, we have that
$$\calG\not\lk \calF^{\otimes[\delta]}\quad.$$
 Combined with the fact that $\calF\lk\calG$, apply Theorem~\ref{thm:fubini-GTK} to conclude $\calF <^\circ_{\mathrm{LT}} \calG$. 
\end{proof}

\begin{corollary}\label{cor:coarse} Given any upper set $\calU\subseteq\Pw$, 
	$$\calU\elt \calU\otimes \calU.$$
In particular, the Gamified Kat\v{e}tov order is strictly coarser than both Kat\v{e}tov and Rudin-Keisler in ZF.
\end{corollary}
\begin{proof} This is easy once we know Remark~\ref{rem:Fubini-generalise} and Theorem~\ref{thm:fubini-GTK}. 
Two basic checks:
\begin{enumerate}
	\item Clearly $\calU\otimes\calU \glt\calU$, since $\calU\otimes\calU\lk\calU\otimes\calU = \calU^{\otimes 2}$.
	\item Clearly $\calU\glt\calU\otimes\calU$, since $\calU\lk\calU\otimes\calU = (\calU\otimes\calU)^{\otimes 1}$.
\end{enumerate}
Hence, by (1) and (2), we get $\calU\elt\calU\otimes\calU.$ The remark about strict coarseness follows either from Fact~\ref{fact:rk-max} (assuming Choice) or from the ZF result $\Fin <_{\rm K} \Fin\otimes\Fin$ \cite{GGMA16}.

\end{proof}


Corollary~\ref{cor:coarse} gives a principled explanation for the coarseness of the Gamified Kat\v{e}tov order relative to Rudin-Keisler: $\glt$ is uniformly less sensitive to Fubini products. This is especially suggestive in light of Blass' remarks discussed in Section~\ref{sec:ultf}. The extent of this coarseness will be examined more closely in Section~\ref{sec:separation}; by Corollary~\ref{cor:princ-ultra}, we know $\glt$ is at least sufficiently sensitive to distinguish principal from non-principal ultrafilters. For now, our results mark an important conceptual checkpoint, summarised below: 


\begin{conclusion}\label{conc:fubini}
The $\delta$-Fubini power construction both generalises and refines the classical Fubini power. Since each fence $\delta\subseteq\lww$ corresponds to subset of {\em finite strings}, the Gamified Kat\v{e}tov order may be viewed as the classical Kat\v{e}tov order closed under well-founded iterations of Fubini powers. In particular, the Gamified Kat\v{e}tov order is strictly coarser than the classical Kat\v{e}tov order.
\end{conclusion}

\subsection{Finite-query Kat\v{e}tov game}\label{sec:game} In this section, we give an explicit game-theoretic description of the Gamified Kat\v{e}tov order.
\medskip

What's changed from before? Previously, when we described the $n$-query Kat\v{e}tov game, the number of queries was fixed beforehand: Player II always makes exactly $n$ queries. When this is relaxed to allow any {\em finite} number of queries, Player II must now {\em also} declare when to make a further query (thus adding another round to the game), and when to give a final answer (thus terminating the game). To describe this accurately, it is helpful to think of Player II as a team of two players, each delegated different responsibilities. We therefore describe the \textbf{Finite-Query Kat\v{e}tov Game} as a three-player game, spelled out below:

\begin{definition}[Finite-Query Kat\v{e}tov Game]\label{def:fin-quer-kat} Let $\calU,\calV\subseteq\Pw$ be two upper sets. The corresponding {\em Finite-Query Kat\v{e}tov Game}, denoted $\GUV$, is a game between Player I vs. Players II$_0$ and II$_1$. A typical play has the shape:
	\[
	\begin{array}{rccccccccccc}
	{\rm I}\colon	& A	&		& x_1	&		& x_2	&	& \dots & x_{k-1} & & x_{k} & \\
	{\rm II}_0\colon	&		& 0	&		& 0	& 		& 0	& \dots & & 0 & & 1,y\\
	{\rm II}_1\colon	&		& B_0	&		& B_1	& 		& B_2	& \dots & & B_{k-1} & &
	\end{array}
	\]
\textbf{Rules of the game.} Similar as before, Players II$_0$ and II$_1$ fix their strategy beforehand, which determines their next move based on the information available to them. Player I then challenges the strategy by issuing an additional problem at the start of each round, which must all be solved within finite time.
\begin{itemize}
	\item Player I begins the exchange by choosing some $A\in \calU$.
\item 	At the $n$th round:
\begin{itemize}
	\item[$\diamond$] Player II$_0$ announces a binary decision $q_n\in\{0,1\}$:
	\begin{itemize}
		\item If $q_n=0$, an additional query is made, and play continues; 
		\item If $q_n=1$, then Player II$_0$ declares termination of the game, and outputs a value $y\in\w$. The output is {\em valid} if $y\in A$.  
	\end{itemize}
The initial decision is always $q_0=0$. 
	\item[$\diamond$] Player II$_1$ is the oracle. If Player II$_0$ makes a query, Player II$_1$ chooses a set $B_n\in\calV$. Otherwise, the oracle stays silent.
\end{itemize} 
\item At the $(n+1)$th round, Player I then responds with $x_{n+1}\in B_n.$
\end{itemize}
\textbf{Player strategies.} To formalise each player's strategy, we first clarify what information is visible to them at each stage. At the $n$th round:
\begin{itemize}
	\item Player I can see all previous moves by Players II$_0$ and II$_1$.
	\item All previous subsets $(A,B_0,\dots, B_{n-1})$ are invisible to Player II$_0$;  the binary decision $q_n$ is based soley on Player I's previous responses $x_1,\dots, x_{n-1}$.
	\item Player II$_1$ observes the current position $(A,x_1,\dots,x_n)$ before selecting $B_n$.
\end{itemize}
To capture this formally, represent Player II$_0$'s strategy as a partial function $$\psi\pcolon\lww\to\w$$ satisfying $\psi(\o)=0$, and for each non-empty string $(x_1,\dots,x_n)$, either 
$$\psi(x_1,x_2,\dots,x_n)=0 \quad\text{or}\quad \psi(x_1,x_2,\dots,x_k)=\lranglet{1}{y}.$$ 
Without loss of generality, assume: if $\sigma\in\lww$ such that $\psi(\sigma)=\lranglet{1}{y}$, then $\psi(\tau)\uar$ for any extension $\tau\succeq\sigma$. Player II$_1$'s strategy is a partial function 
$$\eta\pcolon\Pw\times\lww\to\Pw$$
such that $\eta(A,\o)=B_0$ and $\eta(A,x_1,\dots,x_n)=B_n$ in the game $\GUV$. In particular, Player II$_0$ always responds with numbers, while Player II$_1$ responds with subsets in $\w$.
\smallskip

\noindent \textbf{Winning Condition.} Player II (the pair II$_0$ + II$_1$) {\em wins the game $\GUV$} if every play terminates legally with a valid output $y\in A$ (where $A$ is Player I's first move) {\em or} Player I violates the rules before Player II does. In which case, we say that Player II has {\em a winning strategy} for the game $\GUV$.
\end{definition}
	
We are finally in a position to properly justify our decision to call $\glt$ the {\em Gamified Kat\v{e}tov order}.

	\begin{theorem}\label{thm:game-GTK} Let $\calU,\calV\subseteq\Pw$ be upper sets. \underline{Then},
		$$\calU\glt\calV$$
		if and only if Player II has a winning strategy for the finite-query Kat\v{e}tov game $\GUV$.
	\end{theorem}	
	
\begin{proof} \hfill
	\begin{itemize}
		\item[$\impliedby$\;:] 
		Suppose Player~II has a winning strategy, consisting of a pair of partial functions
		\[ \psi\pcolon\lww\to\w
		\quad\text{and}\quad
		\tau\pcolon \Pw\times\lww\to\Pw \;,
		\]
		corresponding respectively to Players~II$_0$ and~II$_1$. In particular, $\psi$ and $\tau$ must obey the rules of the game $\GUV$. Leveraging this observation, we show $\calU\glt\calV$ by extracting the relevant data from the winning strategy, before establishing the witnessing condition.
		
	\subsubsection*{Step 1:  Constructing the partial continuous function.} Build a partial function
		$\varphi\pcolon\lww\to\w$ by setting:
		\[
		\varphi(\sigma)=y
		\quad\text{iff}\quad
		\psi(\sigma)= \lranglet{1}{y}
		\]
		If $\psi(\sigma)=0$, leave $\varphi(\sigma)$ undefined. Everywhere else, leave $\varphi$ undefined.
		\smallskip
		
		The domain of $\varphi$ is an antichain, since under the rules of $\GUV$, Player II$_0$ terminates the game as soon as a ``good'' string $\sigma$ is found. Examining our formalisation of player strategies, this means: once $\psi(\sigma)=\lranglet{1}{y}$, the strategy specifies no further values on extensions of $\sigma$. By Convention~\ref{con:pcf}, $\varphi$ thus determines a partial continuous function $\Phi\pcolon\ww\to\w$.
		
	\subsubsection*{Step 2: Constructing the $\calV$-branching tree.} Let $\tau\pcolon \Pw\times\lww\to\Pw$ be Player II$_1$'s strategy. Then, for any $A\in\calU$, we build a tree $T_A$ recursively by separating the ``active'' vs. ``lazy'' nodes:	
	\begin{itemize}
		\item[$\bullet$]  \textbf{Active.} A node $\sigma\in T_A$ is {\em active} if $\psi(\sigma)=0$; this includes the root $\sigma=\o$. In which case, define its set of immediate sucessors in $T_A$ to be
				\[
	\{n \mid \sigma\fr n \in T_A\} = \tau(A,\sigma).
		\]
	Any active node is $\calV$-branching: since Player II$_1$'s strategy is winning, $A\in\calU$ implies that $\tau(A,\sigma)\in\calV$.
		\item[$\bullet$] \textbf{Lazy.} All non-active nodes $\sigma\in T_A$ are called {\em lazy}. In which case, $\sigma$ is full-branching:
			\[
		\{n \mid \sigma\fr n \in T_A\} = \w\;.
		\]
		Clearly, all lazy nodes are $\calV$-branching as well.
	\end{itemize}
That $T_A$ is $\calV$-branching is thus immediate by construction.
\medskip

\noindent Here is the informal picture. Given $A\in \calU$, our tree $T_A$ corresponds to Player I's legal plays (with respect to Player II's fixed strategy). If $\psi(\sigma)=0$ (an active node), then a query is made to the oracle, who returns $\tau(A,\sigma)\subseteq \w$, the set of all legal next moves for Player I. If no query is made (a lazy node), the oracle stays silent, and no restrictions are placed on Player I. 

\subsubsection*{Step 3: Finish} Let $\Phi\pcolon\ww\to\w$ be the partial continuous function in Step 1, determined by its finite representative $\varphi\pcolon \lww\to\w$. 
\smallskip 

Fix any $A\in\calU$. By Step 2, we obtain a $\calV$-branching tree $T_A$ whose nodes represent all legal plays by Player I against Player II's fixed strategy. Now consider any infinite path $p\in [T_A]$. Since Player II's strategy is winning, Player II$_0$ must eventually declare termination: there exists $k\in \w$ such that $\psi(p\rstr k)=\lrangles{1,u}$ where $u\in A$. By construction of $\Phi$, we then have $\Phi(p) = \varphi (p\rstr k)=u$. Hence, 
$$p\in [T_A] \implies \Phi(p)\dar\quad\text{and}\quad \Phi(p)\in A.$$
Equivalently, $[T_A]\subseteq \dom(\Phi)$ and $\Phi[T_A]\subseteq A$. Since $A\in\calU$ was arbitrary, conclude that
$$\calU\glt\calV.$$
		\medskip
	\item[$\implies$\;:] Consider the above proof in reverse, with the obvious translations. 
	\end{itemize}
\end{proof}

The translation between trees and games is well-known in various fields, including set theory, so this should come as no surprise. Nevertheless, we still believe the result is worth stating as a theorem since the interplay between these two perspectives will be useful in our context.
	
\begin{conclusion}\label{conc:GTK} The $\glt$-order is a variant of the Kat\v{e}tov order admitting an explicit game-theoretic description. Hence, Definition~\ref{def:GTK} indeed defines a gamified Kat\v{e}tov order.
\end{conclusion}

\subsection{Comparison with $\clt$-order}\label{sec:clt}
As alluded to in Section~\ref{sec:Intro}, our original motivation behind the Gamified Kat\v{e}tov order came not from set theory, but from a longstanding open problem in category theory. We restate this below for convenience.

\begin{problem}\label{prob:LT} An LT topology in the Effective Topos $\Eff$ is an endomorphism 
	$$j\colon \Pw\to\Pw$$
validating all three conditions of Definition~\ref{def:LT-topology}. For any pair of LT topologies $j,k$, we write
	$$j\clt k : \iff \forall p.j(p)\to k(p) \, \text{is valid},$$	
where validity means the existence of a suitable computable witness (Convention~\ref{con:orders}). What is the structure of the $\clt$-order on LT-topologies, and what can this tell us about computable complexity? 
\end{problem}


Lee-van Oosten \cite{LvO13} and Kihara \cite{Kih23} provided {\em combinatorial} and {\em game-theoretic} descriptions of the LT-topologies in $\Eff$, respectively. Combining the two perspectives allows us to connect the Gamified Kat\v{e}tov order with the structure of LT topologies.

\begin{discussion}\label{desc:function-complexity-as-principal-filter}
Kihara's work \cite{Kih23} originates in studying the ``complexity of functions'' in computability theory.
To connect this with the Kat\v{e}tov order, we need a unifying approach that handles the complexity of both ``functions'' and ``upper sets.''
This can be achieved by the following identifications.
\begin{enumerate}[label=(\alph*)]
\item Identify a single-value $f(n)\in\omega$ with the {\em principal ultrafilter} $\{A\subseteq\omega:f(n)\in A\}\subseteq\mathcal{P}(\omega)$
\item Identify a multi-value $g(n)\subseteq\omega$ with the {\em principal filter} $\{A\subseteq\omega:g(n)\subseteq A\}\subseteq\mathcal{P}(\omega)$.
\end{enumerate}
This gives rise to the following slogan:
\begin{enumerate}[label=(\alph*)]
\item A partial function $f\pcolon\om\to\om$ is a partial $\omega$-sequence of principal ultrafilters on $\omega$.
\item A partial multi-valued function $g\pcolon\om\tto\om$ is a partial $\omega$-sequence of principal filters on $\omega$.
\end{enumerate}
We will justify these identifications below. Our guiding thesis is that the $\clt$-order on LT-topologies in $\Eff$ can be characterised by extending the Gamified Kat\v{e}tov order from (single) upper sets $\calU$ to partial $\omega$-sequences of upper sets over $\omega$.
\end{discussion}

\begin{definition}[Upper Sequences]\label{def:upper-seeq} An {\em upper sequence on $\w$} is a partial sequence $$\overline{\mathcal U}=(\mathcal U_n)_{n\in I}$$ of upper sets over $\omega$, where $I\subseteq\omega$. If $I=\w$, then $\calUS$ is called {\em total}.
\end{definition}

\begin{example}
Any upper set $\calU\subseteq\Pw$ can be viewed either as a constant upper sequence (where $\calU_n=\calU$), or an upper sequence with singleton index set $I=\{\ast\}$.
As mentioned above, any partial multi-valued function $g\colon\om\tto\omega$ can be identified with a upper sequence consisting of principal filters.
\end{example}

\begin{definition}\label{def:Katetov-for-upper-seq}
Let $\overline{\U}=(\U_i)_{i\in I}$ and $\overline{\V}=(\V_j)_{j\in J}$ be partial $\omega$-sequences of subset families, i.e., $I,J\subseteq\omega$ and $\U_i,\V_j\subseteq\mathcal{P}(\omega)$.
\begin{itemize}
\item We say that $\overline{\U}$ is {\em Kat\v{e}tov reducible to $\overline{\V}$} if there exist $\varphi$ and $\psi$ such that for any $i\in I$ and $A\in\U_i$ there exists $B\in\V_{\varphi(i)}$ such that $\psi[B]\subseteq A$.
\item We say that $\overline{\U}$ is {\em one-query reducible to $\overline{\V}$} if there exist $\varphi$ and $\psi$ such that for any $i\in I$ and $A\in\U_i$ there exists $B\in\V_{\varphi(i)}$ such that $\psi(i,x)\in A$ for any $x\in B$.
\item If $\varphi$ and $\psi$ are computable, we say that $\overline{\U}$ is {\em computable Kat\v{e}tov/one-query reducible to $\overline{\V}$}.
\end{itemize}
Obviously:
\begin{center}
$\overline{\U}$ is Kat\v{e}tov reducible to $\overline{\V}$ $\implies$ $\overline{\U}$ is one-query reducible to $\overline{\V}$.
\end{center}
\end{definition}

\begin{observation}\label{obs:computable-Katetov-to-upseq}
Any partial $\omega$-sequences $\overline{\U}=(\U_i)_{i\in I}$ of subset families is computable Kat\v{e}tov equivalent to a upper sequence.
\end{observation}

\begin{proof}
Consider $\overline{\U}^\uparrow=(\U_i^\uparrow)_{i\in I}$.
As in Lemma \ref{lem:upward-closure-basic-properties} (ii), one can easily see that $\overline{\U}^\uparrow$ is computable Kat\v{e}tov equivalent to $\overline{\U}$ via $\varphi=\psi={\rm id}$.
\end{proof}

Observation \ref{obs:computable-Katetov-to-upseq} justifies Discussion \ref{desc:function-complexity-as-principal-filter}.

\begin{remark}
Definition \ref{def:Katetov-for-upper-seq} unifies combinatorial complexity and computability-theoretic complexity:
\begin{itemize}
\item Consider constant sequences $\overline{\U}=(\U)$ and $\overline{\V}=(\V)$.
Then:
\[\mbox{$\overline{\U}$ is Kat\v{e}tov reducible to $\overline{\V}$} \iff \mbox{$\overline{\U}$ is one-query reducible to $\overline{\V}$} \iff \U\preceq_{\rm K}\V.\]
\item Consider partial multi-valued functions $f$ and $g$.
Then, via the identification in Discussion \ref{desc:function-complexity-as-principal-filter}:
\[\mbox{$f$ is computable one-query reducible to $g$} \iff \mbox{$f$ is Weihrauch reducible to $g$}.\]
We emphasise here that {\em Weihrauch reducibility} plays a central role in modern computability theory; see the recent survey \cite{BGP21}.
\item In computability theory, the notion of computable one-query reducibility (for upper sequences) is also known as {\em extended Weihrauch reducibility} \cite{Bau22,Kih23,Kih22}.
\end{itemize}
\end{remark}


To proceed, we give a quick review of the first author's game-theoretic characterisation of LT topologies in the Effective Topos $\Eff$ \cite{Kih23}.
A key idea in \cite{Kih23} is to treat each index $i$ of $\overline{\U}=(\U_i)_{i\in I}$ as {\em public} data, while treating each $A\in\U_i$ as {\em secret} data.
This yields the following imperfect information game.

\begin{definition}[LT-Game]\label{def:bilayer-game} Let $\overline{\U}=(\U_i)_{i\in I}$ and $\overline{\V}=(\V_j)_{j\in J}$ be upper sequences.
The corresponding {\em LT-Game}, denoted $\mathfrak{G}(\overline{\U},\overline{\V})$, is an imperfect information three-player game. The players are named $\mer, \art, \nim$, and a typical play has the shape:
		\[
	\begin{array}{rccccccccccc}
\mer \colon	& x_0,A_0	&		& x_1	&		& x_2	&	& \dots & x_{k-1} & & x_{k} & \\
\art\colon	&		& y_0	&		& y_1& 		&  y_2	& \dots & & y_{k-1} & & y_k\\
\nim\colon	&		& B_0	&		& B_1	& 		& B_2	& \dots & & B_{k-1} & &
	\end{array}
	\]
\textbf{Rules of the game.} $\mer$ starts by choosing $x_0\in I$ and $A_0\in\U_{x_0}$. At the $n$th round:
\begin{itemize}
	\item $\art$ reacts with $y_n=\lranglet{j}{u_n}$:
	\begin{itemize}
		\item[$\diamond$] The choice $j=0$ indicates $\art$ makes a new query $u_n$ to $\overline{\V}$; we require $u_n\in J$.
		\item[$\diamond$] The choice $j=1$ indicates $\art$ delcares termination of the game with $u_n$.
	\end{itemize}
\item  $\nim$ makes an advice parameter $B_n\in \V_{u_n}$.
\end{itemize}
After which, at the $(n+1)$th round, $\mer$ responds by selecting $x_{n+1}\in B_n$. We say that $\art$-$\nim$ {\em win the game $\mathfrak{G}(\overline{\U},\overline{\V})$} if every play terminates legally with $u_n\in A_0$, or $\mer$ violates the rules before $\art$ and $\nim$ do.

\noindent \textbf{Player strategies.} Whereas $\mer$ and $\nim$ can see the play history of all previous moves, $\art$ can only see the (public) moves $x_0, x_1, x_2\dots $ made by $\mer$. In particular, $\art$ is unable to see the parameters selected by $\nim$ (``secret input''). Moreover, we shall require that $\art$'s moves be chosen computably. Formally, $\art$'s strategy becomes a partial {\em computable} function
$$\varphi\pcolon \lww \to \w.$$
On the other hand, $\mer$ and $\nim$'s strategies can be any partial function obeying the rules of the game (not necessarily computable).
\end{definition}

\begin{remark} Perhaps an amusing comment on the choice of names, which may (or may not) reinforce intuition. $\art$ is a mortal man who seeks to defeat the evil wizard $\mer$. They engage under the rules of $\mathfrak{G}(\overline{\U},\overline{\V})$, where $\mer$ issues a public challenge $x_n$ to $\art$ for every round $n$. $\nim$ is a benevolent nymph who decides to help $\art$ on his quest: at the $n$th round, she selects a parameter $B_n$, constraining $\mer$'s next move since the wizard can only pick $x_{n+1}\in B_n.$ Since $\art$ is mortal, his abilities are limited: his strategy must be computable and he can only see $\mer$'s public challenges (and nothing else). On the other hand, $\mer$ and $\nim$ are supernatural beings: they can see the entire play history (including $\nim$'s secret inputs), and there are no computability constraints on their strategies.
\end{remark}

\begin{definition}\label{def:LT-reducibility-upper-seq}
Let $\overline{\U},\overline{\V}$ be upper sequences.
We say that $\overline{\U}$ is {\em LT-reducible} to $\overline{\V}$, written
	$$\overline{\U}\leq_{\rm LT}\overline{\V},$$
just in case there exists a winning $\art$-$\nim$ strategy for the game $\mathfrak{G}(\overline{\U},\overline{\V})$.
\end{definition}

Definitions \ref{def:bilayer-game} and \ref{def:LT-reducibility-upper-seq} can be applied to partial $\omega$-sequences of subset families as well, but this does not change the structure of the order, as we explain below.

\begin{fact}\label{fact:LT-basic-for-upper-sequences}
Let $\overline{\U},\overline{\V}$ be partial $\omega$-sequences of subset families.
\begin{enumerate}[label=(\roman*)]
\item If $\overline{\U}$ is computable one-query reducible to $\overline{\V}$, then $\overline{\U}$ is LT-reducible to $\overline{\V}$.
\item $\overline{\U}$ is LT-equivalent to an upper sequence.
\item $\leq_{\rm LT}$ defines a preorder on upper sequences.
\end{enumerate}
\end{fact}

\begin{proof}
\begin{enumerate}[label=(\roman*):]
\item Computable one-query reducibility corresponds to the following play in the one-query LT-game:
\[
	\begin{array}{rcccc}
\mer \colon	& i\in I,A\in\U_i	&		& x\in B		&  \\
\art\colon	&	& 0,\varphi(i)\in J	&		& 1,\psi(i,x)\in A \\
\nim\colon	&	& B\in\V_{\varphi(i)}	&		& 	
	\end{array}
\]

\item  This follows from Observation \ref{obs:computable-Katetov-to-upseq} and item (i) above.

\item This translates \cite[Proposition 2.13]{Kih23}; by examining definitions, it is straightforward to see that so-called ``bilayer functions'' in {\em op. cit.} correspond to the upper sequences considered here.\footnote{\label{fn:bilayer-trans}{\em Details.} An upper sequence $\overline{\U}=(\U_n)_{n\in I}$ determines a bilayer function $f$ such that ${\rm dom}(f)=\{(n\mid A):A\in\U_n\}$ and $f(n\mid A)=A$.
Conversely, a bilayer function $f$ yields a partial $\w$-sequence $\overline{\U}=(\U_n)_{n\in I}$ such that $\U_n=\{f(n\mid A):(n\mid A)\in{\rm dom}(f)\}$, where $n\in I$ iff $(n\mid A)\in{\rm dom}(f)$ for some $A$. By taking upward closures as in Observation~\ref{obs:computable-Katetov-to-upseq}, we obtain the desired upper sequence. 
} 
\end{enumerate}

\end{proof}

This game-theoretic description sets up a major result in \cite{Kih23}, giving one answer to Problem~\ref{prob:LT}.

\begin{theorem}\label{thm:kih-bil}
The $\clt$-order on upper sequences is isomorphic to the Lawvere-Tierney order on LT-topologies over the Effective Topos $\Eff$.
\end{theorem}

\begin{proof}
This is \cite[Corollary 3.5]{Kih23} via the translation described in Footnote~\ref{fn:bilayer-trans}, originally used in the proof of Fact \ref{fact:LT-basic-for-upper-sequences} (iii).
\end{proof}

\begin{discussion}[Two-dimensional complexity]\label{dis:2d-complex}
This suggests that the structure of the LT-topologies reflects a two-dimensional structure, reflecting combinatorial complexity and computable complexity:
\begin{itemize}
\item A principal (ultra)filter is located at the least level in the Kat\v{e}tov order.
\item A constant (multi-valued) function is located at the least level in the Weihrauch order.
\end{itemize}
Thus, informally, combinatorial complexity measures how non-principal an upper sequence is, while computable complexity measures how non-constant it is. These complexity measures appear to be orthogonal, and they combine to determine the structure of the LT-topologies.
\end{discussion}

The reader who has understood thus far will notice that the LT-game is clearly analogous to the game-theoretic characterisation of the Gamified Kat\v{e}tov order (Theorem \ref{thm:game-GTK}).
\begin{proposition}\label{prop:computable-GK-order}
For any upper sets $\calU,\calV\subseteq \Pw$, say that $\calU$ is {\em computable Gamified Kat\v{e}tov reducible} to $\calV$ just in case:
\begin{itemize}
\item 
		Player II has a winning strategy for the game $\GUV$ (in the sense of Definition~\ref{def:fin-quer-kat}), 
        and Player II$_0$'s strategy is computable.
\end{itemize}
	\underline{Then},
	\begin{enumerate}[label=(\roman*)]
		\item $\U\leq_{\rm LT}\V$ iff $\U$ is computable Gamified Kat\v{e}tov reducible to $\V$.
		\item The computable Gamified Kat\v{e}tov order forms a preorder on upper sets.
	\end{enumerate}
\end{proposition}

\begin{proof}
(i) is obvious.
(ii) follows from item (i) here and Fact \ref{fact:LT-basic-for-upper-sequences} (iii).
\end{proof}

\begin{discussion} One also expects that the (non-computable) Gamified Kat\v{e}tov order $\glt$ forms a preorder on upper sets; as presently stated in Definition~\ref{def:GTK}, it is only a relation on upper sets. In principle, one can proceed by unpacking definitions, but the combinatorics makes this approach cumbersome. 

A more conceptual route is to isolate the non-computable aspects of the Gamified Kat\v{e}tov order. In Theorem~\ref{thm:preorder}, we will show that $\glt$ is equivalent to the $\clt$-order relativised to an arbitrary Turing oracle. 
Since $\clt$ is already known to be preorder, transitivity of $\glt$ then follows essentially for free. The main analytic work therefore shifts to clarifying the relationship between $\clt$ and $\glt$ over upper sequences, the first major objective of Section~\ref{sec:Turing}.
\end{discussion}

\subsubsection*{The path forward} Different readers will be interested in different aspects of our programme; the structure of this paper reflects this accordingly. Sections~\ref{sec:Tukey} and~\ref{sec:separation} are combinatorial set theory: Section~\ref{sec:Tukey} shows the Tukey and Gamified Kat\v{e}tov order are incomparable on filters over $\w$, whereas Section~\ref{sec:separation} analyses the induced Gamified Kat\v{e}tov order on ideals. Section~\ref{sec:Turing} introduces and develops the Gamified Kat\v{e}tov order on upper sequences, before investigating its connections with computability. All three sections can be read independently of the other; modulo some minor points of contact, they essentially only depend on the results of this section. For the reader's convenience, the theorem below summarises the main highlights.

\begin{theorem}\label{thm:GTK} The Gamified Kat\v{e}tov order defines a preorder on upper sets over $\w$. In particular:
	\begin{enumerate}[label=(\roman*)]
		\item $\calU\glt\calV$ iff there exists a $\calV$-fence $\delta$ such that $\calU\lk \calV^{\otimes[\delta]}$. In particular, $$\calU\elt \calU\otimes\calU.$$
        \item $\calU\glt \calV$ iff Player II has a winning strategy for the finite-query Kat\v{e}tov game $\GUV$.
		\item The computable Gamified Kat\v{e}tov order is equivalent to $\clt$ on upper sets over $\w$.
	\end{enumerate}
\end{theorem}
\begin{proof} That $\glt$ defines a preorder on upper sets follows from Corollary~\ref{cor:preorder}, to be proved later. (i) is Theorem~\ref{thm:fubini-GTK} and Corollary~\ref{cor:coarse}, (ii) is Theorem~\ref{thm:game-GTK}, and (iii) is Proposition~\ref{prop:computable-GK-order}.
\end{proof}

\begin{remark}
The constant upper sequences (i.e., upper sets) correspond to Lee-van Oosten's basic LT topologies \cite{LvO13}; see Summary Theorem~\ref{sumthm-1}. Hence, Proposition \ref{prop:computable-GK-order} shows that the computable Gamified Kat\v{e}tov order is isomorphic to the structure of the basic LT topologies. This connection gives a principled explanation for the underlying combinatorics of Lee-van Oosten's results.
\end{remark}

\begin{conclusion}\label{conc:set-and-comp} The computable analogue of the Gamified Kat\v{e}tov order on upper sets coincides with the $\clt$ order on basic topologies in $\Eff$. Hence, at least in the present context, set-theoretic and computable complexity are controlled by the same underlying mechanism.
\end{conclusion}

	\section{Incomparability with Tukey}\label{sec:Tukey}

To appreciate the dimensions of a new discovery, a shift in perspective can be helpful. In this section, we compare the Gamified Kat\v{e}tov order with another well-known order: the {\em Tukey order}. This definition originates in Tukey's study of Moore-Smith convergence in topological spaces, providing an abstract framework for comparing directed sets by cofinal type. After some preliminaries (Section~\ref{sec:tukey-prelim}), Observation~\ref{obs:coarser} explains our interest: the Tukey order on filters over $\w$ is also strictly coarser than the Rudin-Keisler order. 
\smallskip

This prompts a natural question: is the Tukey order coarser than the Gamified Kat\v{e}tov order, or vice versa? Intriguingly, neither is true. Section~\ref{sec:not-tukey} establishes that the Tukey and Gamified Kat\v{e}tov orders on filters are incomparable in ZFC. This result gives compelling evidence that the Gamified Kat\v{e}tov order calibrates a fundamentally distinct notion of complexity. However, there is a nuance. When restricted to just ultrafilters, incomparability is independent of ZFC -- while choice ensures non-principal ultrafilters, further assumptions are needed to separate the two orders.

	\subsection{Preliminaries}\label{sec:tukey-prelim} A {\em directed set} is a poset $(P,\leq_P)$ such that any pair of elements have a common upper bound:  $\forall a,b\in P$,  there exists $c\in P$ such that $a\leq c$ and $b\leq c$. 
	Examples include any filter $\calF$, now viewed as $(\calF,\supseteq)$, as well as any ideal $\calI$, now viewed as $(\calI,\subseteq)$.  
	
	\begin{definition} For a directed set $(P,\leq_P)$:
	\begin{enumerate}[label=(\roman*)]
		\item A subset $S\subseteq P$ is  \emph{cofinal} if for every $p\in P$ there exists $s\in S$ with $p\leq_P s$.
		\item A subset $S\subseteq P$ is  \emph{unbounded} if for every $p\in P$ there exists $s\in S$ such that $p\nleq_P s$.
	\end{enumerate}	
	\end{definition}

	\begin{fact}[see e.g.~{\cite[Theorem 513E]{Fremlin}}]\label{fact:Tukey}
		For directed sets $(P,\leq_P), (Q,\leq_Q)$, the following are equivalent:
		\begin{enumerate}[label=(\roman*)]
			\item There exists a function $\varphi\colon P \to Q$ such that if $S\subseteq P$ is unbounded, then so is $\varphi[S]$.
			\item There exists a function $\psi\colon Q\to P$ such that if $S\subseteq Q$ is cofinal, then so is $\psi[S]$.
			\item There exists a {\em Galois-Tukey connection} $(\varphi,\psi)$, i.e. there exists functions $\varphi\colon P \to Q$ and $\psi\colon Q \to P$ such that, for any $A\in P$ and $B\in Q$, 
			$$ \psi(B)\leq_P A \implies B\leq_{Q}\varphi(A).$$
		\end{enumerate}
	\end{fact}
\noindent If any of these conditions hold, we write $$P\lT Q\;,$$
and say that $Q$ \emph{Tukey dominates} $P$. This defines the {\em Tukey order} on directed sets. If $P\lT Q$ and $Q\lT P$, then we call $P$ and $Q$ {\em Tukey equivalent}, and write $P\eT Q$. In particular, the {\em cofinal type} of $P$ is defined as the Tukey equivalence class $[P]_\mathrm{Tuk}$.
\medskip 

To highlight the unique perspective of the Tukey order, it is useful to compare the definition with the preorders considered before. The Kat\v{e}tov order may be viewed either as an order on filters or, dually, as an order on ideals. By contrast, the Tukey order is defined on directed sets, and so provides unified framework that treats filters and ideals simultaneously. As preparation for later arguments, we collect several basic properties concerning this duality in the Tukey context.


\begin{lemma}\label{lem:Tukey} Let $\calI\subseteq\Pw$ be an ideal, and $\calI^\ast:=\{ \,\w\setminus A\mid A\in\calI \,\}$ its dual filter. \underline{Then}, 
	\begin{enumerate}[label=(\roman*)]
		\item  $(\calI^\ast,\supseteq)\equiv_{\mathrm{Tuk}} (\calI,\subseteq).$ 
		\item  $(\calI^\ast\otimes \calI^\ast,\supseteq)\eT (\calI\otimes\calI,\subseteq).$
	\end{enumerate}	
\end{lemma}

\begin{proof} (i) is obvious, since $$A\subseteq B \iff \w\setminus B \subseteq \w\setminus A$$ for any pair of subsets $A,B\subseteq \w$.  Hence, the functions $\varphi\colon \calI\to \calI^*$ and $\psi\colon \calI^*\to \calI$ defined by mapping $A\mapsto \omega\setminus A$ are both cofinal. (ii) is immediate from (i) and the fact that $\calI^*\otimes\calI^*=(\calI\otimes\calI)^*$ (Remark~\ref{rem:fubini-consistent}).
\end{proof}	

	\subsection{Incomparability}\label{sec:not-tukey} As summarised in Conclusion~\ref{conc:fubini}: the Gamified Kat\v{e}tov order is equivalent to the classical Kat\v{e}tov order closed under well-founded Fubini iterates, which explains its relative coarseness. The following observation places this perspective in relation to the Tukey order on filters over $\w$.
	
	\begin{observation}\label{obs:coarser} Let $\calF,\calG\subseteq \Pw$ be filters. \underline{Then},
		\begin{enumerate}[label=(\roman*)]
			\item $\calF\lrk\calG \implies \calF\lT \calG$. 
			\item  $\calF^{\otimes 2}\eT \calF^{\otimes 3}$\,. In particular, Tukey is strictly coarser than Rudin-Keisler.	
		\end{enumerate}
		
	\end{observation}
\begin{proof} \hfill
\begin{enumerate}[label=(\roman*):]
	\item This is well known for ultrafilters \cite[Fact 1]{DoTo11}, but the same argument applies to filters.
	Suppose $h\colon\w\to\w$ witnesses $\calF\lrk\calG$. Then, define
	\begin{align*}
	\psi\colon \calG &\longrightarrow \calF\\
	A &\longmapsto h[A]=\{h(n)\mid n\in A \} \,\,.
	\end{align*}
   Importantly, the Rudin-Keisler condition ensures $\psi$ is well-defined. 
  One then easily checks $\psi$ is cofinal, and thus witnesses $\calF\lT\calG$. 

	\item For principal filters, all finite Fubini powers are again principal, and hence Tukey equivalent to the singleton order $\textbf{1}$. 
    For non-principal filters $\calF$ (not necessarily maximal), the stated Tukey equivalence was proved by Milovich \cite[Theorem 5.2]{Mil08}. The claim about strict coarseness follows from Fact~\ref{fact:rk-max}. 
\end{enumerate}	
\end{proof}

Observation~\ref{obs:coarser} is especially suggestive in light of Corollary~\ref{cor:coarse}, which showed that $\calF\elt \calF^{\otimes 2}$ for all filters $\calF$ on $\Pw$. It was therefore very interesting to the authors when we discovered that the two orders are, in fact, subtly misaligned. 

\begin{theorem}\label{thm:Tukey} The Gamified Kat\v{e}tov and Tukey orders are incomparable on filters over $\w$ in ZFC. More explicitly: 
	\begin{enumerate}[label=(\roman*)]
		\item There exists filters $\calF_1,\calF_2$ such that $\calF_1\elt\calF_2$ yet $\calF_1 \not\eT \calF_2$.
		\item There exists filters $\calG_1,\calG_2$ such that $\calG_1\eT\calG_2$ yet $\calG_1 \not\elt \calG_2$.
	\end{enumerate}
\end{theorem}

	\subsubsection{First Direction} To indicate the way forward, we start with an independence result; this will tell us where {\em not} to look for answers.
	
\begin{proposition}\label{prop:indep} The following statement is independent from ZFC: ``There exists ultrafilters $\calU,\calV$ on $\w$ such that $\calU\elt \calV$, yet $\calU\not\eT\calV$.''
\end{proposition}
\begin{proof}
{\it Consistency.}
If $\frap=\mathfrak{c}$, then there exists a $p$-point $\mathcal{U}$ such that $\calU<_{\mathrm{Tuk}} \calU\otimes \calU$ \cite[Theorem 38]{DoTo11}. In such a model, we obtain $\calU<_{\mathrm{Tuk}} \calU\otimes \calU\elt \calU$ by Corollary~\ref{cor:coarse}.

{\it Consistency of the negation.}
It is also consistent with ZFC that all non-principal ultrafilters are Tukey equivalent \cite[Theorem 1.2]{TukeyTrivial}. In such a model, $\calU\not\eT\calV$ implies that one must be principal and the other non-principal. In which case, $\calU\not\elt\calV$ as well by Corollary~\ref{cor:princ-ultra}.
\end{proof}
	
Turning away from ultrafilters, we direct our attention to the polynomial growth ideal,
	$$\Poly:=\Big\{A\subseteq \w\Bmid (\exists k\in\omega). |A\cap 2^n|\leq n^k, \,\text{for all $n\geq 2$}\Big\}\;.$$
	The desired result will then follow from showing that $\Poly \not\geq_{\mathrm{Tuk}} \ww$ and $\Poly\otimes\Poly\geq_{\mathrm{Tuk}}\ww$.
	This relies on the following key lemma, which adapts \cite[Fact 31]{DoTo11} -- originally stated for non-principal ultrafilters. 

	\begin{lemma}\label{lem:FinFubini} Setup: 
		\begin{itemize}
			\item Let $\calI$ be an ideal such that $\Fin\subseteq \calI$;
		\item Regard any string $x\in\ww$ as a function $x\colon \w\to\w$, where $x(n):=$ the $n$th entry in string $x$.
			\item Consider $(\ww,\leq)$ as a directed set, where $x\leq y : \iff x(n)\leq y(n)$ for each $n<\w$.
		\end{itemize}  \underline{Then}, 
		$$\ww\lT \calI\otimes\calI.$$
	\end{lemma}
	\begin{proof} Let $\calI^\ast:=\{\omega\setminus A\mid A\in\calI\}$ be the dual filter of $\calI$. Hence, $\calI^\ast$ contains all cofinite sets since $\Fin\subseteq \calI$ by hypothesis. By Lemma~\ref{lem:Tukey}, we know
		$$(\calI^\ast\otimes \calI^\ast,\supseteq)=((\calI\otimes \calI)^\ast,\supseteq)\equiv_{\mathrm{Tuk}} (\calI\otimes\calI,\subseteq).$$
		Hence, to establish the result, it suffices to define a cofinal map $f\colon \calI^\ast\otimes\calI^\ast \to \ww$. 
		
		Towards this end, given any $A \in \calI^\ast \otimes \calI^\ast$, define a function $f_A\in\ww$ by:
		\[
		f_A(k) := \min A_{(n_k)}, \quad \text{where } \{n_k\}_{k<\omega} \text{ enumerates those } n \text{ with } A_{(n)} \in \calI^\ast.
		\]
		Now define subset $\calX$ to consist of those $A\in\calI^\ast\otimes\calI^\ast$ satisfying
		\begin{enumerate}[label=(\alph*)]
			\item Whenever $A_{(n)}\neq\emptyset$, then $A_{(n)}\in \calI^\ast$;
			\item Whenever $m<n$ and $A_{(m)},A_{(n)}\in\calI^\ast$, then $\min A_{(m)}\leq \min A_{(n)}$. 
		\end{enumerate}
		Condition (a) eliminates noise from non-empty sections that fail to be $\calI^\ast$-large; Condition (b) enforces monotonicity of minima across $\calI^\ast$-large sections.

		We claim $\calX$ is a cofinal subset of $\calI^\ast\otimes\calI^\ast$. Why? Pick $B\in\calI^\ast\otimes\calI^\ast$. We refine $B$ in two stages.
		\begin{itemize}
			\item {\em Stage 1.} 
			Remove any $B_{(n)}$ that is non-empty but not in $\calI^\ast$ to obtain $B'$. Notice $\{n\mid B'_{(n)}\in\calI^\ast \} = \{n\mid B_{(n)}\in\calI^\ast \}  \in\calI^\ast$ for all $n$, and so $B'\in\calI^\ast\otimes\calI^\ast$ by definition of the Fubini product.
			
			
			\item {\em Stage 2.} To enforce monotonicity, proceed inductively: if $m < n$ and both $B'_{(m)}$, $B'_{(n)} \in \calI^\ast$ but $\min B'_{(n)} < \min B'_{(m)}$, then intersect $B'_{(n)}$ with the cofinite set $C := \N \setminus \{0,\dots, \min B'_{(m)}\}$ to obtain $B''_{(n)} := B'_{(n)} \cap C \in \calI^\ast$, ensuring $\min B''_{(n)} \geq \min B'_{(m)}$. Iterating this yields $B'' \in \calI^\ast \otimes \calI^\ast$ with $B'' \subseteq B'$ and $B'' \in \calX$.
			
		\end{itemize} 
		Thus, every $B$ has a refinement in $\calX$, and so $\calX$ is cofinal. By \cite[Fact 3]{DoTo11}, $(\calX,\supseteq)\equiv_{\mathrm{Tuk}} (\calI^\ast\otimes\calI^\ast,\supseteq)$ and so it suffices to show $f\restr{\calX}$ is cofinal. In particular, by \cite[Fact 5]{DoTo11}, $f\restr{\calX}$ is cofinal if it is monotone and has cofinal range, which we show below.
		
		\begin{itemize}
			\item {\em Monotonicity.} Let $A,B\in\calX$ such that $A\supseteq B$. Then, the sequence $\{i_k\}_{k<\omega}$ enumerating those $n$ for which $B_{(n)}\in\calI^\ast$ is a subsequence of the sequence $\{n_k\}_{k<\omega}$ enumerating those $n$ for which $A_{(n)}\in\calI^\ast.$ In particular, this implies $n_k\leq i_k$ for each $k$. Therefore,
			$$\min A_{(n_k)}\leq \min A_{(i_k)}\leq \min B_{(i_k)}.$$
			The first inequality follows from the fact $A\in\calX$, and thus satisfies Condition (b), while the second follows from the fact $A_{(i_k)}\supseteq B_{(i_k)}$, which holds by assumption. Hence, $f_A(k)\leq f_B(k)$ for all $k<\omega$, and so $f\restr{\calX}$ is monotone.
			\item \emph{Cofinal range.} Given any $h \in \omega^\omega$, define
			\[
			A := \{(n, l) \in \omega \times \omega \mid l > \max \{h(i) \mid i \leq n\}\}.
			\]
			Each vertical section $A_{(n)}$ is cofinite, hence in $\calI^\ast$, and $\min A_{(m)} \leq \min A_{(n)}$ for $m < n$, so $A \in \calX$. By construction, $h(n) \leq f_A(n)$ for all $n$, and so $f\rstr \calX $ is cofinal in $(\ww,\leq)$.
		\end{itemize}
	\end{proof}

	\begin{theorem}\label{thm:not-Tukey} There exists filters $\calU,\calV$ on $\omega$ such that such that $\calU\not\equiv_{\mathrm{Tuk}} \calV$ but $\calU\elt \calV$.
	\end{theorem}
	\begin{proof} By Lemma~\ref{lem:Tukey}, it suffices to show this result for ideals on $\w$. Hence, consider $\Poly$, the polynomial growth ideal, and its Fubini square. Then the following hold:
		\begin{enumerate}
			\item $\Poly\elt\Poly\otimes \Poly$.	
			
			\noindent [Why? Since Fubini products respect taking duals (Remark~\ref{rem:fubini-consistent}), we know $$\Poly^*\otimes \Poly^*=(\Poly\otimes \Poly)^*.$$ The rest is immediate from Corollary~\ref{cor:coarse}.]	
			\item $\ww\lT\Poly\otimes \Poly.$
			
			\noindent [Why? By Lemma~\ref{lem:FinFubini} it suffices to show $\Poly\supseteq \Fin$. Given any $A\in\Fin$, there exists some $k\in\omega$ such that $A\subseteq [0,2^k]$. Notice then that
			$|A\cap 2^n|\leq |A| \leq 2^k\leq n^k$ for all $n\geq 2.$]
			
			\item $\ww\not\leq_{\mathrm{Tuk}}\Poly.$
			
			\noindent [Why? There are various ways to see this, but let us sketch the argument by Louveau-Veli\v{c}kovi\'{c} \cite{LVIdeals}. To start, let $\mathfrak{d}$ denote the dominating number, i.e. $\mathfrak{d}$ is the least cardinality of a dominating family of functions in $\ww$. Theorem 1 of {\em op. cit.} states the following:  for any analytic ideal $\calI$ on $\w$,
            $$\ww\lT \calI \iff \calI\,\, \text{is not the union of less than $\mathfrak{d}$ weakly bounded sets.}$$ 
    However, as shown in Example 1 of the same paper, $\Poly$ is an analytic ideal that can be presented as the union of just {\em countably} many weakly bounded sets. Hence, $\ww\not\leq_{\mathrm{Tuk}}\Poly.$]

		\end{enumerate}
		Summarising (1) - (3), we get: $\Poly\elt\Poly\otimes\Poly$ even though $\Poly\not\equiv_{\mathrm{Tuk}}\Poly\otimes \Poly$.
	\end{proof}

\subsection{Second Direction}\label{sec:not-LT}	Our independence result, Proposition~\ref{prop:indep}, relies on a recent breakthrough in the area by Cancino-Manr\'{i}quez and Zapletal \cite{TukeyTrivial}: it is consistent with ZFC that all non-principal ultrafilters on $\w$ are Tukey equivalent. With this in mind, we give another angle from which to appreciate the difference between the Tukey order and our framework.
	
\begin{theorem}\label{thm:not-LT} There exists non-principal ultrafilters $\calU,\calV$ on $\w$ such that $\calU\not\equiv^\circ_{\mathrm{LT}} \calV$ but $\calU\eT\calV$ (in ZFC).
\end{theorem}
	\begin{proof} The proof is a cardinality argument, and proceeds by bounding the number of ultrafilters belonging to an $\elt$-class of filters. 
		
\subsubsection*{Step 1: A bound}
		Let $\calU$ be an ultrafilter, and $\calV$ a filter on $\w$. Assume $\Phi$ is a partial continuous function witnessing $\calU\glt\calV$. For any $A\subseteq\w$, either $A\in\calU$ or $\omega\setminus A\in\calU$. Expressed in our language, this gives rise to two cases:
		\begin{enumerate}
			\item If $A\in \calU$, there exists a $\calV$-branching tree $T$ such that $\Phi[T]\subseteq A$.
			\item If $A\notin \calU$, there exists a $\calV$-branching tree $T'$ such that $\Phi[T']\subseteq \omega\setminus A.$
		\end{enumerate}
	We claim that Cases (1) and (2)	cannot both be true. To see why, consider $\o\in T\cap T'$. This defines two $\calV$-large sets of successors: $\{n\mid \o\fr n\in T\}$ and $\{n\mid \o\fr n\in T'\}$. Since $\calV$ is a filter, the two successors sets have a non-empty intersection, say $n\in T\cap T'$. Repeating this process, we obtain a common infinite path $p\in [T]\cap [T']$, but this leads to a contradiction: $\Phi(p)\in A$ and $\Phi(p)\in\w\setminus A$. Hence, to summarise: 
	$$A\in\calU \iff \; \exists T\;\text{a}\;\calV\text{-branching tree with } \Phi[T]\subseteq A\;.$$ Consequently, given any pair $(\calV,\Phi)$ with $\calV$ a filter and a partial continuous function $\Phi\pcolon\ww\to\w$, there exists at most one ultrafilter $\calU$ such that $\Phi$ witnesses $\calU\glt\calV$.
\subsubsection*{Step 2: Comparing cardinalities.} Reviewing Convention~\ref{con:pcf}, a partial continuous function $\Phi\pcolon\ww\to\w$ is determined by a finite representative $\varphi\pcolon\lww\to\w$ satisfying a coherence condition. Any partial function $\varphi\pcolon\lww\to\w$ can be viewed as a function from the countable set $\lww$ to the countable set $\w\cup\{\bot\}$, where $\bot$ means ``undefined''. Since the set $(\w\cup\{\bot\})^{\lww}$ has cardinality $(\aleph_0)^{\aleph_0}=2^{\aleph_0}$, deduce that the family of partial continuous functions has cardinality $\leq 2^{\aleph_0}$. 

This observation sets up the punchline. By Step 1, our counting argument shows that any $\elt$-class of filters has at most $2^{\aleph_0}$-many ultrafilters. On the other hand, there exists a $\eT$-equivalence class possessing $2^{2^{\aleph_0}}$-many non-principal ultrafilters on $\w$ (see e.g. \cite[\S 3]{DoTo11}, though the discovery of a Tukey-top ultrafilter goes back to Isbell \cite{Isbell}). Hence, by cardinality considerations, there must exist non-principal ultrafilters $\calU$ and $\calV$ such that $\calU\eT\calV$ but $\calU\not\elt\calV$.
	\end{proof}
	
Having established the two directions, this proves the main result of this section.
	
\begin{proof}[Proof of Theorem~\ref{thm:Tukey}] By Theorems~\ref{thm:not-Tukey} and ~\ref{thm:not-LT}, the Gamified Kat\v{e}tov and Tukey orders are in fact incomparable on filters over $\w$. Unlike Proposition~\ref{prop:indep}, this result holds in ZFC.
\end{proof}	
	
\begin{discussion} The constructive mathematician may ask if Theorem~\ref{thm:Tukey} can be proved within ZF (i.e. without choice). Here is what we can say at the present moment. We believe (though have not checked) that the proof of Theorem~\ref{thm:not-Tukey} can be adapted to a constructive setting (modulo a careful re-examination of the claim that $\Poly\not\geq_{\mathrm{Tuk}} \ww$). By contrast, Theorem~\ref{thm:not-LT} is inherently non-constructive: it depends on non-principal ultrafilters whose existences are (indirectly) deduced from a cardinality argument. A proper constructive proof of incomparability would thus require a different approach, a topic of work in progress. 
\end{discussion}
	
	\section{Separation Results within Gamified Kat\v{e}tov}\label{sec:separation}

In Corollary~\ref{cor:coarse}, we showed that $\calU\otimes\calU\elt\calU$ for any upper set $\calU$, and identified this as a key driver behind the coarseness of the Gamified Kat\v{e}tov order. In this section, we investigate the extent of this coarseness by investigating some well-known examples. As a baseline, we know from Corollary~\ref{cor:princ-ultra} that the order at least distinguishes between principal vs. non-principal ultrafilters. So what other differences can it detect? What sea-changes in structural complexity might the order be responding to?

\begin{convention}\label{con:sep} The main examples in this section will be ideals, and we consider the induced Gamified Kat\v{e}tov order on them by taking duals (Convention~\ref{con:orders}). Whereas upper sets are a natural context for proving abstract results (as done previously), many examples in the literature have been presented as ideals. Of course, since Fubini products respect the duality between upper/lower sets, Corollary~\ref{cor:coarse} still applies:
	$$\calL\otimes\calL\elt \calL\;,\qquad \text{for any lower set}\; \calL\subseteq\Pw.$$

		\end{convention}
		
\subsection{MAD families}\label{sec:MAD} A subset family $\calA\subseteq \Pw$ is called {\em AD (Almost Disjoint)} if: (i) each member set is infinite; and (ii) any two distinct members  of $\calA$ has finite intersection. It is {\em MAD (Maximal Almost Djsoint)} if it is a maximal family with respect to this property. Informally, MAD families are very flat: their subsets are spread out across $\w$ such that each distinct pair only overlaps in a small area. 
\smallskip

Their relevance to our context is that every MAD family $\calA$ defines an ideal $\calI(\calA)$, defined as the ideal generated by $\calA$ and all finite subsets of $\w$. There has been much interest in studying the Kat\v{e}tov complexity of MAD families (via their ideal representations), guided by deep questions related to forcing and cardinal invariants.\footnote{See e.g.  \cite{BreYa05,HruZap08,Hru17,BreGuzHruRag22}.} It is thus natural to wonder what the Gamified Kat\v{e}tov order has to say about such questions, but here we are confronted with an elementary fact.  
	
	\begin{fact}[{{\cite[\S 2]{HruGF03}}}]\label{fact:MAD} Setup:
		\begin{itemize}
			\item For any MAD family $\calA$, we also write $\calA$ to denote its ideal representative. 
			\item Denote $\Fin$ as the ideal of all finite sets in $\w$. 
		\end{itemize}
		\underline{Then},
		\begin{enumerate}[label=(\roman*)]
			\item For every MAD family $\calA$, 
			$$\Fin\lk \calA \lk \Fin\otimes\Fin\;.$$
			\item Under every MAD family $\calA$, there is a strictly decreasing chain of MAD families with length $\mathfrak{c}^+$ in the Kat\v{e}tov order.
			\item Under every MAD family $\calA$, there exists $\mathfrak{c}$-many pairwise Kat\v{e}tov incomparable MAD families.
		\end{enumerate}
	\end{fact}

One way to read this result: there is a high degree of variation amongst MAD families, which is entirely mapped out between $\Fin$ and $\Fin\otimes \Fin$ in the classical Kat\v{e}tov order. However, in light of Theorem~\ref{thm:fubini-GTK}, Fact~\ref{fact:MAD} (i) implies that all MAD families are $\elt$-equivalent. 
In other words, all such differences are annihilated in our setting.
	
	\begin{conclusion}\label{conc:MAD} The Gamified Kat\v{e}tov order regards all forms of MAD-ness as equivalent.
    This may or may not be fortunate. 
		
	\end{conclusion}

\subsection{Some Key Examples}\label{sec:separate-example} Here we isolate the definitions of various ideals of interest to us. For additional context, see e.g.~\cite{Hru17}.
\smallskip

\begin{enumerate}
	\item {\em (Finite Ideal).} $\Fin$ is the set of all finite subsets of $\w$.
	\item {\em (Summable Ideal).} Given a function $f\colon\w\to[0,\infty)$ such that $$\sum_{n\in\w}f(n)=\infty\,,$$ 
	 $\Sumf$ is the set of all $A\subseteq \w$ such that $\sum_{n\in A} f(n)<\infty$. In the case when $f$ is
	$$f(n):=\begin{cases}
	\frac{1}{n} \qquad\text{if}\; n>0\\
	0 \qquad\,\text{if}\; n=0\\
	\end{cases}\quad,$$
	we write the ideal as $\Sumn$ . 
	\item {\em (Eventually Different Ideals).} First, partition $\w$ into infinite consecutive intervals $(I_{a})_{a\in\w}$ where the length of $I_a$ is $a$.\footnote{In other words, $\min I_0=0$, and $\max I_{n} +1 = \min I_{n+1}$.} Then:
	\begin{itemize}
		\item $\EDfin[m]$ defines the set of all $A\subseteq \w$ such that $|A\cap I_n|\leq m$ for all but finitely many $n\in\w$.
		\item $\EDfin=\bigcup_{m\in\w}\EDfin[m]$.
	\end{itemize}
\end{enumerate}

\begin{remark}[Base-change for $\EDfin$]\label{rem:base-EDfin} Having fixed the disjoint interval partition $\{I_a\}_{a\in\w}$, one may regard $\w$ as enumerating the elements of $\Delta=\{(r,s) \mid s\leq r\}$, where $(r,s)$ means the $s$th element in the $r$th interval $I_r$. Denote this enumeration via the bijection $\lranglet{-}{-} \colon\Delta\to\w$.\footnote{{\em Aside.} This is a small abuse of notation since the fixed bijection in Convention~\ref{con:base-change} is also denoted $\lranglet{-}{-}\colon\w^2\to\w$, but this should not cause confusion since they arise in different contexts.} 
\end{remark}
	
\subsection{The Main Technique}\label{sec:separation-technique} To prove separation, we establish various statements of the form $$\calH\not\glt\calI$$
below. As stated in Convention~\ref{con:sep}, we will take $\calH,\calI$ to be ideals/lower sets, so we are actually working with their dual filters/upper sets, obtained by taking complements in $\w$. The technology developed in this subsection reflects this, and is designed to amplify structural differences between the original ideals.



\begin{definition}\label{def:U-positive} Let $\calU\subseteq \Pw$ be an upper set. A set $A\subseteq \Pw$ is {\em $\calU$-null} if $\w\setminus A\in\calU$. Otherwise, say that $A$ is {\em $\calU$-positive}.
\end{definition}

\begin{observation}\label{obs:positive-null} Let $\calU\subseteq\Pw$ be an upper set. If $A$ is $\calU$-positive and $B\in\calU$, then $A\cap B\neq\emptyset.$
\end{observation}
\begin{proof} If $A\cap B=\emptyset$, then $B\subseteq \w\setminus A$. Since $B\in\calU$ and $\calU$ is upward closed, this implies $\w\setminus A\in\calU$. In other words, $A$ is $\calU$-null, a contradiction.
\end{proof}

In general, a $\calU$-positive set $A\subseteq \w$ need not belong to $\calU$; the definition only requires that its complement not be $\calU$-large. Hence, informally, we regard $\calU$-positivity as defining an abstract notion of {\em non-negligibility}. This sets up the main technical definition of the section, a {\em labelling function}.

\smallskip
As preparation, recall from Convention~\ref{con:pcf} that a partial continuous function $\Phi\pcolon\ww\to\w$ may be identified with a suitable partial function $\varphi\pcolon\lww\to\w$. We define a labelling function $\nu^\calU_\Phi$, which provides a natural way of extending $\varphi$ to a total function, guided by the advice of an upper set $\calU$.


\begin{definition}[Labelling Function]\label{def:label} Let $\Phi\pcolon\ww\to\w$ be a partial continuous function, and $\calU$ an upper set. Then:
\begin{enumerate}[label=(\roman*)]
	\item Define the tree $T_\Phi$ as the $\preceq$-downward closure of $\{\sigma\in\lww \mid \Phi(\sigma)\dar\}$. 
	\item Define a {\em labelling function} 
	$$\nu\colon \lww\to\w\cup\{\bot\}$$
	recursively as follows:
	\begin{enumerate}[label=(\arabic*)]
		\item If $\Phi(\sigma)\dar$, then set $\nu(\sigma)=\Phi(\sigma)$.
		\item If $\sigma\notin T_\Phi$, then set $\nu(\sigma)=\bot$.
		\end{enumerate}

		\begin{enumerate}[label=(\arabic*)]
	\setcounter{enumii}{2} 
		\item If there exists $c\in \w$ such that $\{n\in\w\mid \nu(\sigma\fr n)=c\}$  is $\calU$-positive for some $c\in\w$, then set $\nu(\sigma)=c$. If there are multiple options, choose the least $c$.
		\item Otherwise, set $\nu(\sigma)=\bot$.
	\end{enumerate}
 \item For each finite string $\sigma$, its {\em $\nu$-label} is the value $\nu(\sigma)\in\w\cup\{\bot\}$. 
\item A node $\sigma\in\lww$ is {\em $\nu$-critical} if $\nu(\sigma)=\bot$ and $\{n\mid \nu (\sigma\fr n)=\bot\}$ is $\calU$-null; a {\em $\nu$-critical successor} is a successor of a critical node.
 \end{enumerate}	
\end{definition}

\begin{convention} To emphasise dependence, we sometimes write the labelling function as $\nu_\Phi^{\calU}$. Typically, however, we suppress the indices for readability. In the case where $\calU=\calI^\ast$ is the dual filter of an ideal $\calI$, we abuse notation and write $\nu^\calI_\Phi$.
\end{convention}

The following discussion gives some intuition and additional context.

\begin{discussion}\label{dis:label} Let us unpack the various components of Definition~\ref{def:label}. 
	
\smallskip 
\noindent\emph{Recursive construction.}	 Step~(1) reproduces the known values of $\varphi$, and Step~(2) ignores strings outside the tree $T_\Phi$. 
Inside~$T_\Phi$, the recursion proceeds upward towards the root: given labels on successors, we determine the label of their predecessor. Step~(3) says: if a node $\sigma$ has a non-negligible subset of successors that agree on some value $c$, then $\sigma$ inherits that value. If several distinct values occur across non-negligible subsets, the least value is chosen. Otherwise, Step~(4) assigns $\bot$. 

Finally, notice: after Step (2), the set of all $\sigma\in T_\Phi$ such that $\varphi(\sigma)\uar$ forms a well-founded subtree. Hence, the recursion eventually terminates, yielding a total function $\nu^\calU_\Phi$.

\smallskip 
\noindent\emph{Role of $\calU$.} The upper set $\calU$ determines which sets of successors count as non-negligible, and thereby governs how local coherence amongst successor labels propagates upward through the tree. 

\smallskip
\noindent\emph{Critical nodes.} A node~$\sigma$ is \emph{$\nu$-critical} precisely when this coherence breaks down: for every $c\in \w\cup\{\bot\}$, the set $\{n\in\w\mid \nu(\sigma\fr n)=c\}$ is $\calU$-null.
\end{discussion}

Returning to the present context, how does Definition~\ref{def:label} relate to the Gamified Kat\v{e}tov order?  Notice for any $\Phi$ witnessing $\calU\glt\calV$,  the tree $T_\Phi$ extends $\dom(\Phi)$. Hence, for any $A\in\calU$, there exists a $\calV$-branching tree $T_A\subseteq T_{\Phi}$ such that $\Phi[T_A]\subseteq A$ holds. The following two technical lemmas sharpen the picture.

\begin{lemma}\label{lem:separation-1}Let $T$ be a $\calV$-branching tree.
	If a node $\sigma\in T$ is $\nu^\calV_\Phi$-labeled by $c\in\w$, then there exists an infinite path $p$ through $T$ such that $\Phi(p)=c$.
\end{lemma}
\begin{proof} To frame our objective, recall the following characterisation: for any infinite path $p\in [T]$, we have $\Phi(p)=c$ just in case there exists some $l\in\w$ such that $\varphi(p\rstr l)\dar$ and $\varphi(p\rstr l)=c$.	
	
Now suppose $\sigma\in T$ such that $\sigma$ is labelled $c$. The existence of a non-$\bot$ label means $\sigma\in T_\Phi$. If $\varphi(\sigma)\dar$, then we are done: since $T$ is $\calV$-branching, there exists an infinite path $p\in[T]$ extending $\sigma$ such that $\Phi(p)=c$. Suppose instead $\varphi(\sigma)\uar$. By Step (3) of Definition~\ref{def:label}, this means $\{n \in\w \mid \nu(\sigma\fr n)=c\}$ is $\calV$-positive. There is no {\em a priori} reason to assume that the $c$-labelled successors form a subset of $T$. However, since $T$ is $\calV$-branching, $\{n\in\w \mid \sigma\fr n\in T\}\in\calV$ by definition. Hence, by Observation~\ref{obs:positive-null}, there exists {\em some} $n\in\w$ such that $\sigma\fr n\in T$ {\em and} $\nu(\sigma\fr n)=c$. Perform the same argument for $\sigma\fr n$, and repeat this process to obtain an infinite path $p$ through $T$ that extends $\sigma$ such that $\nu(p\rstr l)=c$ for any $l\geq |\sigma|$. Since we are labelling these prefixes with $c$ (and not $\bot$), this means $\varphi(p\rstr l)\dar$ for sufficiently large $l$. Since $p\in[T]$ by construction, conclude that $\Phi(p)\dar$ and $\Phi(p)=\varphi(p\restr l )= c$ for some sufficiently large $l\geq |\sigma|$.
\end{proof}

\begin{lemma}\label{lem:separation-2} Let $T$ be a $\calV$-branching tree. If $[T]\subseteq\dom(\Phi)$ and its root is labeled by $\bot$, then $T$ has a $\nu^\calV_\Phi$-critical node.
\end{lemma}
\begin{proof} Let $\sigma\in T$ be a node labelled by $\bot$, and assume for contradiction $T$ has no critical node. Then, by definition, $\{n \mid \nu(\sigma\fr n)=\bot\}$ must be $\calV$-positive. Hence, playing the same game as in Lemma~\ref{lem:separation-1}, we obtain an infinite path $p\in [T]$ that extends $\sigma$ such that $\nu(p\rstr l)=\bot$ for any $l\geq |\sigma|$. However, $[T]\subseteq \dom(\Phi)$ by hypothesis, which means $\varphi(p\rstr l)\dar$ for sufficiently large $l$. In other words, we obtain $\nu(p\rstr l)\neq\bot$ for some $l\geq |\sigma|$, a contradiction. Hence, $T$ must have a critical node.
\end{proof}


\subsection{A chain above $\Fin$}	We now leverage the previous technical lemmas to construct an infinite strictly ascending chain of lower sets above $\Fin$
	$$\Fin \sglt \EDfin[1]\sglt \dots \sglt \EDfin[m]\sglt \EDfin[m+1]\sglt \dots \sglt \EDfin\sglt \mathrm{Sum}_{1/n}\,,$$
within the Gamified Kat\v{e}tov order; the reader should contrast this with Conclusion~\ref{conc:MAD} regarding MAD families. This will be the main theorem of the section. The technical challenge posed by this result lies in showing {\em strictness} of this ascending chain. We work this out below on a case-by-case basis (Theorems~\ref{thm:EDfin} - \ref{thm:Sumf}), before assembling the final result. 
	
\begin{theorem}\label{thm:EDfin} $\EDfin[1]\not\glt\Fin$.
	
\end{theorem}
\begin{proof}Suppose for contradiction there exists $\Phi$ witnessing $\EDfin[1]\glt \Fin$, and consider the labelling function $\nu:=\nu^\Fin_\Phi$. There are two main cases to check, depending on the $\nu$-label of the root $\o\in T_\Phi$.
\begin{itemize}
	\item\textbf{Case 1: $\nu(\o)\in\w$}. 
	
\noindent 	Suppose the root is $\nu$-labelled $c\in\w$. It is clear $\{c\}\in \EDfin[1]$; thus, denoting $A:=\w\setminus\{c\}$, we have $A\in (\EDfin[1])^\ast$. By assumption, there exists a $\Fin^*$-branching tree $T_A$ such that $[T_A]\subseteq \dom(\Phi)$ and $\Phi[T_A]\subseteq A$. However, by Lemma~\ref{lem:separation-1}, there exists $p\in [T_A]$ such that $\Phi(p)=c\notin A$, a contradiction.
	\smallskip
	\item \textbf{Case 2: $\nu(\o)=\bot$}. 
	
\noindent We start with a key observation. Suppose $\tau\in T_\Phi$ is $\nu$-critical. By definition, $\{n \mid \nu(\tau\fr n)=c\}\in\Fin$ for every $c\in\w\cup\{\bot\}$. Hence, $\tau$ must have infinite distinct $\nu$-labels amongst its successors. Why? If there were only finitely many values of $\nu(\tau\fr n)$, then at least one of those values would occur on infinitely many successors, contradicting criticality. 
\smallskip

To proceed, it will be helpful to organise the main proof into different sections.
\medskip

\noindent \underline{Setup}. We start by introducing some preliminary data:

\smallskip

\begin{enumerate}
	\item  By Lemma~\ref{lem:separation-2} and our assumption that $\Phi$ witnesses $\EDfin[1]\glt\Fin$, $T_\Phi$ contains at least one $\nu$-critical node in $T_\Phi$. Enumerate all such critical nodes as $\{\tau_i\}_{i\in\omega}$.\footnote{
By allowing repetition in enumeration, we may assume it is indexed by $\omega$.
}
	\item Enumerate the (disjoint) partition intervals $\{I_a\}_{a\in\omega}$ as in the definition of $\EDfin[1]$.
	\item Our construction will proceed in stages. Note that each $k\in\omega$ can be written as $k=\langle i,b\rangle$. At stage $k=\langle i,b\rangle$, we focus on the critical node $\tau_i$ and its successors $\geq b$.
\end{enumerate}

\medskip

\noindent \underline{Goal and requirements}.
We construct two sequences in stages
\[
\{a_k\}_{k\in\omega}
\qquad
\{t_k\}_{k\in\omega}\ ,
\]
satisfying the following conditions for all $k=\lranglet{i}{b}$:
\smallskip
\begin{enumerate}[label=(\alph*)]
		\item ({\em Monotonicity}). $\{a_k\}_{k\in\omega}$ is strictly increasing.
		\item ({\em Well-labelled}).  The label $$c_{k}:=\nu(\tau_i\fr t_k)$$
		is a natural number; in other words, we rule out nodes labelled $\bot$.
		\item ({\em ${\rm ED}_{\rm fin}[1]$-smallness}).
		Each label $c_k$ avoids the union of all intervals up to $I_{a_{k-1}}$, i.e.
		$$c_k\notin \bigcup_{a\le a_{k-1}} I_a\quad.$$
		\item ({\em ${\rm Fin^*}$-positivity}). The sequence $\{t_k\}$ intersects every cofinite set\footnote{This is analogous to Observation~\ref{obs:positive-null}, justifying the name.}, i.e.
		$$t_k\in\{n\in\omega\mid n\geq b\} \ .$$
\end{enumerate}
\smallskip

\noindent This can be arranged by a standard diagonalisation over all $k=\langle i,b\rangle$, ensuring that -- under each critical node $\tau_i$ -- we sample (distinct) successors arbitrarily far out along $\w$. 

\medskip 
\noindent \underline{Diagonalisation}. 
At the beginning of stage $k=\langle i,b\rangle$, suppose $a_0<\dots<a_{k-1}$ has been chosen, and let $$D_k:=\bigcup_{a\le a_{k-1}}I_a \cup \{\bot\}\;.$$
Now suppose stage $k=\langle i,b\rangle$ of the diagonalisation requires us to pick a successor of the critical node $\tau_i$. Recall: since $\tau_i$ is a critical node, it has infinitely many distinct $\nu$-labels amongst its successors. Since $D_k$ is finite, there exists $s$ such that $\nu(\tau_i\fr m)\notin D_k$ for all $m\geq s$. Hence, pick $t_k> \max\{s,b\}$; set $c_k:=\nu(\tau_i\fr t_k)$, and pick $a_k$ to be the (unique) index with $c_k\in I_{a_k}$.
\smallskip

 Reviewing the requirements: the construction obviously ensures smallness and positivity.
Monotonicity is preserved at each stage, and all successors are well-labelled since we require $c_k\notin D_k$ (and so {\em a fortiori} $c_k\neq \bot$).

\medskip
\noindent \underline{Verification}. Consider the set of all labels $\{c_k\}$ from the sequence $\{t_{k}\}$, denoted
$$C:=\{ c_k \mid k\in\w\}.$$
By construction, the set $C$ lies in $\EDfin[1]$. Thus, let $A:=\w\setminus C$, so $A\in (\EDfin[1])^\ast$.\footnote{Notice: had we not ruled out nodes labelled $\bot$ from before, then $\w\setminus C$ would no longer be well-defined.}

\medskip
\noindent \underline{Contradiction}. By the assumption that $\Phi$ witnesses $\EDfin[1]\glt\Fin$, there exists a $\Fin^*$-branching tree $T_A$ such that $[T_A]\subseteq\dom(\Phi)$ and $\Phi[T_A]\subseteq A$.  By Lemma~\ref{lem:separation-2}, $T_A$ contains one of the $\nu$-critical nodes enumerated above; fix this node $\tau_i$. Since $T_A$ is $\Fin^*$-branching, the set of successors $\{n \mid \tau_i\fr n\}$ is cofinite. Explicitly: there exists some bound $s$ such that $\tau_i\fr n\in T_A$ for all $n\geq s$.
Then, for $k=\langle i,s\rangle$, by positivity, we have $t_k\geq s$; hence $\tau_i\fr t_k\in T_A$ and $\nu(\tau_i\fr t_k)=c_k\in C$.
But by Lemma~\ref{lem:separation-1}, this implies there exists an infinite path $p\in [T_A]$ such that $\Phi(p)=c_{k} \in A$, contradicting the fact that $c_k\in C$.
\end{itemize}	
\medskip
In sum: since any $\Phi$ witnessing $\EDfin[1]\glt\Fin$ leads to a contradiction, conclude $\EDfin[1]\not\glt\Fin$.
\end{proof}

For the reader who has understood thus far, the other cases are handled similarly as above, modulo some technical adjustments. Remarkably, the technique introduced in Section \ref{sec:separation-technique} is sufficiently powerful to separate a wide range of examples.

\begin{theorem}\label{thm:EDfin-m+1} $\EDfin[m+1]\not\glt\EDfin[m]$ for any $m\in\w$.
\end{theorem}
\begin{proof} Suppose for contradiction there exists $\Phi$ witnessing $\EDfin[m+1]\glt\EDfin[m]$, and consider the labelling function $\nu:=\nu^{\EDfin[m]}_{\Phi}$. If the root of $T_\Phi$ is $\nu$-labelled by $c\in\w$, we get a contradiction by the same argument in Theorem~\ref{thm:EDfin}. We therefore only treat the non-trivial case, as below.
		\medskip
		
\noindent\underline{Setup}. Suppose the root is $\nu$-labelled by $\bot$. Enumerate the (non-empty) set of critical nodes $\{\tau_i\}_{i\in \omega}$, and fix the disjoint interval partition $\{I_{a}\}_{a\in\w}$ as in the definitions of $\EDfin[m],\EDfin[m+1]$.
\medskip

\noindent \underline{Goal}.
We construct two sequences 
$$\{a_k\}_{k\in\w} \qquad\text{and}\qquad \{t^k_r\}_{k\in\w, r\leq m}\;,$$
where at each stage $k=\langle i,b\rangle$, $t_k$ denotes a choice of $(m+1)$-many successors of $\tau_i$ at stage $k$, say
$$\{\tau_i\fr t^k_0,\dots, \tau_i\fr t^k_r,\dots, \tau_i\fr t^k_m\},\qquad\quad 0\leq r\leq m \;,$$
whose labels are denoted $c^k_r:=\nu(\tau_i\fr t^k_r)$. Our aim is to ensure that the set of all labels, denoted 
$$C:=\{ c^k_r \mid r\leq m, \; k\in\w\}\;,$$
lies in $\EDfin[m+1]$. After which, we set $A:=\w\setminus C \in (\EDfin[m+1])^*$ and derive a contradiction.


\medskip
 
\noindent \underline{Requirements}.
The constructions of $\{a_k\}$ and $\{t^k_r\}$ are arranged so that for all $k=\langle i,b\rangle$:
\begin{enumerate}[label=(\alph*)]
\item ({\em Monotonicity}). $\{a_k\}_{k\in\omega}$ is strictly increasing.
\item ({\em Well-labelled}). No successors are labelled $\bot$; in particular, all labels lie in $\w$. 
\item ({\em Smallness}). Each $c^k_r$ avoids the union of previous intervals, i.e.
$$c_{k_r}\notin \bigcup_{a\le a_{k-1}}I_a.$$
\item ({\em Positivity}). All $m+1$ chosen
successors $t^k_0,\dots,t^k_m$ lie in a single interval $I_{u_k}$ with $u_k\ge b$.

\end{enumerate}

\medskip

\noindent \underline{Diagonalisation}. Build $\{a_k\}_{k\in\w}$ inductively.
Suppose $a_0,\dots , a_{k-1}$ have been chosen, and define
$$D_k:= \bigcup_{a\le a_{k-1}} I_{a}\cup\{\bot\}\,.$$
The stage $k=\langle i,b\rangle$ of the diagonalisation requires choosing a set of $m+1$ successors of the critical node $\tau_i$. Criticality means that any consensus amongst the successor nodes is negligible. Explicitly, in our context, this means $\{n\mid \nu(\tau_i\!\fr n)=d\}\in\EDfin[m]$ for every $d\in D_k$. Since $D_k$ is finite, we can extract two useful lower bounds on the indices of $\{I_a\}$:
\begin{itemize}
	\item[$\diamond$] For sufficiently large $s_0$, we obtain the inequality
	$$\Big|\{n \in I_u \bmid \nu(\tau_i\!\fr n)\in D_k\}\Big|\leq m\big|D_k\big |\;,\qquad\text{for all}\, u\geq s_0.$$ 
	\item[$\diamond$] For sufficiently large $s_1$, we obtain the inequality
		$$m\big|D_k\big|+m+1 \leq \Big|I_u\Big| \;,\qquad\text{for all}\, u\geq s_1.$$ 
\end{itemize}
Hence, set the uniform lower bound $$s_2:=\max\{s_0,s_1,b\},$$ 
where $b$ is the second entry of the current inductive stage $k$, and pick an interval $I_{u_k}$ where $u_k\geq s_2$. We may therefore choose $m+1$ distinct elements $$t^k_0,\dots,t^k_m\in I_{u_k}$$
 such that $c^k_r:=\nu(\tau_i\fr t^k_r)\notin D_k$ for each $r\in\{0,\dots, m\}$.\footnote{Both bounds $s_0,s_1$ are essential here. If only $u_k\geq s_0$, then the interval $I_{u_k}$ may not be large enough to contain $m+1$ elements whose labels avoid $D_k$. If only $u_k\geq s_1$, then although $I_{u_k}$ is large, we lose control over  the distribution of labels within $I_{u_k}$; it may be that every element in $I_{u_k}$ has a label in $D_k$. The additional requirement $u_k\geq b$ is only used later to ensure positivity.} 
 
 The fact that $c^k_r\notin D_{k}$ has strong implications for how these labels are distributed. Since $c^k_r\neq\bot$, each label lies in {\em some} interval $I_a\subseteq \w$. Further, each $c^k_r$ avoids all previous intervals up to $I_{a_{k-1}}$, and so
 $$ c^k_r\in I_a\qquad\qquad\text{for some (unique)}\, a>a_{k-1}\;.$$ 
 Take the maximum value of these $m+1$ indices
 $$a_k:=\max\{a \mid \exists r\leq m \;\text{such that}\; c_{k_r}\in I_a\}.$$

 \smallskip
 Reviewing the requirements, our construction guarantees monotonicity, smallness, positivity and well-labelled critical successors at every stage.
 
 \medskip
\noindent \underline{Verification}.  At every stage $k$, the newly-chosen labels $c^k_r$ satisfy
$$ c^k_r\notin \bigcup_{a\le a_{k-1}} I_a\qquad\text{and}\qquad c^k_r\in \bigcup_{a\le a_{k}} I_a\;,$$
and thus lie in intervals $I_a$ such that $a_{k-1}<a\leq a_k$. This is our monotonicity, smallness, and well-labelling requirement. Informally, any fixed interval $I_a$ is only given one opportunity to receive new labels, e.g. at some stage $k$; after which, $I_a$ will be consistently excluded from consideration. Since each stage produces at most $m+1$ distinct labels, deduce that the total set of labels
$$C:=\{ c^k_r \mid r\leq m, \; k\in\w\}\;,$$
intersects at most $m+1$ many elements in each interval $I_a$. In other words, $C\in\EDfin[m+1]$, so if we take $A:=\w\setminus C$, then $A\in (\EDfin[m+1])^*$.

\medskip

\noindent \underline{Contradiction}. By the assumption $\EDfin[m+1]\glt\EDfin[m]$, there exists an 
$(\EDfin[m])^*$-branching tree $T_A$ with $[T_A]\subseteq\dom(\Phi)$ and 
$\Phi[T_A]\subseteq A$.  
By Lemma~\ref{lem:separation-1}, $T_A$ has a critical node, say $\tau_i$. Since $T_A$ is $(\EDfin[m])^*$-branching, this means
\[
\{t\mid \tau_i\fr t\in T_A\}\in(\EDfin[m])^*.
\]
Thus, there exists a lower bound $s$ such that: 
\[
|\,I_u\cap\{t\mid \tau_i\fr t\in T_A\}\,|\ \ge\ |I_u|-m \;, \qquad\text{for all}\, u\geq s.
\]
Or, equivalently, any sufficiently large interval $I_u$ contains at most $m$-many elements $t$ whereby $\tau_i\fr t\notin T_A$.

\smallskip

For $k=\langle i,s\rangle$, by positivity, we have $u_k\ge s$.
Then, for at least one of the $m+1$ chosen successors in $I_{u_k}$, we have $\tau_i\fr t^k_r\in T_A$ (for some $r\leq m$). Applying Lemma~\ref{lem:separation-1} to $c^k_r=\nu(\tau_i\fr t^k_r)$, there exists $p\in [T_A]$ such that $\Phi(p)=c^k_r\in A$, contradicting $c^k_r\in C$. This completes the proof.
\end{proof}

\begin{theorem}\label{thm:Sumf} Let $\mathbb{O}$ be the set of all functions $f\colon\w\to\Q_+$ such that $$\displaystyle\lim_{n\to\infty}f(n)=0\qquad \text{and}\qquad f(n+1)\geq f(n)\,\,\,\text{for all}\; n\in\w.$$
\underline{Then}, $\Sumf\not\glt\EDfin$ for any $f\in\mathbb{O}$. 
\end{theorem}
\begin{proof} By now, we know the way forward. It remains to check that the correct decisions are made at each stage. Suppose for contradiction there exists $\Phi$ witnessing $\Sumf\glt \EDfin,$ and consider the labelling function $\nu:=\nu_\Phi^{\EDfin}$. If the root is $\nu$-labelled by $c\in\w$, the same argument as in Theorem~\ref{thm:EDfin} gives a contradiction. Hence, we only treat the non-trivial case, as below.

\medskip

\noindent\underline{Setup}. Suppose the root is $\nu$-labelled by $\bot$. Enumerate the (non-empty) set of critical nodes $\{\tau_i\}_{i\in\omega}$. Unpacking definitions, if $\tau_i$ is a critical node, then for every label $c\in\w\cup\{\bot\}$, there exists $m_c$ such that $\{n\mid \nu(\tau_i \fr n)=c\}\in \EDfin[m_c]$. Finally, fix the function $f\in\mathbb{O}$ as in $\Sumf$, and the disjoint interval partition $\{I_{a}\}_{a\in\w}$ as in the definition of $\EDfin$.

\medskip 
\noindent \underline{Goal and requirements}.
Now every $k\in\omega$ is of the form $k=\langle i,b,l\rangle$.
We consider such a triple to handle the following data:
\begin{enumerate}
	\item  $i$ indicates the choice of critical node $\tau_i$;
	\item $b$ indicates the current bound for handling eventual finiteness;
	\item  $l$ indicates that we must produce $(l+1)$-many successors of $\tau_i$, say
	$$\{\tau_i\fr t^k_0,\dots, \tau_i\fr t^k_r,\dots, \tau_i\fr t^k_l\},\qquad\quad 0\leq r\leq l \;,$$
	with labels $c^k_r:=\nu(\tau_i\fr t^k_r)$ for each $r\leq l$.
\end{enumerate}

We will of course want to impose additional conditions for a sequence
$$\{t^k_r\}_{r\leq l} \;.$$
At each stage $k=\langle i,b,l\rangle$, our objective is to fulfill the following requirements: 
\begin{enumerate}[label=(\alph*)]
	\item ({\em Well-labelled}). No successors are labelled $\bot$; in particular, all labels lie in $\w$. 
	\item ({\em $\Sumf$-smallness}). We have the inequality: 
	$$\sum_{r\leq l}f(c^k_r)\leq \frac{1}{2^k}\quad.$$
	\item ({\em ${\rm ED}^*_{\rm fin}$-positivity}). All $l+1$ chosen successors  $t^k_0,\dots,t^k_l$ lie in a single interval $I_{u_k}$ with $u_k\ge b$.
\end{enumerate}
Two substantive adjustments worth highlighting. First, the {\em smallness} condition looks different from the one in Theorem~\ref{thm:EDfin} or \ref{thm:EDfin-m+1}: this reflects the fact that asymptotic behaviour of sets in $\EDfin$ is measured by the interval partition $\{I_a\}_{a\in\w}$ whereas in $\Sumf$, it is measured by the {\em sum-of-$f$-values}. Second, the {\em positivity} condition now has an additional parameter. As in earlier arguments, every critical node $\tau_i$ must be treated infinitely often, but we now also need to supply a set of $(l+1)$-many successors for every arity $l+1$. Hence, diagonalisation scheme will have to adapted to account for the third parameter; otherwise, the overall shape of the argument remains the same.

\medskip 
\noindent \underline{Diagonalisation}. We build $\{t^k_r\}$ inductively, arranging for our diagonalisation to visit each critical node $\tau_i$ infinitely often and account for all successor arities $l+1$. At stage $k=\langle i,b,l\rangle$ of the diagonalisation,  we must pick $l+1$ successors $\{t^k_{r}\}$ of the critical node $\tau_i$, all lying a suitably large interval $I_{u_k}$. Since $\lim f = 0$, calculate $x_k$ such that
$$f(x_k)\leq \frac{1}{(l+1) 2^k}.$$
Suppose $d<x_k$ or $d=\bot$. Since $\tau_i$ is critical, we have $\{n\mid\nu(\tau_i\fr n)=d\}\in{\rm ED}_{\rm fin}[m_d]$ for some $m_d$; hence, there exists some sufficiently large $s_d$ such that for any $u\geq s_d$, $\{n\in I_u\mid \nu(\tau_i\fr n)=d\}$ has at most $m_d$  elements. Extending this, we obtain a uniform lower bound $s=\max\{s_d\mid d<x_k\mbox{ or }d=\bot\}$ so that $\{n\in I_u \mid \nu(\tau_i\fr n)<x_k\;\text{or}\; \nu (\tau_i\fr n )=\bot\}$ has at most $m:=\sum_{d<x_k}m_d + m_\bot$ elements for any $u\geq s$.

Now pick some sufficiently large $u_k$ such that $u_k\geq \max\{s,b\}$ (where $b$ is the second entry of the current inductive stage $k$) and $|I_{u_k}|\geq m+l+1$. We  may therefore choose $(l+1)$-many elements
$$t^k_0,\dots, t^k_l\in I_{u_k}$$
such that $c^k_r:=\nu(\tau_i\fr t^k_r)\geq x_k$ for each $r\leq l$. Since $f$ is non-increasing, compute
$$\sum_{r\leq l} f(c^k_r)\leq (l+1)f(x_k)\leq \frac{l+1}{(l+1)2^k}=\frac{1}{2^k}\quad. $$
\smallskip

Reviewing the requirements, we have clearly handled well-labelling and smallness. Positivity follows from our diagonalisation scheme, and the index of our chosen interval $I_{u_k}$ satisfying $u_k\geq \max\{s,b\}$.

\medskip 
\noindent \underline{Verification}.  Define the set of all labels 
$$C:=\{c^k_r\mid k\in\w, r\leq l\}.$$
At every stage $k$, we have $(l+1)$-many labels $\{c^k_r\}_{r\leq l}$ such that $\sum_{r\leq l} f(c^k_r)\leq \frac{1}{2^k}$.
This means 
$$\sum_{n \in C} f(n)=\sum_{k\in\omega}\sum_{r\leq l}f(c^k_r)\leq \sum_{k\in\omega}\frac{1}{2^k}=2<\infty\;.$$ 
In other words, $C\in \Sumf$, so if we take $A:=\w\setminus C$, then $A\in (\Sumf)^*$.

\medskip 
\noindent \underline{Contradiction}. By the assumption $\Sumf\glt \EDfin$, there exists an $(\EDfin)^*$-branching tree $T_A$ with $[T_A]\subseteq \dom(\Phi)$ and $\Phi[T_A]\subseteq A$. By Lemma~\ref{lem:separation-2}, $T_A$ has a critical node, say $\tau_i$. Since $T_A$ is $(\EDfin)^*$-branching, this means there exists $b_0,l_0$ such that for any $u\geq b_0$, we have $\tau_i\fr t\notin T_A$ for at most $l_0$-many $t\in I_u$.
For $k=\langle i,b_0,l_0\rangle$, by positivity, we have $u_k\geq b_0$. 
Then, for at least one of the $l+1$ chosen successors in $I_{u_k}$, we have  
$\tau_i\fr t^k_r\in T_A$ (for some $r\leq l$).
Applying Lemma~\ref{lem:separation-1} to $c^k_r=\nu(\tau_i\fr t^k_r)$, there exists $p\in [T_A]$ such that $\Phi(p)=c^k_r\in A$, contradicting $c^k_r\in C$. This completes the proof.
\end{proof}

We now assemble all the pieces of the puzzle to establish the main result of the section.

\begin{theorem}\label{thm:strict-ascend} There exists an infinite strictly ascending chain
	$$\Fin \sglt \EDfin[1]\sglt \dots \sglt \EDfin[m]\sglt \EDfin[m+1]\sglt \dots \sglt \EDfin\sglt \Sumn\,$$
within the Gamified Kat\v{e}tov order on lower sets over $\w$.	
\end{theorem}
\begin{proof} We divide the proof into two stages.
	
\subsubsection*{Step 1: Non-strict chain} It is clear that 
	$$\Fin \subseteq\EDfin[1]\subseteq \dots \subseteq \EDfin[m] \subseteq \EDfin[m+1]\subseteq \dots \subseteq \EDfin\,,$$
which therefore implies
		$$\Fin \lk \EDfin[1]\lk \dots \lk \EDfin[m] \lk\EDfin[m+1]\lk \dots \lk \EDfin\,,$$
in the classical Kat\v{e}tov order. Applying Theorem~\ref{thm:fubini-GTK}, this in turn implies 
		$$\Fin \glt \EDfin[1]\glt \dots \glt \EDfin[m]\glt \EDfin[m+1]\glt \dots \glt \EDfin\;.$$
		
It remains to show $\EDfin\lk\Sumn$. Define the map
\begin{align*}
h\colon \w&\longrightarrow \w\\
2^r+s & \longmapsto \lranglet{2^r}{s} \qquad \text{for}\; s<2^r\;,
\end{align*}
where $\lranglet{-}{-}\colon \Delta\to\w$ enumerates each interval $I_a$ in increasing order (cf. Remark~\ref{rem:base-EDfin}). In which case, for any $A\subseteq \w$ ,
 $$h^{-1}[A]= \Big\{\, 2^r +s \Bmid  \lranglet{2^r}{s}\in A\cap I_{2^r} \Big\}.$$
Thus, $h^{-1}[A]$ records only the data of $A$ on the subcollection $\{I_{2^r}\}_{r<\w}$, and discards the rest. 

Now suppose $A\in\EDfin$. Then there exists $m,l\in\w$ such that $|A\cap I_u|\leq m$ for all $u\geq l$. Use the lower bound $l$ to split the harmonic sum into two terms, as below:
$$
\sum_{n\in h^{-1}[A]} \frac1n
= \underbrace{\sum_{2^r<l} \sum_{\substack{s<2^r\\ \langle 2^r,s\rangle\in A}} \frac1{2^r+s}}_{<\infty}
\;+\;
\underbrace{\sum_{2^r\ge l} \sum_{\substack{s<2^r\\ \langle 2^r,s\rangle\in A}}
\frac1{2^r+s}}_{<\infty}\,\,.
$$
The first term is clearly finite since it is a finite sum (of finite values). For the second term, notice for $2^r\geq l$, we have $|A\cap I_{2^r}|\leq m$, and thus
$$\sum_{\substack{s<2^r\\ \langle 2^r,s\rangle\in A}} \frac1{2^r+s}
\le \sum_{i=0}^{m-1} \frac1{2^r+i}
\le \frac{m}{2^r}\,.$$
Summing over all $2^r\geq l$ shows the second term is finite too:  
$$\sum_{2^r\ge l} \sum_{\substack{s<2^r\\ \langle 2^r,s\rangle\in A}}
\frac1{2^r+s} \leq  
 m\cdot \sum_{r} \frac{1}{2^r}<\infty\,.$$
In sum: our computation shows
$$\sum_{n\in h^{-1}[A]} \frac1n
\;\;= \;\;\sum_{2^r<l} \;\;\cdots\;\;
\;\;+\;\;
\sum_{2^r\ge l} \;\;\cdots\;\;
<\infty\,,$$
and so $h^{-1}[A]\in\Sumn$ for any $A\in\EDfin$. Hence, $\EDfin\lk\Sumn$. Applying Theorem~\ref{thm:fubini-GTK} once more, we obtain the full chain:
$$\Fin \glt \EDfin[1]\glt \dots \glt \EDfin[m]\glt \EDfin[m+1]\glt \dots \glt \EDfin\glt \Sumn\;.$$

\subsubsection*{Step 2: Strictness.} To upgrade the inequalities in Step 1 to strict inequalities, we rule out $\glt$-equivalence. Most are immediate by quoting previous results:
\begin{itemize}
	\item By Theorem~\ref{thm:EDfin}, $\EDfin[1]\not\glt\Fin$. Hence, $\Fin\sglt\EDfin[1]$.
	\item By Theorem~\ref{thm:EDfin-m+1}, $\EDfin[m+1]\not\glt\EDfin[m]$. Hence, $\EDfin[m]\sglt\EDfin[m+1]$. 
	\item By Theorem~\ref{thm:Sumf}, $\Sumn\not\glt \EDfin$. Hence, $\EDfin\sglt\Sumn$.
\end{itemize}
It remains to show 
$$\EDfin[m]\sglt \EDfin\qquad\text{for any}\, m\in\w.$$
But this follows from the proof of Theorem~\ref{thm:EDfin-m+1}. In broad strokes, the argument starts by assuming (for contradiction) there exists some $\Phi$ witnessing $\EDfin[m+1]\glt\EDfin[m]$. Then, we construct a special set $C\in\EDfin[m+1]$, before taking its complement to get $\w\setminus C\in (\EDfin[m+1])^*$. Finally, applying the separation lemmas from Section~\ref{sec:separation-technique}, we verify that there cannot exist an $(\EDfin[m])^*$-branching tree $T_{\w\setminus C}$ such that $\Phi (T_{\w\setminus C})\subseteq \w\setminus C$, contradicting our assumption. Read in the present context: since $\EDfin[m+1]\subseteq \EDfin$ by definition, the same $C$ also shows that $\EDfin\not\glt\EDfin[m]$ for all $m\in\w$.\footnote{Here is an alternative argument. Suppose for contradiction $\EDfin\glt \EDfin[m]$ for some $m\geq 1$. Since $\EDfin[m]\glt \EDfin[m+1]\glt \EDfin$, this implies $\EDfin[m]\elt \EDfin[m+1]\elt \EDfin$, contradicting $\EDfin[m]\sglt \EDfin[m+1]$ from Theorem~\ref{thm:EDfin-m+1}. However, notice this argument implicitly assumes that $\glt$ is a preorder, which we have not yet shown.}
\end{proof}

\begin{conclusion}\label{conc:separation} The structure of the Gamified Kat\v{e}tov order is both coarse and subtle. Although it collapses many of the finer distinctions between ideals (Conclusion~\ref{conc:MAD}), it is still sufficiently sensitive to separate infinitely many classes.
\end{conclusion}

Conclusion~\ref{conc:separation} leads one to look for deeper explanations regarding what $\glt$ regards as significant vs. insignificant differences in ideal structure. At the present moment, we are already at the frontier, and it is not yet clear where the path may take us. One potentially illuminating perspective comes from our work in Section~\ref{sec:brave-new-order}, which points out the connection to topos theory. 


\begin{corollary}\label{cor:strict-clt} There exists an infinite strictly ascending chain
		$$\Fin \sclt \EDfin[1]\sclt \dots \sclt \EDfin[m]\sclt \EDfin[m+1]\sclt \dots \sclt \EDfin\sclt \Sumn\,$$
within the $\clt$-order on lower sets over $\w$. 	
\end{corollary}
\begin{proof} Review the argument of Theorem~\ref{thm:strict-ascend}. 
\begin{itemize}
	\item 	\textbf{Step 1} constructs the non-strict chain of ideals in the classical Kat\v{e}tov order $\lk$, before applying Theorem~\ref{thm:fubini-GTK} to obtain the non-strict chain within $\glt$ . By inspection, all the functions $h\colon\w\to\w$ witnessing Kat\v{e}tov dominance are computable. Hence, the same non-strict chain of ideals exists within $\clt$ as well. 
	\item 	\textbf{Step 2} upgrades this to a strict chain of ideals in $\glt$ by showing that $\calH\not\glt \calI$ for each consecutive pair of ideals. Clearly, $\glt$ is coarser than $\clt$, and so $\calH\not\glt \calI$ implies $\calH\not\clt \calI$ for all ideals $\calH,\calI\subseteq\Pw$. Hence, we obtain strictness in $\clt$ essentially for free.
\end{itemize}
\end{proof}

\begin{discussion}\label{dis:separation-takeaway}This section represents the technical core of the paper and showcases the strength of the approach developed here. Beyond the specific results obtained, it opens up several conceptually rich avenues for further investigation.  We mention two particularly interesting points of contact with the literature:
\begin{itemize}
\item {\em Idealised Subtoposes.}
Restated in the language of \cite{LvO13}, Corollary~\ref{cor:strict-clt} says there exists an infinite strictly descending\footnote{Note that the order of LT topologies and the inclusion order of subtoposes are reversed.} chain of basic subtoposes of $\Eff$.
This raises a natural and compelling question: what does the internal logic and mathematics of these idealised effective subtoposes look like? What exactly is {\em realisability relative to an ideal}?

\item {\em Extended Weihrauch reducibility.} Bauer \cite{Bau22} introduced the notion of extended Weihrauch reducibility, in a setting far removed from combinatorial set theory. 
Bauer's framework brings into focus a new class of degrees, {\it non-modest degrees}, which lie outside the standard Weihrauch lattice and whose structure remains mysterious. Subsequently, the first author observed that the LT degrees of basic topologies (in the sense of Convention~\ref{con:orders}) give rise to non-modest degrees, and highlighted its analysis as an important open problem \cite[Question 1]{Kih23}.

Our work in this paper addresses this problem on two major fronts. First, Theorem~\ref{thm:GTK} shows that {\em basic} non-modest degrees are tightly connected to the Kat\v{e}tov order on upper sets, an unexpected connection between extended Weihrauch complexity and classical filter combinatorics. Second, Corollary~\ref{cor:strict-clt} establishes strong separation results within this class of degrees, and its proof strategy suggests a robust framework for obtaining further separation results in this setting.
\end{itemize}
Taken together, these interactions between the combinatorial and the computable bring into focus a wide range of questions — some potentially deep and subtle — which we discuss more fully in Section~\ref{sec:qn-comp-th-cons-log}.

%
%
%
\end{discussion}

\section{Combinatorial Complexity via Turing Degrees}\label{sec:Turing}



Here we extend the Gamified Kat\v{e}tov order from upper sets to upper sequences. The core ideas remain the same as in Section~\ref{sec:brave-new-order}, except we now allow the relevant constraints to ``evolve over time'' (as opposed to being fixed by a single upper set $\calU$). 

Section~\ref{sec:upper-seq} gives both the tree-based and the game-based definition, before showing their equivalence. Section~\ref{sec:GTK-preorder} further develops this perspective and delivers on our promise in Section~\ref{sec:brave-new-order} by showing that $\glt$ defines a pre-order. Interestingly, this follows from a new structure theorem we prove: the Gamified Kat\v{e}tov order is equivalent to the $\clt$-order relativised to an arbitrary Turing oracle. These results lay the groundwork for Sections~\ref{sec:cofinal-thm} and ~\ref{sec:generalised-Phoa}, where we isolate some surprising interactions between the combinatorial and computable within the $\clt$-order.



\subsection{Extended Gamified Kat\v{e}tov Order}\label{sec:upper-seq}
We remind the reader: an {\em upper sequence on $\w$} is a partial sequence $$\overline{\mathcal U}=(\mathcal U_n)_{n\in I}$$ of upper sets, where $I\subseteq\omega$. 
\subsubsection*{Tree-based definition}
We start by re-examining Definition~\ref{def:GTK}. To extend this version of the Gamified Kat\v{e}tov order to upper sequences, we first need to define trees with ``good'' branching behaviour. In the original setting, a single upper set $\calU$ imposes a uniform branching requirement across all nodes in a tree. By contrast, an upper sequence $\calUS=(\calU_n)_{n\in I}$
may assign different constraints at different nodes. To keep track of which upper set applies where, we introduce a partial ``stage'' function  
$$\gamma\pcolon \lww\to \w\;,$$
where $\gamma(\sigma)$ identifes the upper set $\calU_{\gamma(\sigma)}$ governing the immediate successors of $\sigma\in T$. More formally:
	
\begin{definition}\label{def:US-branching} Let $\calUS=(\calU_n)_{n\in I}$ be an upper sequence on $\w$, and let $\gamma\pcolon \lww\to\w$ be a partial function. A non-empty tree $T\subseteq \lww$ is {\em $(\calUS,\gamma)$-branching} if, for every node $\sigma\in T$
$$\gamma(\sigma)\;\text{ is defined and}\; \{n\in\w\mid \sigma\fr n\in T\}\in \calU_{\gamma(\sigma)}.$$ 
{\em Convention.} If $\calUS=(\calU)_{n\in\w}$ is a constant upper sequence, we revert to the original notion of a $\calU$-branching tree instead of explicitly defining a stage function $\gamma$ assigning $\calU$ to every node. 
\end{definition}	

In Definition~\ref{def:GTK}, one good tree $T_A$ was required for each $A\in\calU$. For upper sequences $\calUS=(\calU_n)_{n\in I}$, we must now provide one good tree for each pair \((n,A)\) with \(n\in I\) and \(A\in\mathcal U_n\). These trees are handled uniformly by a single global witness \(\Phi\) via level-wise reindexing. 

\begin{definition}[Tree-based definition]\label{def:GTK-full} 
	Let $\calUS,\calVS$ be upper sequences on $\w$, where $\calUS=(\calU_n)_{n\in I}$. Write
	$$\calUS\glt \calVS$$
	if there exists partial functions $\gamma\pcolon \lww \to \w$ and $\Phi\pcolon\ww\to \w$ satisfying the condition:
	\begin{itemize}
		\item[($\star$)] For every $n\in I$ and every $A\in\calU_n$, there exists a $(\calVS,\gamma^{[n]})$-branching tree  $$T^n_A\subseteq\lww$$ such that $[T^n_A]\subseteq \dom(\Phi^{[n]})$ and $\Phi^{[n]}[T^n_A]\subseteq A$, whereby 
		$$ \gamma^{[n]}(\sigma):=\gamma(n\fr\sigma),\;\qquad \Phi^{[n]}(p):=\Phi(n\fr p)\;.$$
	\end{itemize}	        
	This defines the {\em Gamified Kat\v{e}tov order} on upper sequences. If, in addition, $(\Phi,\gamma)$ is required to be computable, then this defines the {\em computable Gamified Kat\v{e}tov order} on upper sequences.
\end{definition}

\begin{convention} In the context of Definition~\ref{def:GTK-full}, we call $\gamma$ a {\em stage function}, $\Phi$ a {\em witness function}, and $T^n_A$ a {\em witness tree} for $A\in\calU_n$.
\end{convention}

\subsubsection*{Game-based definition}


An alternative way of introducing the Gamified Kat\v{e}tov order on upper sequences is simply to omit the computability constraint from the LT-game in Definition \ref{def:bilayer-game}.

\begin{definition}[Game-based definition]\label{def:game-based-def}
	The {\bf Oracle LT-Game} $\mathfrak G(\calUS,\calVS)$ is exactly the
	$\mathfrak G(f_{\calUS},f_{\calVS})$ (Definition~\ref{def:bilayer-game}),
	with the sole modification that $\art$ is \emph{no longer required} to use a computable strategy.
We say there exists a {\em winning strategy} for $\mathfrak G(\calUS,\calVS)$ whenever the
$\art$-$\nim$ team has a (possibly non-computable) winning strategy for $\mathfrak G(f_{\calUS},f_{\calVS})$.
\end{definition}

The following structure theorem establishes the expected picture. 

\begin{theorem}\label{thm:GTK-full} The following is true:
\begin{enumerate}[label=(\roman*)]
	\item  Let $\calUS,\calVS$ be upper sequences on $\w$. Then $$\calUS\glt\calVS$$ 
	iff there exists a winning strategy for the Oracle LT-game $\mathfrak{G}(\calUS,\calVS)$.
	\item The computable Gamified Kat\v{e}tov order on upper sequences is isomorphic to the $\clt$ order on LT topologies.
\end{enumerate}	  
\end{theorem}
\begin{proof} The main ideas have already been presented in Section~\ref{sec:brave-new-order}, so we will be brief.
\begin{enumerate}[label=(\roman*):]	
	\item The proof is analogous to Theorem~\ref{thm:game-GTK}, modulo some additional re-indexing.

Suppose $\art$-$\nim$ have a winning strategy for $\mathfrak{G}(\calUS,\calVS)$. In the opening move, $\mer$ selects some $n\in I$ and $A\in\calU_n$. At each stage, $\art$ reacts by either: 
\setlist[enumerate]{leftmargin=1.3cm}
\begin{enumerate}[label=(\arabic*)]
\item making a public query to the oracle $\calVS$, or
\item declaring an answer and terminating the game.
\end{enumerate}
Adapting Step~1 of the proof of Theorem~\ref{thm:game-GTK}, this may be reformulated as:
\begin{enumerate}[label=(\arabic*)]
    \item a partial stage function $\gamma^{[n]}\pcolon \lww\to\w$; and
    \item a partial continuous function $\Phi^{[n]}\pcolon\lww \to \w,$
\end{enumerate}
respectively. Likewise, following the same ``active vs. lazy'' construction in Step 2, one can reformulate $\nim$'s strategy as a $(\calVS,\gamma^{[n]})$-branching tree $T^n_A$.  Step 3 then applies verbatim to show $[T^n_A]\subseteq \dom(\Phi^{[n]})$ and $\Phi^{[n]}[T^n_A]\subseteq A$. 

Conversely, suppose $(\Phi,\gamma)$ witnesses $\calUS\glt\calVS$. For any $A\in\calU_n$ chosen by $\mer$, there exists a $(\calVS,\gamma^{[n]})$-branching tree $T^n_{A}$ such that $\Phi^{[n]}[T_A^n]\subseteq A$. This tree captures the set of all plays that $\art$-$\nim$ must be prepared to answer: each node represents a finite sequence of $\mer$'s moves, and the $(\calVS,\gamma^{[n]})$-branching condition ensures that a legal response exists at every stage. 

Notice: along any infinite path $p\in[T^n_A]$, $\Phi^{[n]}$ eventually defines an output in $A$.  Thus $\art$-$\nim$ obtain a winning strategy simply by following the finite initial segment of $p$ on which $\Phi^{[n]}$ becomes defined and replying according to the allowed successors in $T^n_A$. This is exactly the reverse construction of Steps~1-2 in Theorem~\ref{thm:game-GTK}, and we are done.
	\item From the proof of (i), $\art$'s strategy corresponds to a stage function $\gamma^{[n]}$ and a witness function $\Phi^{[n]}$. Consequently, the computable Gamified Kat\v{e}tov order on upper sequences is isomorphic to the induced $\clt$ order on upper sequences. Since upward closure does not affect the LT degree (Lemma~\ref{lem:upward-closure-basic-properties}), this is isomorphic to the original $\clt$ order. The rest follows from Kihara's Theorem~\ref{thm:kih-bil}.
\end{enumerate}	
	
\end{proof}
	
\begin{remark} Implicit within Definition~\ref{def:GTK-full} is the correct generalisation of a $\delta$-Fubini power to a $\delta$-Fubini {\em product}. We have left this out since it is not needed for the present paper, but the motivated reader (perhaps a curious set theorist) may wish to work out the details for themselves. In particular, Remark~\ref{rem:dob} points to interesting connections with Dobrinen's work in \cite[\S 3]{Dob20}.
\end{remark}

As Theorem~\ref{thm:GTK-full} illustrates, the game-based definition of the extended Gamified Kat\v{e}tov order highlights a direct connection with $\clt$ on LT topologies via Theorem~\ref{thm:kih-bil}.

\subsection{Gamified Kat\v{e}tov Order is Transitive}\label{sec:GTK-preorder}
In this section, we finally show that the Gamified Kat\v{e}tov order $\glt$ defines a preorder. Throughout, we work with upper sequences (i.e. not just upper sets) and adopt the game-based Definition~\ref{def:game-based-def} of the Gamified Kat\v{e}tov order.




%
%
\medskip

There are various operations defined on upper sets and multivalued functions: the Fubini product\footnote{This is studied as the compositional product (or the sequential composition) in the context of multivalued functions.}, the Cartesian product, and lattice operations (only for multivalued functions \cite{BGP21}).
Let us extend some of these notions to upper sequences.

\begin{definition}
Let $\overline{\U}=(\U_i)_{i\in I}$ and $\overline{\V}=(\V_j)_{j\in J}$ be upper sequences.
\begin{enumerate}
\item Their {\it Cartesian product} is defined by $\overline{\U}\times\overline{\V}=(\U_n\times\V_m)_{\langle n,m\rangle\in I\times J}$.
\item Their {\it sum} $\overline{\U}+\overline{\V}$ is indexed by $I+J$ and defined as 
\[
(\overline{\U}+\overline{\V})_{\langle i,n\rangle}=
\begin{cases}
\U_n&\mbox{if $i=0$ and $n\in I$,}\\
\V_n&\mbox{if $i=1$ and $n\in J$.}
\end{cases}
\]
\end{enumerate}

Here, we often identify a subset family with its upward closure.
\end{definition}

\begin{remark}
Clearly, the sum of constant sequences is not a constant sequence.
This is one reason we shift our discussion from upper sets to upper sequences.
\end{remark}

We denote by $\leq_{\rm cK}$ the computable Kat\v{e}tov reducibility for upper sequences (Definition \ref{def:Katetov-for-upper-seq}), which is finer than LT-reducibility $\leq_{\rm LT}$.

\begin{proposition}\label{}
The sum $\overline{\U}+\overline{\V}$ is the least upper bound of $\overline{\U}$ and $\overline{\V}$ in the LT-order.
\end{proposition}

\begin{proof}~
\begin{itemize}
\item {\it Upper bound.}
The pair of functions $\varphi(n)=\langle 0,n\rangle$ and $\psi(x)=x$ witnesses that $\overline{\U}\leq_{\rm cK}\overline{\U}+\overline{\V}$:
For any $A\in \U_n$, we have $A\in (\overline{\U}+\overline{\V})_{\varphi(n)}=\U_n$, and for any $x\in A$, we get $\psi(x)=x\in A$.

Similarly, the pair of functions $\varphi(n)=\langle 1,n\rangle$ and $\psi(x)=x$ witnesses that $\overline{\V}\leq_{\rm cK}\overline{\U}+\overline{\V}$.

\item {\it Least.}
Assume $\overline{\U},\overline{\V}\leq_{\rm LT}\overline{\W}$.
Let $(\varphi_0,\eta_0)$ be \art-\nim's winning strategy for $\mathfrak{G}(\overline{\U},\overline{\calW})$, and $(\varphi_1,\eta_1)$ be \art-\nim's winning strategy for $\mathfrak{G}(\overline{\V},\overline{\calW})$.
In the game $\mathfrak{G}(\overline{\U}+\overline{\V},\overline{\calW})$, if \mer's first move is $(i,n)$ then, for \art-\nim~to win, they just need to follow the strategy $(\varphi_i,\eta_i)$.
Formally, \art's strategy is $\varphi(\langle i,n\rangle,\sigma)=\varphi_i(n,\sigma)$, and \nim's strategy is $\eta(\langle i,n\rangle,\sigma)=\eta_i(n,\sigma)$.
One can easily check that this strategy witnesses $\overline{\U}+\overline{\V}\leq_{\rm LT}\overline{\calW}$.
\end{itemize}
\end{proof}

\begin{corollary}\label{fact:prod-upper-seq}
Let $\calUS,\calVS,\calWS$ be upper sequences.
\begin{enumerate}[label=(\roman*)]
     \item {\em (Associativity).} $\calUS+ (\calVS+ \calWS)\eclt (\calUS+\calVS)+ \calWS.$
    \item {\em (Monotonicity).} $\calUS\clt\calVS \implies \calUS+\calWS\clt \calVS+\calWS.$
\end{enumerate}
\end{corollary}

\begin{proof}
By the general property of the least upper bound.
\end{proof}


We say that an upper sequence $\overline{\U}=(\U_i)_{i\in I}$ is {\it pointed} if $I$ is nonempty.
Note that if $\calUS$ is not pointed then it is the least element in $\glt$ (and $\clt$).

\begin{proposition}\label{prop:sum-vs-product}
If upper sequences $\overline{\U}$ and $\overline{\V}$ are pointed, we have $\overline{\U}+\overline{\V}\leq_{\rm cK}\overline{\U}\times\overline{\V}$.
\end{proposition}

\begin{proof}
For $\overline{\U}=(\U_i)_{i\in I}$ and $\overline{\V}=(\V_i)_{j\in J}$, assume $i_0\in I$ and $j_0\in J$.
Then $\overline{\U}+\overline{\V}\leq_{\rm cK}\overline{\U}\times\overline{\V}$ is witnessed by $\varphi(0,n)=\langle n,j_0\rangle$, $\varphi(1,n)=\langle i_0,n\rangle$, and $\psi(x_0,x_1)=x_k$, where $k$ as in $\lranglet{k}{n}\in I+J$.

To see this, let $\langle k,n\rangle\in I+J$ and $A\in (\overline{\U}+\overline{\V})_{\langle k,n\rangle}$ be given.
If $k=0$, we have $A\in\U_n$.
Then there is some $A\times B\in(\U\times\V)_{\varphi(0,n)}=\U_n\times\V_{j_0}$.
Now, for any $\langle x_0,x_1\rangle\in A\times B$, we have $\psi(x_0,x_1)=x_0\in A$.
A similar argument applies to $k=1$.
This shows that $(\varphi,\psi)$ witness $\overline{\U}+\overline{\V}\leq_{\rm cK}\overline{\U}\times\overline{\V}$.
\end{proof}

\begin{observation}\label{obs:product-preserves-Turing}
Let $f,g\colon\w\to\w$ be (total, single-valued) functions. \underline{Then}, there exists a single-valued total function $h$ such that
    $$f\times g \eclt h .$$
    In other words, if $f$ and $g$ are Turing oracles, then so is $f\times g$.
\end{observation}
\begin{proof}
Recall that any $f\colon\w\to\w$ uniquely determines a sequence of elements\footnote{To obtain an honest upper sequence, one can of course represent $f$ as a sequence of principal ultrafilters generated by these elements, but the two representations are $\clt$-equivalent (cf.~Observation~\ref{obs:computable-Katetov-to-upseq}).}
$$\{f(n)\}_{n\in\w}\;.$$
 $f\times g$ is clearly well-defined. In fact, $f\times g$ is computably equivalent to the total single-valued function 
$$h(\lranglet{n}{m}):=\lranglet{f(n)}{g(m)}.$$
\end{proof}

\begin{theorem}\label{thm:preorder}
Let $\calUS,\calVS$ be upper sequences on $\w$, where $\calVS$ is pointed. \underline{Then}, $$\calUS\glt\calVS \iff \calUS\clt f+ \calVS\;,\quad\text{for some }\; f\colon\w\to\w\;.$$
\end{theorem}

\begin{proof} \hfill 

\begin{itemize}
\item[$\implies$:]
Assume $\overline{\U}\leq_{\rm LT}\overline{\V}$ via \art's (not necessarily computable) strategy $\varphi$ and \nim's strategy $\eta$.
A play of the game $\mathfrak{G}(\calUS,\calVS)$ proceeds as follows:
\begin{itemize}
\item[R0.] \mer\ plays the initial move $x_0\in I$ and $A\in\U_{x_0}$.
\item[] \art\ makes a query $y_0\in J$; then \nim\ chooses $B_0\in\V_{y_0}$.
\item[R1.] \mer\ responds with $x_1\in B_1$.
\item[] \art\ makes a query $y_1\in J$; then \nim\ chooses $B_1\in\V_{y_1}$.
\item[] $\vdots$
\item[R$k$.] \mer\ responds with $x_k\in B_{k-1}$.
\item[] \art\ answers $y_k\in A$, and declares the termination of the game.
\end{itemize}
Here, the strategies give the moves $\varphi(\langle x_0,x_1,\dots,x_n\rangle)=\langle i_n,y_n\rangle$ and $\eta(A,x_0,x_1,\dots,x_n)=B_n$, where $i_k\in\{\uparrow,\downarrow\}$ indicates whether to make a query or declare the termination (coded as $0$ and $1$, respectively).

Since $\varphi$ is not necessarily a computable strategy, we modify it to a computable strategy $\psi$ that queries $\varphi$.
Let $\hat{\varphi}$ be a total extension of $\varphi$, and consider the game $\mathfrak{G}(\calUS,\hat{\varphi}+\calVS)$.
The new strategy will proceed as follows:
\begin{itemize}
\item[R0.] \mer\ plays the initial move $x_0\in I$ and $A\in\U_{x_0}$.
\item[] \art\ makes a query $\langle x_0\rangle$ to $\hat{\varphi}$.
\item[R1.] \mer\ responds with $\hat{\varphi}(\langle x_0\rangle)=\langle \uparrow,y_0\rangle$.
\item[] \art\ makes a query $y_0\in J$ to $\calVS$; then \nim\ chooses $B_0\in\V_{y_0}$.
\item[R2.] \mer\ responds with $x_1\in B_1$.
\item[] \art\ makes a query $\langle x_0,x_1\rangle$ to $\hat{\varphi}$.
\item[R3.] \mer\ responds with $\hat{\varphi}(\langle x_0,x_1\rangle)=\langle \uparrow,y_1\rangle$.
\item[] \art\ makes a query $y_1\in J$ to $\calVS$; then \nim\ chooses $B_1\in\V_{y_0}$.
\item[] $\vdots$
\item[R$2k$.] \mer\ responds with $x_k\in B_{k-1}$.
\item[] \art\ makes a query $\langle x_0,x_1,\dots,x_k\rangle$ to $\hat{\varphi}$.
\item[R$2k+1$.] \mer\ responds with $\hat{\varphi}(\langle x_0,x_1,\dots,x_k\rangle)=\langle \downarrow,y_k\rangle$.
\item[] \art\ answers $y_k\in A$, and declares the termination of the game.
\end{itemize}
Here, the above is based on an example of play where the original strategy $\varphi$ declares termination in some round $k>1$.
Formally, \art's new strategy $\psi$ of is given by the following:
\begin{align*}
&\psi(\langle x_0,p_0,x_1,p_1,\dots,x_n\rangle)=\langle \uparrow,\langle 0,x_0,x_1\dots,x_n\rangle\rangle\\
&\psi(\langle x_0,p_0,x_1,p_1,\dots,x_n,p_n\rangle)=
\begin{cases}
\langle\uparrow,\langle 1,y_n\rangle\rangle&\mbox{if }p_n=\langle\uparrow,y_n\rangle\\
\langle\downarrow,y_n\rangle&\mbox{if }p_n=\langle\downarrow,y_n\rangle
\end{cases}
\end{align*}

Clearly, $\psi$ is computable.
Similarly, \nim's strategy $\eta$ is adjusted to the following strategy $\theta$.
\[\theta(A,x_0,p_0,x_1,p_1,\dots,x_n,p_n)=\eta(A,x_0,x_1,\dots,x_n).\]

Here, \nim's moves in even rounds are trivial.
Confirm that $(\psi,\theta)$ yields a winning strategy:

Inductively, we show that the history $A,x_0,p_0,x_1,p_1,\dots,x_n$ of \mer's moves up to the $(2n)$th round in $\mathfrak{G}(\calUS,\hat{\varphi}+\calVS)$ against \art-\nim's strategy $(\psi,\theta)$ yields the history $A,x_0,x_1,\dots,x_n$ of \mer's moves up to the $n$th round in $\mathfrak{G}(\calUS,\calVS)$ against \art-\nim's strategy $(\varphi,\eta)$.
\begin{itemize}
\item In the $(2n)$th round of the game $\mathfrak{G}(\calUS,\hat{\varphi}+\calVS)$, \art\ makes a query $\langle x_0,x_1,\dots,x_n\rangle$ to $\hat{\varphi}$, while \nim\ does nothing.
\item 
\mer's response $p_n$ in the $(2n+1)$st round is $\varphi(x_0,\dots,x_n)$.
By induction, this is the query by \art's strategy $\varphi$ in the $(n+1)$st round of $\mathfrak{G}(\calUS,\calVS)$.
In the $(2n+1)$st round of $\mathfrak{G}(\calUS,\hat{\varphi}+\calVS)$, \art\ adopts this response as the next move:

If $p_n$ instructs querying $y_n$, \art\ interprets this as querying $y_n$ to $\overline{\calV}$; then \nim\ chooses $B_n\in\V_{y_n}$ according to strategy $\theta$.

The value $y_n$ is the query by \art's strategy $\varphi$, so $B_n\in\V_{y_n}$ is the choice of \nim's strategy $\eta$, in the $(n+1)$st round of $\mathfrak{G}(\calUS,\calVS)$.
\item
Therefore, \mer's response $x_{n+1}\in B_n$ in the $(2n+2)$nd round of $\mathfrak{G}(\calUS,\hat{\varphi}+\calVS)$ is the same as \mer's response in the $(n+1)$th round of $\mathfrak{G}(\calUS,\calVS)$.
Thus, the claim is verified.
\end{itemize}
Since $\varphi$ is \art's winning strategy in $\mathfrak{G}(\calUS,\calVS)$, it declares the termination at some round $k$ and outputs $y_k\in A$.
Then $\psi$ receives this information at round $2k+1$ and returns the same output $y_k\in A$.
This shows that $(\psi,\theta)$ is \art-\nim's winning strategy in $\mathfrak{G}(\calUS,\hat{\varphi}+\calVS)$.

\item[$\impliedby$:]
Let us assume $\calUS\clt f+\calVS$.
Since both $f$ and $\calVS$ are pointed, by Proposition \ref{prop:sum-vs-product}, we may assume $\calUS\clt f\times\calVS$.
In this case, there exists \art's computable strategy $\varphi$ in $\mathfrak{G}(\calUS,f\times\calVS)$.
We correct this strategy to \art's strategy $\psi$ in $\mathfrak{G}(\calUS,\calVS)$:
When the computable strategy $\varphi$ makes a query $\langle a,z\rangle$ to $f\times\overline{\V}$, the strategy $\psi$ (not necessarily computable) can freely use the value of $f(a)$ using $f$ as an oracle.
Thus, $\psi$ only needs to make the query $z$ to $\overline{\V}$.

Formally, for \art's move in the $n$th round of the original game
\[
\varphi(\langle x_0,\langle f(a_0),x_1\rangle,\langle f(a_1),x_2\rangle,\dots,\langle f(a_{n-1}),x_n\rangle\rangle)=y_n,
\]
\art's new strategy $\psi$ is defined as follows:
\[
\psi(\langle x_0,x_1,\dots,x_n\rangle)=
\begin{cases}
\langle\uparrow,z_n\rangle&\mbox{if }y_n=\langle\uparrow,\langle a_n,z_n\rangle\rangle,\\
\langle\downarrow,z_n\rangle&\mbox{if }y_n=\langle\downarrow,z_n\rangle.
\end{cases}
\]
Here, $y_n$ uniquely determines the value $a_n$, and since $f$ is single-valued, given \mer's play $x_0,x_1,\dots$, one can inductively see that the strategy $\varphi$ uniquely generates the sequence $a_0,a_1,\dots$.
Therefore, $\psi$ is well-defined.

Similarly, the following modification of \nim's strategy is well-defined.
\[
\theta(A,x_0,x_1,\dots,x_n)=\eta(A,x_0,\langle f(a_0),x_1\rangle,\langle f(a_1),x_2\rangle,\dots,\langle f(a_{n-1}),x_n\rangle).
\]

It is easy to see that $(\psi,\theta)$ is a winning strategy:
As mentioned above, the history $A,x_0,x_1,\dots,x_n$ of \mer's moves up to the $n$th round in $\mathfrak{G}(\calUS,\calVS)$ against $(\psi,\theta)$ uniquely determines the history $A,x_0,\langle f(a_0),x_1\rangle,\langle f(a_1),x_2\rangle,\dots,\langle f(a_{n-1}),x_n\rangle$ of \mer's moves up to the $n$th round in $\mathfrak{G}(\calUS,\hat{\varphi}\times\calVS)$ against $(\varphi,\eta)$. 
Since $\varphi$ is winning in the latter game, it declares the termination at some round $k$ and outputs $y_k\in A$.
Then $\psi$ receives this information at round $k$ and returns the same output $y_k\in A$.
This shows that $(\psi,\theta)$ is \art-\nim's winning strategy in $\mathfrak{G}(\calUS,\calVS)$.
\end{itemize}
\end{proof}

\begin{corollary}\label{cor:preorder}  $\glt$ defines a preorder on upper sequences (or indeed, upper sets).
\end{corollary}
\begin{proof} We need to show $\glt$ is reflexive and transitive.
\begin{itemize}
    \item {\em Reflexivity.} 
   Obviously, $\calUS\leq_{\rm cK}\calUS$, so $\calUS\glt\calUS$.

    \item {\em Transitivity.} Given any triple of upper sequences $\calUS,\calVS,\calWS$, suppose 
$$\calUS\glt \calVS\qquad\text{and}\qquad \calVS\glt\calWS\;.$$ 
If $\calUS$ is not pointed then $\calUS$ is $\glt$-least, so trivially $\calUS\glt\calWS$.
Thus, we can assume that $\calUS$ is pointed.
In this case, $\calVS$ is pointed by $\calUS\glt \calVS$, and then $\calWS$ is pointed by $\calVS\glt\calWS$.

By Theorem~\ref{thm:preorder} there exists $f,g\colon\w\to\w$ such that 
$$\calUS \clt f+ \calVS \qquad\text{and}\qquad \calVS\clt g+ \calWS\;.$$

Totality of $f$ and $g$ ensure their pointedness, so we have $f+g\clt f\times g$ by Proposition \ref{prop:sum-vs-product}.
Applying Fact~\ref{fact:prod-upper-seq} and Observation \ref{obs:product-preserves-Turing}, this implies $$\calUS\clt f+ \calVS  \clt (f+ g)+ \calWS\clt(f\times g)+\calWS,$$
where we regard $f\times g$ as a single-valued total function on $\w$. Since $\clt$ is transitive (Theorem~\ref{thm:kih-bil}), apply Theorem~\ref{thm:preorder} once more 
to conclude $\calUS\glt\calWS$. 
\end{itemize}

\end{proof}

\begin{conclusion}\label{con:preorder} The extended Gamified Kat\v{e}tov order is the relativisation of $\clt$ to arbitrary Turing oracles; this perspective is reflected in the $\mathsf{o}$ notation in `` $\glt$ ''. In particular, $\glt$ is a preorder.
\end{conclusion}

\subsection{Combinatorial Complexity via Oracle Power}\label{sec:cofinal-thm} Let us pause to review our work. It is consistent with previous results that Turing degrees and filters may be generally incomparable within $\clt$, and so direct points of contact between them (without, say, taking sums or Cartesian products) are rare. However, encouragingly, this is in fact very much not the case. As it turns out, the knowledge of any Turing oracle can always be encoded into an ideal, as follows:

\begin{proposition}\label{prop:turing-coding} For any $f\colon\w\to\w$, there exists a summable ideal $\calI$ on $\w$ such that 
	$$f\clt \calI^\ast.$$
\end{proposition}
	\begin{proof} Some conventions: 
		\begin{enumerate}[label=(\alph*)]
			\item Each $\sigma\in\lww$ is identified with a natural number via the fixed computable bijection $\lww\simeq \w$. 
			\item We write $f\rstr n$ to mean the finite string $(f(0),f(1),f(2),\dots f(n))$. 
		\end{enumerate}

\noindent Now define the function
	\begin{align*}
	g\colon \lww&\longrightarrow \w \\
	\sigma&\longmapsto \begin{cases}1 \qquad\text{if}\; \exists \,n\in\w\; \text{such that}\; \sigma=f\rstr n\\
	0\qquad \text{otherwise\;}
	\end{cases}.
	\end{align*}
	In which case, $A\in \mathrm{Sum}_g$ iff $|A\cap \{f\rstr n \mid n\in\w\}|<\infty$, where we regard $\mathrm{Sum}_g$ as an ideal on $\w$ via the computable bijection in Convention (a). Taking the dual filter, we claim that
	$$f\clt (\mathrm{Sum}_g)^*\quad. $$
		
	Why? Consider the corresponding Bilayer Game below:
			\[
		\begin{array}{rccccccccccc}
		\mer \colon	& n	&		&& \sigma	 \\
		\art\colon	&		&  0	&		&&&&& y_1=\lranglet{1}{\sigma\rstr n} \\
		\nim\colon	&		& \{ f\rstr m\in\lww \mid  m\geq n\}	&		
		\end{array}
		\]
In English: $\mer$ starts the game by selecting some $n\in\w$, challenging $\art$ to determine $f(n)$. $\nim$ replies by picking the subset $A=\{ f\rstr m\in\lww \mid  m\geq n\}$, which clearly belongs to $(\mathrm{Sum}_g)^*$. As per the rules of the game, $\mer$ picks some $\sigma\in A$, and so $\sigma=f\rstr m$ for some $m\geq n$. In response, $\art$ outputs the answer $\sigma\rstr n = f(n)$, successfully terminating the game.
	\end{proof}
	
\begin{remark} The proof of Proposition~\ref{prop:turing-coding} gives a prototypical argument for showing LT-reducibility via games. Informally, $\nim$ has the power to defeat $\mer$ whenever she can force the wizard to produce a response so that $\art$ only needs to make a finite number of checks to deduce the correct answer.\footnote{In particular, this finite bound ensures $\art$'s strategy remains {\em computable}.} For more examples, see \cite{Kih23}. 

\end{remark}
	
Proposition~\ref{prop:turing-coding} shows that the Turing degrees embed cofinally into the $\clt$-order on filters over $\w$. We next show that the cofinality is non-trivial: while each Turing degree is dominated by some filter, there exists no universal filter bounding every degree. In this sense, we can systematically sound out the (combinatorial) complexity of filters by identifying which Turing degrees they fail to bound within $\clt$. 

\smallskip 

It is not difficult to see why no such universal filter exists -- this follows almost immediately from a result by Phoa \cite[Proposition 3]{Pho89}.
%
%

\begin{theorem}[Phoa {{\cite{Pho89}}}]\label{thm:phoa} Let $j$ be an LT topology in $\Eff$, and $\dn$ the double negation topology. \underline{Then}, $$\dn \clt j \iff k\clt j \;\;\text{for every Turing degree topology}\; k\;.$$
	 In particular, if $j$ is non-trivial, then $j\eclt \dn$ .
\end{theorem}
	


In Section \ref{sec:generalised-Phoa}, we will explain the underlying mechanism of this theorem by placing it in the right context, which points the way to a more general theorem.
For now, we apply Theorem~\ref{thm:phoa} to establish the following cofinality result.

\begin{theorem}[Cofinality]\label{thm:cofinal} The following is true:
		\begin{enumerate}[label=(\roman*)]
		\item For any $f\colon\omega\to\omega$, there exists a summable ideal $\calI$ such that
$$f \clt \calI^*,$$
where $\calI^*$ is the dual filter.
			\item There does not exist a filter $\calF\subseteq\Pw$ such that 
					$$k\clt\calF,\qquad\text{for all Turing degree topologies}\; k\;.$$
		\end{enumerate}

\end{theorem}
	\begin{proof} \hfill 
\begin{enumerate}[label=(\roman*):]
	\item By Proposition~\ref{prop:turing-coding}.
	\item  This is easy, but we elaborate.
	Suppose for contradiction that there exists a filter $\calF\subseteq\Pw$ such that $k\clt \calF$ for all Turing degree topologies $k$.
	 Theorem~\ref{thm:phoa} implies that $$\dn\clt \calF \;.$$	 
	 Observe from \cite[Prop.~3.1]{LvO13} that $\calA_{\dn}=\{\{0\},\{1\}\}$ is the generating subset family for the double-negation topology (see also Example \ref{ex:A-function-turing} below). 
However, it is clear that $\calA_{\dn}$ fails the 2-intersection property [since $\{0\}\cap \{1\}=\emptyset$], whereas $\calF$ satisfies it [being closed under finite intersection]. Applying Summary Theorem~\ref{sumthm-1}, this implies $$\dn\not\clt\calF\;,$$ 
a contradiction, so we are done. 
\end{enumerate}
	\end{proof}
	
\begin{discussion}
 It is natural to ask if we also get cofinality in the converse direction. Namely, for any filter $\calF\subseteq\Pw$, does there exist a Turing degree $k$ such that $\calF\clt k$? However, this is emphatically false: given any filter $\calF$ and any Turing degree $k$, it is easy to see that the relation $\calF\clt k$ forces $\calF$ to be principal \cite[Prop. 3.4.4]{Lee}; see also Discussion \ref{desc:function-complexity-as-principal-filter} (regarding a Turing degree as a ``varying principal ultrafilter''). One can therefore read item (ii) of Cofinality Theorem~\ref{thm:cofinal} as giving the next best possible result: for any filter $\calF$, there exists a Turing degree $k$ such that $k\not\clt \calF$.
\end{discussion}

\begin{conclusion}\label{conc:turing} Every filter $\calF$ determines an {\em initial segment} of the Turing degrees,
$$\calD_{\rm T}(\calF):=\big \{ \,[f\colon\w\to\w] \bmid f\clt \calF \big\}\,,$$
which we call the {\em Turing degree profile of $\calF$}. 

By the Cofinality Theorem~\ref{thm:cofinal}, every Turing degree lies in the profile of some filter, while each $\calD_{\rm T}(\calF)$ omits at least one degree.  In this surprising sense, Turing degrees calibrate the relative complexity of filters by measuring their computable strength within $\clt$.
\end{conclusion}

Conclusion~\ref{conc:turing} presents a remarkable suggestion: given the $\clt$ order, we may use the Turing degrees to probe the structural differences between filters on $\w$. As a first step in developing this insight, we identify a large class of filters whose Turing degree profiles are precisely the hyperarithmetic degrees.



\medskip

We recall the definition of hyperarithmetic (see also \cite{Sac90,ChYu15}), before stating the theorem.

\begin{definition}[Hyperarithmetic] A subset $A\subseteq \w$ is $\Pi^{1}_1$ if it can be defined in the language of second-order arithmetic by
$$A=\{x\mid \forall X\,.\, \psi(X,x)\}\,,$$
where $\forall X$ is the only second-order quantifier in $\forall X\,.\, \psi(X,x)$. A set is $\Sigma^{1}_1$ if its complement is $\Pi^1_1$; and a set is $\Delta^1_1$ or {\em hyperarithmetic} if it is both $\Sigma^1_1$ and $\Pi^1_1$. A function $f\colon\omega\to\omega$ is {\em hyperarithmetic} if 
$$\mathrm{graph}(f):=\{\lranglet{n}{f(n)} \mid  n\in\w\, \}$$
is a hyperarithmetic set.
\end{definition}

\begin{theorem}[Hyperarithmetic Strength]\label{thm:filter-hyperarithmetic} Let $\calF\subseteq \Pw$ be a non-principal $\Delta^1_1$ filter. Then, for any $f\colon\w\to\w$, 
$$f\clt \calF \iff f\;\;\text{is hyperarithmetic}\;.$$
\end{theorem}

The actual proof of the theorem requires some preparation, so it will be helpful to divide the argument into two stages.

\subsubsection*{Step 1: Reduction to $\Pi^1_1$} First a basic observation, before some useful known facts in computability theory and descriptive set theory.

\begin{observation} A (total, single-valued) function $f\colon\w\to\w$ is hyperarithmetic iff the graph of $f$ is $\Pi^1_1$.
\end{observation}
\begin{proof}
 The forward implication is immediate, so we prove the converse. 
Since $f$ is total and single-valued, 
$$ f(n)\not=k \iff \exists k'\neq k \,.\underbrace{f(n)=k'}_{\Pi^1_1}\,.$$
Since $\Pi^1_1$ is closed under number quantification, 
conclude that $\w\setminus \mathrm{graph}(f)$, defined by the above expression, is also $\Pi^1_1$.
\end{proof}

\begin{definition}
An operator $\Phi\colon\mathcal{P}(\omega)\to\mathcal{P}(\omega)$ is {\em monotone} if $A\subseteq B$ implies $\Phi(A)\subseteq\Phi(B)$. 
\begin{enumerate}[label=(\roman*)]
    \item If $\Gamma$ is a class (of sets) such that $\{(A,n):n\in\Phi(A)\}\in\Gamma$, then call $\Phi$ a {\em $\Gamma$-operator}.
    \item Denote the {\em least fixed point of $\Phi$} as $\mu\Phi$.
\end{enumerate}
\end{definition}

\begin{fact}[see e.g.~{\cite[Proposition 2.2.2]{Aczel}}]\label{fact:inductive-operator}
The least fixed point of a monotone $\Pi^1_1$-operator is $\Pi^1_1$.
\end{fact}

\begin{construction} Let $T\subseteq\lww $ be a well-founded tree. For each $\sigma\in T$, set $\mathrm{ar}(\sigma):=\{i\in\omega \mid \sigma\fr i\in T\}$.
Given $g\colon\sum_{\sigma\in T}X^{\mathrm{ar}(\sigma)}\to X$, one can recursively construct $R_g\colon T\to X$ as follows:
\begin{itemize}
\item Base step: If $\sigma$ is a leaf, i.e., ${\rm ar}(\sigma)=\emptyset$, then $R_g(\sigma)=g(\sigma,\ast)$.
\item Inductive step: 
Otherwise, $R_g(\sigma)=g(\sigma,\langle R_g(\sigma\fr i)\rangle_{i\in{\rm ar}(\sigma)})$.
\end{itemize}
A similar construction can be performed for a well-founded forest $T\subseteq\omega\times\omega^{<\omega}$.
    
\end{construction}

\begin{fact}[$\Pi^1_1$ transfinite recursion]\label{fact:transfinite-recursion}
Let $T$ be a well-founded $\Pi^1_1$ forest.
If a function $g\colon \sum_{\sigma\in T}\omega^{\mathrm{ar}(\sigma)}\to \omega$ is $\Pi^1_1$, then $R_g$ is hyperarithmetic.
\end{fact}
\begin{proof}
Fact \ref{fact:inductive-operator} $\Rightarrow$ Fact \ref{fact:transfinite-recursion}:
We construct a monotone $\Pi^1_1$-operator $\Phi$:
Let $A$ and $m$ be given.
Declare $m\in\Phi(A)$ if there is $\sigma\in T$ such that, for each $i\in{\rm ar}(\sigma)$, $\langle\sigma\fr i,a_i\rangle\in A$ for some $a_i\in\omega$, and $m=\langle \sigma,g(\sigma,\langle a_i\rangle_{i\in{\rm ar}(\sigma)})\rangle$.
This condition is clearly $\Pi^1_1$ (since it contains only number quantifiers, excluding references to $g$).

As usual (see e.g.~\cite{Aczel}), $\mu\Phi$ is constructed by the transfinite iteration $\Phi^\alpha(\emptyset)$:
$\Phi(\emptyset)$ corresponds to the base step of the construction of $R_g$, and $\Phi^{\alpha+1}(\emptyset)$ corresponds to the inductive step of the construction of $R_g$.
By transfinite induction, one can easily verify that for each $\sigma\in T$ there is a unique $m$ such that $(\sigma,m)\in\mu\Phi$, and indeed, $m=R_g(\sigma)$.
\begin{itemize}
\item Base step:
If $\langle\sigma,m\rangle\in\Phi(\emptyset)$, then $\sigma$ must be a leaf.
In this case, $m=g(\sigma,\ast)=R_g(\sigma,\ast)$.
\item Inductive step:
If $\langle\sigma,m\rangle\in\Phi(\Phi^\alpha(\emptyset))$, then for each $i\in{\rm ar}(\sigma)$, $\langle\sigma\fr i,a_i\rangle\in\Phi^\alpha(\emptyset)$ for some $a_i$.
By the induction hypothesis, $a_i=R_g(\sigma\fr i)$.
By our definition of $\Phi$, we have 
\[m=g\left(\sigma,\langle R_g(\sigma\fr i)\rangle_{i\in{\rm ar}(\sigma)}\right)=R_g(\sigma).\]
\end{itemize}
Thus, the least fixed point $\mu\Phi$ is the $\Pi^1_1$ graph of $R_g$. 
\end{proof}

\begin{proof}[Step 2: Proof of Theorem~\ref{thm:filter-hyperarithmetic}] \hfill 
\begin{itemize}
    \item[$\impliedby$:] If $\calF$ is a non-principal filter, then $\Fin^*\subseteq \calF$; in particular, the identity map $\id\colon\w\to\w$ defines a witness for $\Fin^*\clt\calF$. By a result of van Oosten \cite[Theorem 2.2]{vO14}, $f\clt \Fin^*$ iff $f$ is hyperarithmetic. Hence, for any hyperarithmetic $f$, we have $f\clt \Fin^*\clt \calF$.
       \item[$\implies$:] 
Assume $f\leq_{\sf LT}\mathcal{U}$ via a partial computable function $\Phi$; that is, for any $n\in\omega$ there is an $\mathcal{U}$-branching tree $T_n$ such that $\Phi(n,p)=f(n)$ for any $p\in[T_n]$.
Consider $\Psi_n$ defined as $\Psi_n(p)=\Phi(n,p)$.
Recall from Definition \ref{def:label} that $\mathcal{U}$ and $\Psi_n$ yield a labeling function $\nu_n:=\nu^\mathcal{U}_{\Psi_n}$.

Let us show that $\nu_n$ is a $\Delta^1_1$-function.
Note that $T_{\Psi_n}$ in Definition \ref{def:label} is a $\Sigma^0_1$ tree.
The forest $\sum_nT_{\Psi_n}$ is also $\Sigma^0_1$.
\begin{itemize}
\item[$\bullet$] Base step:
The labeling according to (1) in Definition \ref{def:label} corresponds to the base step of the recursive construction $g(n\fr\sigma,\ast)=\Psi_n(\sigma)$.
\item[$\bullet$] Inductive step:
The labeling according to (3),(4) in Definition \ref{def:label} corresponds to the inductive step of the recursive construction.
The idea is that if we correctly know the labels $\alpha(i)$ of the $i$th successor $\sigma\fr i$, we can determine the labels of $\sigma$ inductively.
Formally, given $\alpha\in(\omega\cup\{\bot\})^\omega$, put 
\[L(\alpha,c)=\{i\in\omega\mid i\in{\rm ar}(\sigma)\mbox{ and }\alpha(i)=c\}.\]

Since $n\in \mathrm{ar}(\sigma)$ is a $\Sigma^0_1$ relation, so is $L(\alpha,c)$.
Declare $(\sigma,\alpha,c)\in G_g$ if either $c\in\omega$ and $L(\alpha,c)\in\mathcal{U}$ or $c=\bot$ and $L(\alpha,d)\not\in \mathcal{U}$ for any $d\in\omega$.
This is a $\Delta^1_1$ condition since $\mathcal{U}$ is $\Delta^1_1$.

Moreover, since $\mathcal{U}$ is a filter, for any $(\sigma,\alpha)$ there is a unique $c$ such that $(\sigma,\alpha,c)\in G_g$.
Thus, $G_g$ gives the graph of a hyperarithmetical function $g$.
\end{itemize}

By Fact \ref{fact:transfinite-recursion}, we get a hyperarithmetical function $R_g$.
By well-founded induction, one can verify that $R_g(n\fr \varepsilon)=\nu_n(\varepsilon)$, where $\varepsilon$ is the root of the tree $T_{\Psi_n}$.

We claim that the root $\varepsilon$ must be $\nu_n$-labeled by $f(n)$.
Otherwise, similar to Lemmas \ref{lem:separation-1} and \ref{lem:separation-2}, we can obtain an infinite path $T_n$ consisting of nodes with labels other than $f(n)$.
To see this, inductively assume that we have chosen a node $\sigma\in T_n$ whose $\nu_n$-label is not $f(n)$.
If $\nu_n(\sigma)\not=\bot$, by Lemma \ref{lem:separation-1}, there is an infinite path $p\in[T_n]$ such that $\Psi_n(p)=\nu_n(\sigma)\not=f(n)$, which contradicts our assumption.
Hence, we have $\nu_n(\sigma)=\bot$,
In this case, $\{k\in\omega\mid \nu_n(\sigma\fr k)=f(n)\}$ is $\mathcal{U}$-null, so there is $k$ such that $\sigma\fr k\in T_n$ with $\nu_n(\sigma\fr k)\not=f(n)$ since $T_n$ is $\mathcal{U}$-branching.
Eventually, we must construct a path through $T_n$ consisting of nodes with the label $\bot$, but as shown in Lemma \ref{lem:separation-2}, this is impossible.

Consequently, the value $f(n)$ is obtained as the $\nu_n$-label of the root $\varepsilon$.
This means $f(n)=R_g(n\fr\varepsilon)$, which is hyperarithmetic.
\end{itemize}    
\end{proof}

\medskip

\begin{discussion} Theorem~\ref{thm:filter-hyperarithmetic} is anticipated by several earlier results. Restated in our language, work of van Oosten \cite{vO14} and Kihara \cite{Kih23} shows that
\[
\calD_{\rm T}({\rm Fin}^\ast)
\;=\;
\calD_{\rm T}({\rm Den}_0^\ast)
\;=\;
\Delta^1_1,
\]
where ${\rm Fin}$ and ${\rm Den}_0$ denote the ideals of finite sets and asymptotic density zero sets, respectively. Our Theorem~\ref{thm:filter-hyperarithmetic} is thus a significant generalisation of these results. Moreover, from this perspective, ${\rm Fin}^\ast$ (Pitts’ topology) is not particularly special: in fact, there exists many natural filters whose structure corresponds to ``the world of hyperarithmetical mathematics.''
\end{discussion}

\begin{discussion}\label{dis:hyper-DTF} Combined with Conclusion~\ref{conc:turing}, Theorem~\ref{thm:filter-hyperarithmetic} illustrates how the invariant $\calD_{\rm T}$ is coarse yet non-trivial. One reading of this coarseness is positive: any distinction detected by $\calD_{\rm T}$ should be regarded as significant (since Turing degrees are in some sense very forgiving), prompting a closer structural analysis of the underlying filter. On the other hand, for large classes of filters — such as $\Delta^1_1$ filters —$\calD_{\rm T}$ may be seen as collapsing too much information. In which case, finer invariants, such as Weihrauch complexity, may be better suited to detect differences in complexity invisible at the level of single-valued computation. This motivates the shift in focus in the next subsection.

\end{discussion}

\subsection{Hierarchies of Choice Problems}\label{sec:generalised-Phoa}
%

In modern computability theory, the computable complexity of ``{\em choice problems}'' is of central interest, and represents a very active area of research.\footnote{For the uninitiated reader, there is a substantial literature on the Weihrauch degrees of choice problems -- see e.g. ~\cite{BG11,BdBP12,BGH15,KiPa16,KiPa19,BLMP19,AnKi21,BGP21}.} This perspective extends naturally to our context: 
\begin{itemize}
\item A subset family $\calA \subseteq \mathcal{P}(\omega)$ may be regarded as the problem of choosing an element $n \in A$ from a given $A \in \calA$.
\end{itemize}
For instance, stated in this language, the cofinite filter ${\rm Fin}^*\subseteq\Pw$ defines the ``{\it cofinite choice} problem.'' Moreover, by imposing {\bf complexity restrictions} on $\Fin^*$, we obtain a class of derived choice notions, as below:

\begin{example}\label{exa:cofinite-choice} Fix $k\in\omega$. We define the {\em $\Pi_k$-cofinite choice} problem as the multivalued function $$\Pi_k\text{-Cof}\colon\omega\tto\omega,$$ 
defined informally as follows.
\begin{itemize}
\item Input: Any $\Pi_k$-index of a cofinite set $A\subseteq\omega$, and so
\[{\rm dom}(\Pi_k\text{-Cof})=\{e\in\omega\mid\text{the $e$th $\Pi_k$ subset of $\omega$ is cofinite}\}\ .\]
\item Output: Any element $n\in A$ is accepted, and so
\[
\Pi_k\text{-Cof}(e)=\text{the $e$th $\Pi_k$ subset of $\omega$.}
\]
\end{itemize}
More generally, for any $\om$-parameterised pointclass $\Gamma$ (i.e., a complexity class indexed by $\omega$), the $\Gamma$-cofinite choice problem $\Gamma\text{-Cof}$ can be expressed as a multivalued function.
\end{example}

\begin{remark} Notice the LT-degree of a multivalued function is a purely computational notion (i.e. a Gamified Weihrauch degree).
The first author \cite[\S 5.2]{Kih23} has shown that  the LT-degree of $\Pi_{1+\alpha}\text{-Cof}$ (for computable ordinals $\alpha$) is exactly the $(\alpha+1)$st Turing jump. 
 In other words, the hyperarithmetic hierarchy of cofinite choice coincides precisely with the hyperarithmetic degrees. This provides an alternative explanation of the hyperarithmetic lower bound on the complexity of ${\rm Fin}$ in \cite{vO14}. To see why, think of $\Gamma\text{-Cof}$ as a subproblem of the full cofinite choice problem ${\rm Fin}^\ast$, and so $\Gamma\text{-Cof}\clt{\rm Fin}^\ast$.
Hence,
\[\emptyset^{(\alpha+1)}\equiv_{\rm LT}\Pi_{1+\alpha}\text{-Cof}\leq_{\rm LT}{\rm Fin}^\ast,\text{ and therefore \ } \ \Delta^1_1\subseteq\calD_{\rm T}({\rm Fin}^\ast)\ .\]
\end{remark}

Importantly, the underlying idea in Example \ref{exa:cofinite-choice} can be extended to arbitrary subset families, not necessarily $\Fin^*$. We consider the notion of an $\omega$-parametrisation of a subset family as follows:

\begin{definition}\label{def:A-function} Let $I\subseteq \w$ and a subset family $\calA\subseteq \calP(I)$ be given. An  {\em $\calA$-function} $F$ is a partial multifunction 
	$$F\pcolon \w\rightrightarrows I $$
whereby $F(n)\in\calA$ for any $n\in\dom(F)$.
\end{definition}

\begin{example}
$\Pi_n\text{-Cof}$ is a total ${\rm Fin}^\ast$-function.
\end{example}

Informally, an $\calA$-function is the {\em $\calA$-choice problem restricted to a complexity class} (an $\omega$-parametrised pointclass).
How does this relate to the original ``unrestricted'' choice problem $\calA$? In Kihara’s language of bilayer functions \cite{Kih23}, the distinction is whether the input is {\bf public} or {\bf secret}.
\begin{itemize}
\item A restricted choice problem $\Gamma\text{-}\calA$ (i.e. an $\calA$-function):
\begin{itemize}
\item Input: A $\Gamma$-code of $A\in\calA$ is given {\bf publicly}.
\item Output: Choose an $n\in A$.
\end{itemize}
\item An (unrestricted) choice problem $\calA\subseteq\mathcal{P}(I)$:
\begin{itemize}
\item Input: $A\in\calA$ is given {\bf secretly}.
\item Output: Choose an $n\in A$.
\end{itemize}
\end{itemize}
Framed this way, it is natural to conjecture that solving the unrestricted choice problem $\calA$ is equivalent to solving all restricted choice problems $\Gamma\text{-}\calA$. Generalised Phoa confirms this intuition, as below:

	\begin{theorem}[Generalised Phoa]\label{thm:gen-phoa} Let $\calU\subseteq \Pw$ be an upper set, and let $\calVS$ be an upper sequence on $\w$. \underline{Then},
		\begin{align*}
        \calU\clt \calVS &\iff F\clt \calVS\;\;\text{for every}\;\calU\text{-function}\; F\\
        &\iff F\clt \calVS\;\;\text{for every total}\;\calU\text{-function}\; F\quad.
        \end{align*}
	\end{theorem}

\begin{proof} 
\hfill
\begin{itemize}
	\item[$\implies$:] Assume $\calU\clt\calVS=(\calV_n)_{n\in I}$, witnessed by partial computable functions $\gamma,\Phi$. Let $F$ be a $\calU$-function. By definition: for any $n\in\dom (F)$, we have $F(n)\in\calU$, and so there exists a $(\calVS,\gamma)$-branching tree $T$ such that $\Phi[T]\subseteq F(n)$. This clearly implies $F\clt \calVS$.
	\smallskip
	\item[$\impliedby$: ]  For contradiction, assume $\calU\not\clt\calVS$. Fix an enumeration  $$(\gamma_n,\Phi_n)_{n\in\w}$$ of all pairs of partial computable functions as in Definition~\ref{def:GTK-full}. Since $\calU\not\clt\calVS$, for each $n$ there must exist some $A_n\in\calU$ such that $\Phi_{n}[T]\not\subseteq A_n$ for any $(\calVS,\gamma_{n})$-branching tree $T$. Define $F(n)=A_n$. It is clear that $F$ is a total $\calU$-function.  
	
	We can now apply the given hypothesis. Suppose $F\clt\calVS$ is witnessed by a computable pair $(\gamma,\Phi)$. This means: for any $n$, there exists a $(\calVS,\gamma^{[n]})$-branching tree $T$ such that $\Phi^{[n]}[T]\subseteq F(n)$. This induces a computable enumeration $c$ whereby $(\gamma_{c(n)}, \Phi_{c(n)})=(\gamma^{[n]},\Phi^{[n]})$. By Kleene's recusion theorem (see e.g.~\cite[Theorem 11.2.10]{Odi89}), there exists $r$ (a ``fixed point'') such that $$(\gamma_{c(r)},\Phi_{c(r)})=(\gamma_r, \Phi_r).$$
	
By construction:
\begin{enumerate}
	\item $\Phi_r[T] \not\subseteq A_r$ for any $(\calVS,\gamma_r)$-branching tree $T$.
	\item $\Phi^{[r]}[T]\subseteq F(r)=A_r$ for some $(\calVS,\gamma^{[r]})$-branching tree $T$.
\end{enumerate}
\smallskip
Since $r$ was chosen such that $\Phi_r=\Phi^{[r]}$ and $\gamma_r=\gamma^{[r]}$, items (1) and (2) yield a contradiction. Hence, conclude that $\calU\clt\calVS$.
\end{itemize}	

\end{proof}

\begin{remark} Notice the proof of Theorem~\ref{thm:gen-phoa} uses the tree-based definition of the $\clt$-order, as opposed to game-based definition (which we used in the proof of Theorem~\ref{thm:preorder}).
\end{remark}

\begin{remark} Theorem~\ref{thm:gen-phoa} becomes natural once subset families and multi-valued functions are viewed uniformly as choice problems. Outside this context, however, their connection becomes less obvious. 
Indeed, since subset families capture the combinatorial aspects of LT topologies and multi-valued functions capture the computational aspects, there are ways in which the two notions may be regarded as orthogonal in $\clt$ (cf. Discussion~\ref{dis:2d-complex}).
\end{remark}



\subsubsection*{Phoa's Theorem in context} We now explain how Theorem \ref{thm:gen-phoa} generalises the original Phoa's Theorem \ref{thm:phoa}, which relates Turing complexity to the double-negation topoplogy $\dn$.
The key observation is that Phoa's theorem admits a natural interpretation within the computable analysis of the hierarchy of the law of excluded middle ($\neg\neg$-elimination).

For context, non-constructive principles such as ${\rm LEM}$ (law of excluded middle) and ${\rm DNE}$ (double negation elimination) admit natural hierarchies according to their logical and computational complexity. In proof theory, such principles can be organised into an arithmetical hierarchy, which has been studied extensively.\footnote{See e.g.~ \cite{ABHK04,FuKu21,FuKu22,Nak24}.} From the viewpoint of computability theory, non-constructive principles at restricted complexity levels may be represented as multi-valued functions, and thus analysed in terms of their computational strength \cite{BGP21}. As a concrete example:

\begin{example}
For $n,k\in\om$, we express $\Sigma_n\text{-WLEM}$ ($\Sigma_n$ weak law of excluded middle) and $\Sigma_n\text{-DML}_{1/k}$ ($\Sigma_n$ de Morgan's law for $k$ tuples) as follows:\footnote{Again, we do not give the formal details here, but see e.g.~\cite{BGP21} (notice that the notation is slightly different).}
\begin{itemize}
\item $\Sigma_n\text{-WLEM}\colon\om\to 2$
\begin{itemize}
\item Input: An index of a $\Sigma_n$-sentence.
\item Output: Whether it is true or false.
\end{itemize}
\item $\Sigma_n\text{-DML}_{1/k}\pcolon\om\tto k$
\begin{itemize}
\item Input: A $k$ tuple of $\Sigma_n$ sentences $\varphi_1,\dots,\varphi_k$, where at most one $\varphi_i$ is true.
\item Output: Any $i$ such that $\varphi_i$ is false.
\end{itemize}
\end{itemize}
$\Sigma_1$-WLEM and $\Sigma_1$-DML$_{1/2}$ are sometimes referred to as LPO and LLPO, respectively.
\end{example}

This picture has a natural translation to the setting of LT topologies. Since we are interested in the original Phoa's theorem, let us focus on the example of double-negation:

\begin{example}[Double negation topology]\label{ex:A-function-turing} Logically, the semantics (of the sheaf topos) relativised to the $\neg\neg$-topology is given by the G\"odel-Gentzen $\neg\neg$-translation (which embeds classical logic into intuitionistic logic).\footnote{More generally, the semantics relativised to the LT topology $j$ is given by the so-called $j$-translation \cite{vdB19}.} Hence, the $\dn$-topology plays the formal role of $\dn$-elimination, $$\neg\neg\varphi\to\varphi\ .$$ We express this as a choice problem as follows.
\begin{itemize}
\item Input: A realisable sentence $\varphi$ is given secretly.
\item Output: Return any realiser for $\varphi$.
\end{itemize}
This choice problem expresses the following subset family:
\[\calC_{{\rm DNE}}=
\{\text{the set of all realisers of $\varphi$}\mid \varphi\mbox{ is a realisable sentence}\}
=\mathcal{P}(\om)\setminus\{\emptyset\}
\]

Considering the (unrestricted) unique choice problem $\calC_{!2}=\{\{0\},\{1\}\}$, it is easy to verify that $\calC_{{\rm DNE}}\equiv_{\rm LT}\calC_{!2}$, which corresponds to the double-negation topology \cite[Prop.~3.1]{LvO13}.
A $\calC_{!2}$-function is merely a single-valued function; for instance, $\Sigma_n\text{-WLEM}$ is a $\calC_{!2}$-function.
\end{example}

\begin{example}\label{exa:Ck-DML}
Let $\calC_{(k)}=\{A\subseteq k\mid k\setminus A\mbox{ has at most one element}\}$.
Its upward closure has been studied in \cite{LvO13} and \cite{Kih23} as $\mathcal{O}^k_1$ and ${\rm Error}_{1/k}$, respectively.
Clearly, $\Sigma^0_n\text{-DML}_{1/k}$ is a partial $\calC_{(k)}$-function.
\end{example}

\begin{proof}[Proof of Phoa's Theorem \ref{thm:phoa} using Generalised Phoa's Theorem \ref{thm:gen-phoa}]
Observe 
	that double negation corresponds to the basic topology generated by $\calC_{!2}=\{\{0\},\{1\}\}$ as mentioned in Example \ref{ex:A-function-turing}.
	Then, a total $\calC_{!2}$-function is exactly a single-valued (total) function $F\colon\w\to 2$.
Thus, $\calC_{!2}$-functions recover the classical Turing degrees.
By Lemma~\ref{lem:upward-closure-basic-properties}, passing to the upward closure does not change its LT degree, so we may regard $\calC_{!2}$ as an upper set.
\end{proof}

	\section{Further Discussion and Questions}\label{sec:test-prob}

As a benchmark for progress, it is worth revisiting the opening remarks of Lee-van Oosten in their paper \cite{LvO13}, written in the early 2010's:

\smallskip

\begin{fquote}[Lee-van Oosten \cite{LvO13}]
    The lattice of local operators [i.e. LT topologies] in $\Eff$ is vast and notoriously difficult to study. We seem to lack methods to construct local operators and tell them apart.
\end{fquote}
Part of the original difficulty stems from the fact that the Effective Topos $\Eff$ is not a Grothendieck topos, and so many standard tools from topos theory are no longer available. The results of this paper locate a subtler issue: the structure of the $\clt$-order in $\Eff$ is tightly connected to the combinatorics of filters on $\w$. In particular, its complexity is driven by the same deep structural questions studied in the theory of ultrafilters, now also intertwined with questions arising from computability theory. This surprising connection reveals a rich interplay between combinatorial and computable complexity, which underscores the order's power but also contributes to its difficulty: progress requires advances in combinatorial set theory on the one hand (e.g. separation techniques), and advances in computability-theoretic analyses on the other. 

\smallskip

We are now at an exciting juncture, where the framework developed here both connects and clarifies several previously disparate lines of work, thereby opening up a wide range of new questions. We conclude the paper with a broad list of inter-related test problems, laying the groundwork for future work.


\subsection{Combinatorial Set Theory}\label{sec:qn-set-theory}

\begin{problem} Is $\glt$ non-linear on ideals? If yes, what is the size of its maximum anti-chain?
\end{problem}

\begin{discussion} It would be consistent with Theorem~\ref{thm:strict-ascend} to conjecture that $\glt$ defines a countable linear order on the equivalence classes of ideals. However, in forthcoming work, we show that this is emphatically not the case. In particular, we construct an embedding of $\Pw/\Fin$ into $\glt$. Hence, the order has continuum many classes, in fact an antichain of size continuum.
\end{discussion}

In light our proof of Theorem~\ref{thm:Tukey}, the following problem is natural:

\begin{problem} Give a constructive proof that the Gamified Kat\v{e}tov order and the Tukey order are incomparable. 
\end{problem}

A potentially rich line of inquiry would be to explore deeper comparisons with the Kat\v{e}tov order on ideals. Conclusion~\ref{conc:MAD} shows that the Gamified Kat\v{e}tov order collapses all MAD families into a single class, but there are other interesting structural questions to investigate.

\begin{problem} Let $\calR$ be the random graph ideal, and $\calS$ be Solecki's ideal. To our knowledge, it is still an open problem whether $\calR\lk\calS$ \cite[Question 5.5]{Hru17}. What about $\calR\glt\calS$?
\end{problem}

\begin{discussion}\label{dis:coarse-katetov} The fact that the Gamified Kat\v{e}tov order is strictly coarser than the classical Kat\v{e}tov order significantly reorganises the landscape, which may open up new lines of attack. For instance, it is known that $\Fin\otimes\Fin\not\lk \EDfin$ \cite[\S 2]{Hru17}, yet $\Fin\otimes \Fin\elt \Fin \glt \EDfin$ in our context (see Section~\ref{sec:separate-example} for the definitions). 
\end{discussion}

\begin{problem} More generally, the literature contains many negative results of the form $\calH\not\lk\calI$ , see e.g. \cite{HrMe12, Hru17, FKK24}. It is worth revisiting them and investigating if the same holds in the Gamified Kat\v{e}tov order. A systematic analysis of such examples may lead to new separation techniques tailored to our setting. As a concrete test case, is it true that
\[
\mathrm{Den}_0 \not\glt \Sumn\; ?
\]
\end{problem}

\begin{problem}[Category Dichotomy]\label{prob:dichotomy} A key structure theorem in the study of ideals is Hru\v{s}\'{a}k's Category Dichotomy \cite{Hru17}, which states: for any Borel ideal $\calI$ on $\w$, either 
\[
\calI \lk \mathrm{nwd}
\quad\text{or}\quad
\exists X \in \calI^+ \text{ such that } \mathrm{ED} \lk \calI\rstr X,
\]
where $\mathrm{nwd}$ is the Nowhere Dense ideal and $\mathrm{ED}$ the Eventually Different ideal. This dichotomy is known to hold for Borel ideals, but to fail (or be independent of ZFC) for other classes of ideals; see \cite[\S 1]{DFGH26}. 

By our Theorem~\ref{thm:fubini-GTK}, the Category Dichotomy for Borel ideals automatically transfers to the Gamified Kat\v{e}tov setting. What can we say about other classes of ideals?
\end{problem}

\begin{discussion} Problem~\ref{prob:dichotomy} raises the possibility that an ideal may satisfy the gamified version of the Category Dichotomy, even if it fails the classical one. 
Discussion~\ref{dis:coarse-katetov} may be relevant here. The fact that Hru\v{s}\'{a}k's original proof \cite[Theorem 3.1]{Hru17} relies on game-theoretic arguments is also suggestive.
\end{discussion}

For those interested in cardinal invariants of the continuum, a compelling problem to consider is:

\begin{problem}\label{prob:card-char} How does $\glt$ interact with the cardinal invariants of ideals? What about  $\clt$? 
\end{problem}

\begin{discussion} There are two levels to Problem~\ref{prob:card-char}.
\begin{enumerate}
    \item It is well-established that the classical Kat\v{e}tov order on ideals has a close connection with their cardinal invariants. For instance, if $\calH\lk\calI$, then $\mathrm{cov}^*(\calI)\geq \mathrm{cov}^*(\calH)$ \cite[Prop. 3.1]{HHH07}
    as well as $\mathfrak{p}_{\rm K}(\calH)\leq\mathfrak{p}_{\rm K}(\calI)$ \cite[Prop. 2.9]{BNF11}. To what extent does this picture extend to the gamified setting?
    \item Independently, there is a very interesting line of work \cite{RupPhD,BrendleEffCardinal,GKT19} developing some formal analogies between cardinal invariants of the continuum and highness of Turing oracles. with a view towards reproducing appropriate analogues of Cicho\'{n}'s diagram. Since the $\clt$-order simultaneously encodes computable and combinatorial information (Theorem~\ref{thm:GTK-full}), it provides a natural framework from which to re-investigate this connection.
\end{enumerate}
\end{discussion}

\begin{discussion}[Pseudo-Intersection] In fact, given {\em any} partial order $\sqsubseteq$ on ideals over $\w$, Borodulin-Nadzieja and Farkas \cite[Prop. 2.9]{BNF11} define a $\sqsubseteq$-intersection number $\frap_{\sqsubseteq}$ such that 
$$\calH\sqsubseteq \calI \implies \frap_{\sqsubseteq}(\calH)\leq \frap_{\sqsubseteq}(\calI)\,.$$  
This motivates a systematic study of pseudo-intersection phenomena across $\lk$, $\glt$, and $\clt$. Additionally, in light of the discussion in Section~\ref{sec:qn-MT}, it is also very interesting to ask how these ideas relate to Malliaris and Shelah's framework of cofinality spectrum problems \cite{MS15,MS16,MS17}, which played a key role in their celebrated discovery that $\frap=\frat$ in ZFC.
\end{discussion}

\begin{discussion}[On the Tukey Order] Cardinal invariants also have a close interaction with the Tukey order, e.g. if $\calH\lT\calI$, then $\mathrm{add}^*(\calH)\leq \mathrm{add}^*(\calI)$ \cite[Prop. 2.1]{HHH07}. This makes the interaction between cardinal invariants and $\glt$ potentially subtle, since the Gamified Kat\v{e}tov order and the Tukey order
are incomparable on filters/ideals (Theorem~\ref{thm:not-Tukey}).\footnote{As a small indication that cardinal invariants already play a role in understanding $\glt$ in ZFC, notice the dominating number $\mathfrak{d}$ appears implicitly in our proof of Theorem~\ref{thm:not-Tukey}.} 
\end{discussion}

\begin{discussion} It is also interesting to ask about potential connections with game-theoretic variants of cardinal invariants \cite{CGH24}. Their setup also makes use of filter games, originating with Laflamme \cite{La96}, except their games are of infinite length (in contrast with our finite-query Kat\v{e}tov game).
\end{discussion}

For those interested in forcing methods:

\begin{discussion} There exists a notion of forcing, known as the Laver-Prikry forcing (see e.g.~\cite{MiHi14,CGGH16,Kh17}), where each forcing condition is a $\calU$-branching tree for a fixed filter $\calU$. The connection is suggestive and invites further investigation -- are there substantive links between forcing and our work?
\end{discussion}

\subsection{Computability Theory and Constructive Logic}\label{sec:qn-comp-th-cons-log} 

\subsubsection{Computability by Majority}\label{sec:comp-majority}
From the perspective of computability theory, the key contribution of this paper is establishing a uniform framework for investigating {\bf computability by majority}, a new way of parametrising computability based on filter combinatorics. Here is a simple working definition:
\medskip


\begin{definition}[$\calI$-computability]\label{def:comp-majority} Let $\calI$ be any ideal on $\w$. Now run $\omega$-many computations $(\Phi_n)_{n\in\omega}$ in parallel. If the majority of these computations are correct, that is $$\{n\in\omega:\Phi_n\mbox{ is incorrect}\}\in\mathcal{I},$$ then the computation is considered successful. In which case, we say the computation is {\em $\calI$-computable}.
\end{definition}

This marks an important shift in understanding. In the context of this paper, each filter is regarded as defining an abstract notion of largeness, which the $\clt$-order then organises into different notions of computability (Theorem~\ref{thm:GTK}). Definition~\ref{def:comp-majority} gives a concrete interpretation: each filter $\calF$ determines a dual ideal $\calF^*$ specifying which sets of computational errors are to be regarded as negligible. More fundamentally, this framework unifies an interesting range of computability notions scattered throughout the literature. For example, random and generic oracles (computational complexity theory) as well as {\em basic} non-modest degrees\footnote{To understand the qualifier `basic'', see Discussion~\ref{dis:separation-takeaway}.} (Bauer's extended Weihrauch reducibility \cite{Bau22}) may all be viewed as instances of computability by majority. 

\smallskip
A crucial move: Theorem~\ref{thm:GTK} shows that the $\clt$-order is isomorphic to the computable Gamified Kat\v{e}tov order on upper sets. This identification isolates the underlying combinatorics of this new hierarchy of computability degrees, opening it up to systematic study. So we may ask:

\begin{problem}\label{problem:LT-degree-filters} What is the structure of the $\clt$-order on filters (or upper sets)? What about the computable Kat\v{e}tov order?
\end{problem}

\begin{discussion}\label{dis:LT-deg-filt} There are many interesting variants of Problem~\ref{problem:LT-degree-filters}. One may explore order-theoretic or first-order properties of the $\clt$-order for all filters, or for special classes such as analytic ideals, $F_\sigma$-ideals, $P$-ideals, etc. Either positive or negative results could yield important insights.

\end{discussion}

\begin{discussion}[Comparisons with other orders]
Related attempts to compare the computable strength of filters include
Blass's work on the Kleene degrees of ultrafilters \cite{Bl85} and the study
of (ultra)filter jumps by Andrews et al.~\cite{ACDS23}.%
\footnote{%
Andrews et al.\ associate to each filter a notion of jump defined via filter
limits, inducing a preorder on filters according to the strength of the
resulting jump.  From a technical perspective, this construction appears
closely related to a non-uniform variant of our computable Kat\v{e}tov
order.  A systematic comparison between these orderings remains an
interesting open direction.
}
Although not formulated in terms of computability by majority, these works share a certain resonance with our framework, and it would be interesting to investigate potential connections.
\end{discussion}

We can also ask about the computable strength of an {\em individual} filter within $\clt$. Recalling Conclusion~\ref{conc:turing}, a general test problem would be:

\begin{problem}\label{prob:DTF} Given a filter $\calF$, characterise its Turing degree profile $\Dt(\calF)$.
\end{problem}

\begin{discussion} 

Again there are many interesting approaches to this problem. Some natural directions:
\begin{itemize}
    \item Determine the Turing degree profile for specific classes of filters, as in Theorem~\ref{thm:filter-hyperarithmetic}. 
    \item Explore the converse: given a natural class of Turing degrees (e.g. hyperarithmetic degrees), determine the class of filters whose Turing degree profile coincides with it?\footnote{For instance, Theorem~\ref{thm:filter-hyperarithmetic} shows that (non-principal) $\Delta^1_1$ filters have the hyperarithmetic degrees as their profile, but is this a characterisation? Are there non-$\Delta^1_1$ filters whose profiles coincide with the hyperarithmetic degrees?}
    \item Since $\Dt$ is a relatively coarse invariant (Discussion~\ref{dis:hyper-DTF}), one may consider profiles with respect to other degree notions, e.g. Weihrauch degrees.
    \item For any filter, how does its $\Dt$ relate to cardinal invariants of its dual ideal? See also Problem~\ref{prob:card-char}.
\end{itemize}  
\end{discussion}

\subsubsection{Reverse Mathematics and Separation Problems} Another important application of this paper's results lies in the connection betweeen constructing realisability models (idealised subtoposes) and the separation of non-constructive logical principles. Here, the separation results in Section~\ref{sec:separation} are central: they show how distinct filters (upper sets) yield genuinely different realisability models. This sets up the natural test problem of investigating the internal logic of these models:

\begin{problem}\label{prob:log-principle} Which logical/mathematical principles hold true in the $\calU$-realisability associated to a filter (upper set) $\calU$?
\end{problem}

\begin{discussion} As an example of previous work in this direction, van Oosten \cite{vO14} has analysed principles that hold in the realisability model based on the Pitts' topology ${\rm Fin}^\ast$.\footnote{Note, however, that a small error is contained in \cite[Proposition 3.3]{vO14}. Any computable, finite-branching, infinite tree does in fact possess an arithmetical path, indeed, a $\emptyset''$-computable path; see e.g.~\cite[Theorem 2.6 (b)]{CeRe98}. Only an {\bf infinite}-branching tree may lack a hyperarithmetical path.} What about realisability models based on, say, the dual filters of $\EDfin$ or $\Sumn$?
\end{discussion}

\begin{discussion}[Modest vs. non-modest LT degrees] Gamified Weihruach degrees were originally introduced to analyse derivation relations in classical reverse mathematics \cite{HiJo16}; see also \cite{DHR22}. 
Subsequently, using a completely different approach, the first author \cite{Kih20} linked the Gamified Weihrauch degrees to constructive reverse mathematics, obtaining a range of separation results for mathematical principles via game-theoretic techniques. However, these results appear to rely almost exclusively on {\em modest} LT degrees. 

By contrast, the present paper focuses on {\em basic topologies}, which lie in the non-modest region of the $\clt$-order. In this setting, comparatively less is known about the separation of mathematical principles, motivating Problem~\ref{prob:log-principle}. Two further comments, which may orient the interested reader. First, analogous to Gamified Weihrauch reducibility, Bauer's extended Weihrauch reducibility \cite{Bau22} also arose from connections with certain derivation relations within constructive mathematics. Second, basic topologies appear to correspond more closely to purely logical principles; see e.g. Example \ref{exa:Ck-DML}, where the upper set $\calC_{(k)}$ corresponds to a weak version of De Morgan's law. 
\end{discussion}

\begin{discussion}[Operations on Degrees]
Recently, Brattka \cite{Brattka25} introduced an infinite loop operation on Weihrauch problems.  Yoshimura subsequently observed (in unpublished work) that closure under this infinite loop operation is closely connected to the axiom of {\em dependent choice}. Building on this insight, the first author has obtained various separation results along these lines, to appear in future work. 

It remains unclear how to lift such operations on Weihrauch problems
(multi-valued functions) to general LT-degrees (upper sequences), and
techniques for analysing closure under these operations are largely
undeveloped. In particular, in the setting of non-modest degrees (which includes the basic topologies), the absence of effective complexity bounds (recall Section
\ref{sec:generalised-Phoa}) poses a key obstacle; new methods are required to analyse these operations at the level of upper sequences, or at least upper sets.  
\end{discussion}

\subsection{Category Theory}\label{qn-Category} Beyond the initial surprise, why else might the category theorist find the connection between the $\clt$-order in $\Eff$ and the Gamified Kat\v{e}tov order interesting?

\subsubsection{Sites and Classifying Toposes} One answer is that by sharpening our understanding of what separates LT degrees, we also clarify what separates subtoposes at a semantic level. To develop this remark, it is worth revisiting, with some appreciation, the beautiful definition of a classifying topos.\footnote{For additional background on classifying toposes, see e.g. the second author's PhD thesis \cite[Chapter 2]{NgThesis}.} 

\smallskip

The {\em classifying topos} of a geometric theory\footnote{For the model theorist, we emphasise that we are presently working in the context of {\em geometric logic}, an infinitary positive logic with important differences from classical first-order logic; see \cite[\S1.1]{NgBerk} for a discussion.} $\thT$ is a Grothendieck topos $\baseS[\thT]$ satisfying the following universal property: for any Grothendieck topos $\calE$, there is a natural equivalence of categories
$$\mathbf{Geom}(\E,\baseS[\thT])\simeq \Tmod(\E),$$
where {\em geometric morphisms} $f\colon \E\to \calS[\thT]$ correspond to {\em models} of $\thT$ internal to $\calE$. Two remarks are worth highlighting:
\begin{enumerate}
    \item Classifying toposes capture the {\em semantic expressiveness} of theories. In practice, there are many theories that are syntactically-inequivalent, yet define equivalent categories of models.\footnote{A standard example is provided by the Dedekind reals, which admits well-known axiomatisations as both a propositional geometric theory and a predicate geometric theory; see \cite[\S 4.7]{Vi07}. In particular, although predicate logic is strictly more expressive than the propositional fragment, the two theories are shown to have equivalent classifying toposes.} Topos theory resolves this tension by providing a more nuanced criterion for logical equivalence: two theories are equivalent just in case their classifying toposes are equivalent as categories. 
    \item Classifying toposes  characterise Grothendieck toposes: every Grothendieck topos $\E$ classifies some geometric theory $\thT_\E$, and every geometric theory $\thT$ has a classifying topos $\baseS[\thT]$.

It is this structure theorem that gives classifying toposes much of their power. However -- and this is an important point -- the standard proofs make crucial use of the fact that Grothendieck toposes admit site presentations: 
{\it if $\E \simeq \mathbf{Sh}(\mathcal{C},J)$, then $\E$ classifies the theory of $J$-continuous flat functors on $\mathcal{C}$ \cite[\S 5.3]{Vi07}, while the classifying topos of a theory $\thT$ is constructed as the category of sheaves $\textbf{Sh}(\calC_\thT,J_\thT)$, where $(\calC_\thT,J_\thT)$ is the syntactic site associated to $\thT$ \cite[D3.1]{Elephant}.}
\end{enumerate}

\medskip 
It is natural to ask if the framework of classifying toposes extends to our setting. Unfortunately, as is well-established, the Effective Topos $\Eff$ is not a Grothendieck topos \cite[p. 222]{HJP80}, and thus cannot be presented as a category of sheaves on a Grothendieck site. Consequently, Remark (2) does not apply:  $\Eff$ cannot be regarded as a classifying topos in the usual sense. This motivates the following test problem.

\begin{problem}\label{prob:Eff-site} With the framework of classifying toposes in mind, develop an analogue of a Grothendieck site applicable to subtoposes of $\Eff$. 
\end{problem}

\begin{discussion} Problem~\ref{prob:Eff-site} brings into focus the same tension between syntactic presentation and semantic invariance highlighted in Remark (1). 
\begin{itemize}
    \item For any elementary topos $\calE$, subtoposes are classified by LT topologies: each subtopos is equivalent to the category of $j$-sheaves  $\textbf{sh}_j(\calE)$
    for some LT topology $j$ \cite[A.4.4]{Elephant}. However, this is a fairly abstract characterisation. Taken by itself, it is not clear how to extract a first-order theory from this description, as in the case of Grothendieck toposes.
    \item On the other hand, realisability toposes (such as $\Eff$) admit syntactic presentations by {\em triposes} \cite{HJP80,Pi99}, with the associated toposes obtained via the {\em tripos-to-topos} construction. Concretely, a tripos is a particular kind of hyperdoctrine, and thus corresponds to a  first-order intuitionistic theory in the sense of Seely \cite{Seely83}. Nonetheless, while certainly powerful, the tripos perspective is fundamentally organised around logical presentations. Even setting aside the significant overhead in terms of syntactic details, it is also generally difficult to determine how distinct triposes compare at the level of their associated toposes.
\end{itemize}
Some interesting recent work has explored more algebraic presentations of triposes, e.g. via implicative and arrow algebras \cite{Miq20,vdBB23}. Our results suggest a different approach, independent of triposes. By Theorem~\ref{thm:GTK-full}, LT topologies in $\Eff$ admit a concrete presentation in terms of upper sequences, which can be leveraged to separate LT degrees (and thus the corresponding subtoposes).\footnote{Consider, for instance, Corollary~\ref{cor:strict-clt}.} Can we extend this description to obtain a useful site presentation? Related ideas appear in the work of Moerdijk-Palmgren \cite{MP97}, in particular their construction of a syntactic site using so-called {\em provable filter bases}. 
\end{discussion}

For readers familiar with S. Vickers' work, we mention two interesting points of contact.

\begin{discussion}[Arithmetic Universes] In a series of papers \cite{Vi17AU,Vi19}, Vickers extends the framework of classifying toposes to Arithmetic Universes (AUs), a very general class of categories that includes all elementary toposes with nno, such as $\Eff$. 
This is encouraging for Problem~\ref{prob:Eff-site}, but a familiar caveat remains. 
Strictly speaking, Vickers works with the presentations of AUs (using the framework of {\em sketches}) rather than with AUs themselves. In his words: ``It is as if we defined Grothendieck toposes to be the geometric theories they classify, with no attempt to identify equivalent presentations'' \cite[\S 1]{Vi19}.\footnote{Nonetheless, Vickers’ notion of AU-functors is interesting. This seems a good occasion to revisit previous work on geometric morphisms involving realisability toposes \cite{AwBau08,Joh13,FvO14}, particularly in light of the various non-existence and triviality results, and compare it with the AU perspective.} 
\end{discussion}

\begin{discussion}[Point-free perspective] Vickers advocates the view that a (Grothendieck) topos should be regarded as a ``generalised space'' whose points are models of the theory it classifies \cite{Vi07,VickersPtfreePtwise}. In the case of propositional geometric theories, global points correspond to completely prime filters, and subtoposes (``generalised subspaces'') may be seen as imposing conditions on which opens these points belong to. By contrast, the Kat\v{e}tov order \cite{Kat68} instead organises filters by their convergence patterns, a perspective especially suggestive in light of Theorem~\ref{thm:GTK-full}.
\end{discussion}

Finally, a very natural question, posed to us by M. Maietti:

\begin{problem} To what extent do our results on the LT topologies in $\Eff$ extend to other elementary toposes with nno?
\end{problem}

\subsubsection{Sheaf Cohomology} The notion of a Grothendieck topos originates in Grothendieck’s development of \'{e}tale cohomology, which eventually led to the resolution of the Weil Conjectures in algebraic geometry. One may wish to extend this framework to our setting:

\begin{problem} Develop an account of sheaf cohomology applicable to subtoposes of $\Eff$.
\end{problem}

\begin{discussion} A recurrent theme in mathematics is that locally consistent data may fail to assemble into a global structure, which can often be detected by an appropriate cohomology. While most familiar from geometry, analogous phenomena also arise in logic, even when not initially framed as a local/global issue. 

In a remarkable paper \cite{Bla83}, Blass shows how sheaf cohomology can detect failures of the Axiom of Choice, a perspective suggestive in light of the connection between Phoa's Theorem~\ref{thm:phoa} and semi-constructive principles (cf. Section~\ref{sec:generalised-Phoa}). Similar cohomological techniques were later applied by Talayco \cite{Tal95,Tal96} in his study of combinatorial objects such as Hausdorff gaps; might they also say something meaningful about the structure of upper sequences?
\end{discussion}

\subsubsection{Polynomial Functors and $\calW$-types}\label{sec:poly-functors} 

In Theorem~\ref{thm:fubini-GTK} we showed that the Gamified Kat\v{e}tov order can be obtained by closing the classical Kat\v{e}tov order under well-founded Fubini powers. Combinatorially this was captured by the notion of a $\delta$-Fubini power, but the core idea also admits a compact description in terms of \emph{polynomial} constructions. We give a brief overview.

\medskip 
\noindent
{\it Gamified Kat\v{e}tov Order:}
Let $\calU,\calV\subseteq\Pw$. The key observation is that $\calU\ast \calV$ is not only a {\em set of sums} (as originally defined), but is {\em itself} a sum of some kind (which we express as a polynomial). Explicitly, unpacking Definition~\ref{def:Kat-sum}, 
\[
\U\ast \V=\left\{\sum_{a\in A}B_a\ \middle|\ A\in\U\mbox{ and }(\forall a\in A)\ B_a\in\V_a\right\}
=\left\{\sum_{a\in A}f(a)\ \middle |\ A\in\calU\mbox{ and }f\colon A\to\calV\right\}\,\,,
\]
and so we may write:
\[
\calU\ast\calV\simeq\sum_{A\in\mathcal{U}}\calV^A=\left\{(A,f)\ \middle |\ A\in\calU\mbox{ and }f\colon A\to\calV\right\}\ .
\]
As already seen in Section \ref{sec:fubini}, $\calU\ast\calU=\sum_{A\in\calU}\calU^A$ is the set of all $\mathcal{U}$-branching trees of height $2$
and its upward closure is the Fubini product $\mathcal{U}\otimes\mathcal{U}$. This can also be described as iterations of the following ``polynomial'' function:
\[
P_\mathcal{U}(X)=\sum_{A\in\mathcal{U}}X^A=\{(A,f)\mid A\in\mathcal{U}\mbox{ and }f\colon A\to X\}.
\]
In particular, notice: 
\[P_{\mathcal{U}\ast\mathcal{U}}=P_\mathcal{U}\circ P_\mathcal{U}\ .\]
Proposition \ref{prop:n-concatenation-branching-tree} can therefore be rephrased as: the set of all $\mathcal{U}$-branching trees of height $n$ is exactly
\[\underbrace{P_\mathcal{U}(P_\mathcal{U}(\dots P_\mathcal{U}}_{n\text{ times}}(1))).\]
Informally, this suggests the Gamified Kat\v{e}tov order can also be obtained by taking the least fixed point of the operator $X\mapsto P_\calU(X)+Y$ (where $Y$ is a chosen type for possible outputs).

\medskip
\noindent
{\it Basic LT Topologies:} The same mechanism appears implicitly in Lee-van Oosten's construction of basic topologies \cite[\S 2]{LvO13}. Their construction proceeds in two main steps.
\begin{itemize}
    \item {\em Step one.} Given any subset family $\calU\subseteq\Pw$,  define the monotone endomorphism\footnote{By ``monotone endomorphism'', we mean a function $j\colon\Pw\to\Pw$ satisfying Condition (1) of Definition~\ref{def:LT-topology}.
    }
    \begin{align*}
	G_\calU\colon\Pw &\longrightarrow \Pw\\
	p&\longmapsto \bigcup_{A\in\calU}(A\to p) \quad.
\end{align*}
Interpreting $\bigcup$ as the union $\sum$, we may rewrite  $G_\calU(X)$  suggestively as $\sum_{A\in\calU}(A\to X)$, or in fact as a polynomial
$$\sum_{A\in\calU}X^A\ .$$

\smallskip

\item {\em Step two.} Apply the reflector map $L$ to obtain an LT topology
\[
j_\calU:=L(G_\calU),
\]
which is the $\clt$-least LT topology satisfying $G_\calU(p)\leq_{\mathrm{LT}}j_\calU(p)$ \cite[Prop.~2.1]{LvO13}. That is, basic topologies arise canonically as least fixed points associated with polynomial operators of the form $X\mapsto\sum_{A\in\calU}X^A$.
\end{itemize}

\medskip

Viewed this way, both constructions are governed by a common underlying principle: taking the least fixed points of polynomial operators. In categorical terms, this corresponds to taking the initial algebra of a {\it polynomial functor} \cite{NiSp25}, i.e.~a $\calW$-type. Moreover, for any polynomial functor $P$, the construction $$\lambda Y.\mu X.P(X)+Y$$
is precisely the construction of the {\it free monad} generated by the functor $P$. In light of Theorem~\ref{thm:GTK-full}'s characterisation of the $\clt$-order, a natural test problem would be:

\begin{problem}\label{prob:poly} Develop a purely category-theoretic account of the $\clt$-order on upper sequences using the language of polynomial functors, $\calW$-types, and free monads.
\end{problem}

\begin{discussion} This provides an alternative framework for making Theorem~\ref{mthm:main-thm} (i) precise. Instead of regarding ``well-founded iterations of Fubini powers'' as $\delta$-Fubini powers, we may present them as least fixed points of polynomial operations, a perspective already anticipated in, e.g. \cite{Aczel,AMM18}.    
\end{discussion}

\begin{discussion} 
The guiding idea behind Problem~\ref{prob:poly} is that {\bf oracles} may be identified  with {\bf polynomials}. Concretely, each polynomial $$\sum_{q\in Q} X^{F(q)}$$ 
represents a ``{\it computable algorithm using an oracle $F$ (exactly once)}'': each pair $(q,\psi)\in \sum_{q\in Q} X^{F(q)}$ represents a query $q$ to the oracle $F$ and an algorithm $\psi$ that converts the oracle's response $z\in F(q)$ into a solution $\psi(z)$ of type $X$.
As supporting evidence for this perspective, it was recently observed that morphisms between polynomial functors in an appropriate category correspond to (extended) Weihrauch reductions \cite{PrPr25,AhBa25}.
\end{discussion}

\begin{discussion} Fact \ref{fact:transfinite-recursion} in the proof of Theorem \ref{thm:filter-hyperarithmetic} can also be obtained as a computational rule for $\calW$-types. Here, $\Pi^1_1$ yields a partial combinatory algebra (see e.g.~\cite{BeeBook}), and the category of assemblies over a partial combinatory algebra has $\calW$-types \cite{vdBerg00}.
\end{discussion}

	\subsection{Perspectives from Model Theory}\label{sec:qn-MT} The starting point for this paper actually came from model theory. Recently, two embeddings of the Turing degrees into two {\em a priori} unrelated preorders have emerged:
	
	\begin{itemize}
		\item Hyland \cite{HylandEffective}: the Turing Degrees embed (effectively) into the $\clt$-order in $\Eff$.
		\item Malliaris-Shelah \cite{MSTuring}: the Turing Degrees embed (effectively) into Keisler’s order on simple unstable theories.\footnote{Not to be confused with the Rudin-Keisler order from before.} 
	\end{itemize}
The juxtaposition of these two results is both striking and puzzling: the $\clt$-order arises in category theory, whereas Keisler's order uses ultrafilters\footnote{The ultrafilters in the definition of Keisler's order are regular (i.e. they have a regularising family), and possibly on uncountable cardinals. For details on Keisler's order (including the importance of considering uncountable cardinals), see \cite[\S 7 -8]{MalliarisUltrafilter}.}  to compare the complexity of first-order theories.
We were thus led to ask if there exists a hidden combinatorial core organising these disparate manifestations of computability.
	This prompted a careful re-examination of Lee-van Oosten's results in \cite{LvO13}, which led us to notice the connection with a different preorder: the Rudin-Keisler order on ultrafilters over $\w$. The rest of this paper emerged as a development of that insight. 
	
\smallskip
	
	Returning to model theory,  are there productive connections between simple unstable theories and the structure of the LT degrees within $\Eff$?  At present this remains unclear; our understanding of both settings is still incomplete. Nonetheless, our results give some very interesting clues on where to start looking for answers. 
	
\subsubsection{Keisler's Order} The proofs in Malliaris-Shelah's work on Turing degrees \cite{MSTuring} were informed by an important result from an earlier paper \cite[Theorem 11.10]{MSK-not-simple}, which we roughly translate as: Keisler's order induces a well-defined preorder on ideals containing $\mathrm{Fin}$. Let us call this derived order the {\em Malliaris-Shelah order}, written as
$$\calH\leq_{\mathrm{MS}} \calI\qquad  \text{ideals} \  \, \calH,\calI \subseteq \Pw, \quad \mathrm{Fin}\subseteq\calH,\calI \ .$$
Informally,  Malliaris-Shelah's result highlights a certain coherence amongst regular ultrafilters (used to measure the complexity of theories), governed by combinatorics from the countable. In light of this paper's results, there are some natural questions regarding the nature of this coherence.


\begin{problem} How does $\leq_{\mathrm{MS}}$ compare with the Kat\v{e}tov order on ideals? With $\glt$? In particular, how does $\leq_{\mathrm{MS}}$ interact with Fubini powers? 
\end{problem}

\begin{discussion} More generally, regarding Keisler's order as a preorder on regular ultrafilters (see e.g. \cite[Definition 1.6]{MS17}), one may ask how taking Fubini powers of regular ultrafilters affects their Keisler complexity. 
\end{discussion}

Recalling Phoa's Theorem~\ref{thm:phoa} and the fact many important basic oracles do {\em not} correspond to ideals/filters on $\w$, one may also ask:

\begin{problem}\label{prob:MS-lowersets} Is there a meaningful extension of $\leq_{\mathrm{MS}}$ from ideals to lower sets in $\Pw$ (possibly containing $\Fin$ if necessary)? What are the implications for locating the Keisler-maximal simple theory, if such a theory exists?
\end{problem}

\subsubsection{Game Semantics} Since simple theories generalise stable theories, another natural approach is to look for meaningful extensions of results that characterise stability. A classical result by Shelah states: a formula $\varphi(x;y)$ of a first-order theory is {\em stable} iff  it has finite 2-rank \cite{Sh90}. Chase-Freitag subsequently observed that 2-rank could be reformulated using game semantics \cite{ChFr19}, which we now adapt to our context. 


\begin{definition}\label{def:online-game} Let $X$ be an arbitrary set, and let $\calH\subseteq\calP(X)$ be a subset family. The {\em Prediction Game} is an imperfect information game between three players: $\mer$ vs. $\art$ and $\nim$. A typical play has the following shape:
	\[
	\begin{array}{rccccccccccc}
	{\rm \mer}\colon	& x_0 &		& x_1	&		& x_2	&	& \dots & x_{k-1} & & x_{k} & \quad \;\;\dots\\
	{\rm \art} \colon	&		& y_0	&		& y_1	& 		& y_2	& \dots & & y_{k-1} & & y_k\;\;\dots\\
	{\rm \nim} \colon	&		& z_0	&		& z_1	& 		& z_2	& \dots & & z_{k-1} & & z_k\;\;\dots
	\end{array}
	\]

\noindent \textbf{Rules.} The (deceitful) wizard $\mer$ claims have fixed some $A\in\calH$, and invites $\art$ to guess which elements belong to it. At the $n$th round, $\mer$ plays $x_n\in X$. $\art$ predicts membership by choosing $y_n\in\{0,1\}$, where $y_n=1$ means “$x_n\in A$”. Merlin always declares Arthur’s prediction incorrect.

$\nim$ knows $\mer$ cannot be trusted, and thus checks $\calH$ if there indeed exists  $A\in\calH$ satisfying the negation of all of $\art$'s predictions (i.e. $\lnot y_0, \dots, \lnot y_n$). If so, she outputs $z_0=0$, and $\mer$ chooses a new challenge $x_{n+1}$; otherwise $z_1=1$, and the game terminates. 

\smallskip

\noindent \textbf{Winning Condition.} $\mer$ aims is to maximise the length of the mistake-run. The {\em Littlestone dimension} $\mathrm{Ldim}(\calH)$ is the largest $d$ for which Merlin can force $d$ consecutive mistakes, regardless of $\art$'s strategy. $\art$-$\nim$ win if and only if $\mathrm{Ldim}(\calH)<\infty$. 
\end{definition}

Readers familiar with Online Learning (in the sense of \cite[\S18]{BDSS14}) will recognise that our Prediction Game encodes the same mistake tree.\footnote{The difference is essentially presentational: here, we split the adversary into $\mer$, who issues challenges to the learner, and $\nim$, who enforces the constraints on the adversary.} In particular, when $\calH=\{\varphi(x;a)\mid a\in B\}$ is the family of $B$-definable sets determined by a formula $\varphi(x;y)$, the Littlestone dimension $\mathrm{Ldim}(\calH)$ coincides with the model-theoretic $2$-rank of $\varphi$ -- as originally observed by Chase-Freitag.

\medskip

This motivates the natural question:

\begin{problem}\label{prob:game} What is the relationship between the Prediction Game (Definition~\ref{def:online-game}) and our finite-query Kat\v{e}tov games (Definition~\ref{def:fin-quer-kat})? Do these games illuminate the relationship between stable and simple theories?
\end{problem}

\begin{discussion} There are many interesting angles from which to examine Problem~\ref{prob:game}.
For instance, if $\calH$ is a subset family with $\mathrm{LDim}(\calH)=n$, then $\art$'s mistake tree is necessarily of bounded height $n$. What changes if one merely requires the mistake tree to be well-founded, rather than uniformly bounded? 
\end{discussion}

    \bibliography{Eff}

\end{document}